\documentclass[a4paper]{amsart}
\usepackage{graphicx}
\usepackage{amssymb, latexsym, hyperref}
\usepackage{a4wide}
\usepackage{microtype}
\usepackage{psfrag}
\usepackage[all, rotate]{xy}
\usepackage{tikz}

\input{CyclicPictures}

\let\cal\mathcal
\def\AA{{\cal A}}
\def\BB{{\cal B}}
\def\CC{{\cal C}}
\def\DD{{\cal D}}
\def\EE{{\cal E}}
\def\FF{{\cal F}}
\def\GG{{\cal G}}

\def\LL{{\cal L}}

\def\PP{{\cal P}}
\def\QQ{{\cal Q}}
\def\RR{{\cal R}}

\def\TT{{\cal T}}
\def\UU{{\cal U}}
\def\VV{{\cal V}}
\def\WW{{\cal W}}
\def\XX{{\cal X}}
\def\YY{{\cal Y}}
\def\ZZ{{\cal Z}}

\let\blb\mathbb

\def\bF{{\blb F}} 
\def\bQ{{\blb Q}}

\def\bZ{{\blb Z}}
\def\bS{{\blb S}}

\def\bN{{\blb N}}
\def\bR{{\blb R}}

\def\bZ{{\blb Z}}

\let\frak\mathfrak

\def\aa{\frak{a}}
\def\bb{\frak{b}}

\def\Mod{\operatorname{Mod}}
\def\Modlfd{\operatorname{Mod^{lfd}}}

\def\mod{\operatorname{mod}}
\def\modc{\mod^{\rm{cfp}}}

\def\rad{\operatorname {rad}}

\def\Rep{\operatorname {Rep}}

\def\repc{\rep^{\rm{cfp}}}
\def\rep{\operatorname{rep}}

\def\Ext{\operatorname {Ext}}

\def\Hom{\operatorname {Hom}}

\def\End{\operatorname {End}}

\def\Fun{\operatorname {Fun}}

\def\im{\operatorname {im}}

\def\cone{\operatorname {cone}}
\def\coker{\operatorname {coker}}
\def\ker{\operatorname {ker}}
\def\Ker{\operatorname {ker}}

\def\End{\operatorname {End}}

\def\add{\operatorname {add}}

\def\exist{\exists}

\DeclareMathOperator{\Irr}{Irr}

\DeclareMathOperator{\colim}{colim}

\DeclareMathOperator{\Fac}{Fac}

\DeclareMathOperator{\ind}{ind}

\newcommand\Db{D^{b}}

\newcommand\DL{D_{\LL}}

\renewcommand\t{\tau}

\newtheorem{lemma}{Lemma}[section]
\newtheorem{proposition}[lemma]{Proposition}
\newtheorem{theorem}[lemma]{Theorem}
\newtheorem{corollary}[lemma]{Corollary}

\theoremstyle{definition}

\newtheorem{example}[lemma]{Example}
\newtheorem{definition}[lemma]{Definition}

\newtheorem{construction}[lemma]{Construction}

{

}

\theoremstyle{remark}

\newtheorem{remark}[lemma]{Remark}

\newtheorem{notation}[lemma]{Notation}

\newdimen\uboxsep \uboxsep=1ex
\def\uboxn#1{\vtop to 0pt{\hrule height 0pt depth 0pt\vskip\uboxsep
\hbox to 0pt{\hss #1\hss}\vss}}

\def\uboxs#1{\vbox to 0pt{\vss\hbox to 0pt{\hss #1\hss}
\vskip\uboxsep\hrule height 0pt depth 0pt}}

\def\Ob{\operatorname{Ob}}

\def\COrd{\mathfrak{C.Ord}}
\def\LOrd{\mathfrak{L.Ord}}

\def\arc{\operatorname{arc}}

\newcommand\exa{\nopagebreak \begin{center}\smallskip \nopagebreak               \begin{minipage}[t]{6cm}\sloppy}
\newcommand\exb{\end{minipage}\kern 1cm\begin{minipage}[t]{8cm}\sloppy}
\newcommand\exc{\end{minipage}\kern -3cm \smallskip\end{center}}

\renewcommand{\emptyset}{\varnothing}

\title{2-Calabi-Yau categories with a directed cluster-tilting subcategory}

\author{Jan \v{S}\v{t}ov\'{\i}\v{c}ek}
\address{Jan \v{S}\v{t}ov\'{\i}\v{c}ek\\ Dept. of Algebra \\ Charles University in Prague \\ Sokolovsk\'{a} 83 \\ 186 75 Praha 8 \\ Czech Republic}
\email{stovicek@karlin.mff.cuni.cz}

\author{Adam-Christiaan van Roosmalen}
\address{Adam-Christiaan van Roosmalen\\ Hasselt University \\ Agoralaan - gebouw D \\ B-3590 Diepenbeek}
\email{adamchristiaan.vanroosmalen@uhasselt.be}

\begin{document}
\date{\today}

\bibliographystyle{amsplain}

\begin{abstract}
As a generalization of acyclic 2-Calabi-Yau categories, we consider 2-Calabi-Yau categories with a directed cluster-tilting subcategory; we study their cluster-tilting subcategories and the cluster combinatorics that they encode.  We show that such categories have a cluster structure.

Triangulated 2-Calabi-Yau categories with a directed cluster-tilting subcategory are closely related to representations of certain semi-hereditary categories, more specifically to representations of thread quivers.  Thread quivers are a tool to classify and study certain semi-hereditary categories using both quivers and linearly ordered sets (threads).

We study the case where the thread quiver consists of a single thread (so that representations of this thread quiver correspond to representations of some linearly ordered set), and show that, similar to the case of a Dynkin quiver of type $A$, the cluster-tilting subcategories can be understood via triangulations of an associated cyclically ordered set.

In this way, we gain insight into the structure of the cluster-tilting subcategories of 2-Calabi-Yau categories with a directed cluster-tilting subcategory.  As an application, we show that every 2-Calabi-Yau category which admits a directed cluster-tilting subcategory with countably many isomorphism classes of indecomposable objects has a cluster-tilting subcategory $\VV$ with the following property: any rigid object in the cluster category can be reached from $\VV$ by finitely many mutations.  This implies that there is a cluster map which is defined on all rigid objects, and thus that there is a cluster algebra whose cluster variables are exactly given by the rigid indecomposable objects.
\end{abstract}


\maketitle

\tableofcontents

\section{Introduction}

Cluster algebras, introduced by Fomin and Zelevinsky in ~\cite{FominZelevinsky02} in an attempt to understand canonical bases and total positivity for algebraic groups~\cite{Lusztig90,Lusztig94}, provide a fascinating meeting point for diverse fields of algebra, geometry and combinatorics.  We refer to~\cite{Keller10} for several examples and many further references.  There is a  particularly fruitful connection between cluster algebras and representation theory, emphasized in~\cite{Keller10} and originating from~\cite{BuanMarshReinekeReitenTodorov06}: to triangulated or stably 2-Calabi-Yau categories with cluster-tilting subcategories.  Such 2-Calabi-Yau categories can be seen as (additive) categorifications of cluster algebras and will be the focal point of this paper.

\smallskip

\begin{sloppypar}
Originally, cluster algebras were defined in~\cite{FominZelevinsky02} as constructively defined subalgebras of $R(x_1,\dots,x_n)$, where $R$ is a suitable commutative ring, equipped with a distinguished set of generators (cluster variables) grouped into overlapping subsets (clusters) of the same finite cardinality (the rank of the algebra in question); several important homogeneous coordinate rings (of Grassmannians, Schubert varieties or double Bruhat cells) are known to or expected to carry a cluster structure (\cite{Zelevinsky07}). Recently, there has been growing interest in the case where we start with the ring of rational functions $R(x_i \mid i \in I)$ in infinitely many indeterminates; see~\cite{GrabowskiGratz14,Gratz15,Ndoune16}.
\end{sloppypar}

Such an infinite rank setting appeared on the representation-theoretic side even more naturally---in an attempt to understand algebraic triangulated categories generated by a 2-spherical object~\cite{HolmJorgensen12}, to find explicit Frobenius models for triangulated orbit categories~\cite{IgusaTodorov13,IgusaTodorov15,IgusaTodorov15b}, or to understand certain intriguing aspects of the infinite cluster combinatorics per se~\cite{BaurGratz16,HolmJorgensen13,LiuPaquette15,Yang16}.  One standing problem which we tackle in this paper is to establish a well-behaved decategorification procedure relating the representation-theoretic approach to cluster algebras of infinite rank, enjoying the same favorable properties as the case of finite rank algebras.  We find a satisfactory solution in terms of the existence of a cluster structure (see Theorem \ref{theorem:IntroductionClusterStructure} below for the exact statement).

\smallskip

Part of the combinatorial datum for finite rank cluster algebras can often be encoded into a finite quiver, and the best understood cluster algebras are those where this quiver can be chosen to have no oriented cycles. On the representation-theoretic side, the corresponding object of interest is a triangulated 2-Calabi-Yau category $\CC$ with a cluster-tilting object $T$ such that $\Lambda = \End_\CC(T)$ has an acyclic (Gabriel) quiver. It has been shown in \cite{KellerReiten08} that the acyclicity of the quiver of $\Lambda$ is equivalent to a strong homological condition---that $\Lambda$ is a hereditary algebra.

In the infinite rank case we focus on a similar homological condition: analogously to~\cite{HolmJorgensen12,LiuPaquette15,Yang16}, we begin with a triangulated 2-Calabi-Yau category and we require that it admits a cluster-tilting subcategory $\TT$ where $\mod \TT$ is hereditary (we do not require that $\TT = \add T$, for some cluster-tilting object $T$).  In such a case, the acyclicity of the underlying (Auslander-Reiten) quiver of $\TT$ is no longer sufficient to conclude that $\mod \TT$ is hereditary (see Example \ref{example:ClusterD}).  This means that a na\"{\i}ve generalization of the criteria in \cite{KellerReiten08} to this setting fails.  However, it is possible to salvage the result by replacing the condition on the quiver of $\TT$ by the condition that there should be no cycles of morphisms in $\TT$ itself (see \S\ref{subsection:IntroDirected}). We call such cluster-tilting subcategories directed.  Somewhat surprisingly, we show that it is sufficient to only exclude cycles of morphisms in $\TT$ of length three; here is the statement of our first main result (see Theorem~\ref{theorem:Trianglefree} and Corollary \ref{corollary:Trianglefree} in the text):

\begin{theorem}\label{theorem:TriangleFreeIntroduction}
Let $\CC$ be a Krull-Schmidt 2-Calabi-Yau triangulated category with a cluster-tilting subcategory $\TT$. If $\TT$ is directed (or equivalently, has no cycles of length three), then $\mod \TT$ is hereditary with Serre duality.
\end{theorem}

Cluster-tilting subcategories also encode a more combinatorial side of 2-Calabi-Yau categories. The key concept in this context is the Iyama-Yoshino mutation (\cite{IyamaYoshino08}): if $\TT$ is a cluster-tilting subcategory of a 2-Calabi-Yau Krull-Schmidt category $\CC$, then for every indecomposable $T \in \ind \TT$ there is exactly one cluster-tilting subcategory $\TT^*$ of $\CC$ such that $\TT \cap \TT^* = \add(\ind \TT \setminus \{T\})$.  That is, we can exchange the indecomposable object $T \in \ind\TT$ with a unique $T^* \in \ind\TT^*$. Put differently, a 2-Calabi-Yau category has a weak cluster structure in the sense of \cite{BuanIyamaReitenScott09} (see also \S\ref{section:ClusterStructures}).  If the quivers of the cluster-tilting subcategories do not have loops or 2-cycles, then the Iyama-Yoshino mutation induces the usual Fomin-Zelevinsky cluster mutation on the underlying Auslander-Reiten quiver; one can say that the cluster-tilting subcategories encode a cluster structure \cite{BuanIyamaReitenScott09}.  Following \cite{LiuPaquette15}, we call a Krull-Schmidt 2-Calabi-Yau category with such a cluster structure a cluster category.  As a motivating example, the orbit categories constructed in \cite{BuanMarshReinekeReitenTodorov06} have a cluster structure (see Example \ref{example:ClusterStructureFinite}).

Let $\CC$ now be an algebraic 2-Calabi-Yau category with a directed cluster-tilting subcategory.  In the finite rank case, i.e.\ when $\CC$ has a cluster-tilting object $T$ (it follows from Theorem \ref{theorem:TriangleFreeIntroduction} that $\End T$ is hereditary), the category $\CC$ has a cluster structure and there exists a so-called cluster map (also called the Caldero-Chapoton map, see~\cite{CalderoChapoton06}) connecting the combinatorics of the category to the combinatorial structure of the corresponding cluster algebra.  Specifically, this cluster map sends rigid indecomposable objects to cluster variables and cluster-tilting objects to clusters.

In the infinite rank case, i.e. when $\CC$ admits cluster-tilting categories but no cluster-tilting objects, the existence of an appropriate combinatorial structure, namely a cluster structure, on $\CC$, as well as the existence a cluster map, namely the Caldero-Chapoton map, was understood only in a few specific cases; see~\cite{IgusaTodorov15,JorgensenPalu13,LiuPaquette15,Yang16}.  In this paper, we extend these results, and establish the existence of a cluster structure (and thus also a cluster map, following ~\cite[Theorem 2.3]{JorgensenPalu13}) on all algebraic 2-Calabi-Yau categories with a directed cluster-tilting subcategory.   

\begin{theorem}\label{theorem:IntroductionClusterStructure}
Let $\CC$ be an algebraic Krull-Schmidt 2-Calabi-Yau triangulated category.  If $\CC$ admits a directed cluster-tilting subcategory $\TT$, then cluster-tilting subcategories form a cluster structure on $\CC$.
\end{theorem}

One aspect of (weak) cluster structures being studied is their exchange graphs: vertices of this graph are cluster-tilting subcategories, and two vertices are connected by an edge if and only if they differ in exactly one indecomposable object (hence, vertices are connected if and only if they differ up to an Iyama-Yoshino mutation at a single indecomposable).  The exchange graph is connected when $\CC$ is an algebraic acyclic cluster category (see \cite{BuanMarshReinekeReitenTodorov06, Hubery11}, {and} also \cite[Theorem 8.8]{BarotKussinLenzing10} for cluster categories associated to some canonical algebras).  However, if the cluster-tilting subcategories of $\CC$ have infinitely many isomorphism classes of indecomposable objects, then one cannot expect the exchange graph to be connected.  Indeed, a cluster-tilting subcategory $\TT$ of $\CC$ differs in infinitely many indecomposable objects from the cluster-tilting subcategory $\TT[1]$, and hence $\TT$ and $\TT[1]$ cannot be in the same connected component of the exchange graph. We will say that a cluster-tilting subcategory $\UU$ is \emph{reachable} from a cluster-tilting subcategory $\VV$ if $\UU$ and $\VV$ lie in the same connected component of the exchange graph.  A rigid object $X$ is \emph{reachable} from a cluster-tilting subcategory $\VV$ if $X$ lies in a cluster-tilting subcategory reachable from $\VV.$

As considered in~\cite{JorgensenPalu13}, the Caldero-Chapoton map acts on all rigid objects reachable from a chosen cluster-tilting subcategory $\VV$.  This means that the clusters and the cluster variables are given by the cluster-tilting subcategories and the rigid indecomposables reachable from $\VV$, respectively.  Hence, the Caldero-Chapoton map decategorifies a different part of $\CC$, depending on the initial cluster-tilting subcategory $\VV$.  In our setting, we prove the following theorem (Theorem \ref{theorem:ExistenceOfCCMap} in the text), which can be interpreted as a ``weak connectedness'' of the exchange graph and postulates the existence of a cluster map defined on all rigid objects, thus providing as good a connection between cluster categories and infinite rank cluster algebras as one could hope for.

\begin{theorem}\label{theorem:IntroductionCountableReachable}
Let $\CC$ be an indecomposable Krull-Schmidt 2-Calabi-Yau triangulated category with a directed cluster-tilting subcategory $\TT$.  If $\ind \TT$ is countable, then there exists a cluster-tilting subcategory $\VV$ of $\CC$ such that:
\begin{enumerate}
\item every rigid object is reachable from $\VV$, and
\item there exists a cluster map $\rho_\VV: \Ob \EE \to \bQ(x_v)_{v \in \ind \VV}$ where $\EE$ is the full subcategory of $\CC$ given by the rigid objects. 
\end{enumerate}
\end{theorem}

We note that the condition that $\ind \TT$ is countable is necessary for the result that every rigid object is reachable from $\VV$, see Remark \ref{remark:WhyCountable}.

\smallskip

Let us now give a more detailed overview of the other results and techniques used in the paper. Throughout, let $\CC$ be an essentially small algebraic Krull-Schmidt 2-Calabi-Yau triangulated category, and let $\TT$ be a directed cluster-tilting subcategory of $\CC$ (see Theorem \ref{theorem:AllThreadQuiversOccur} below for a rich source of examples).

\subsection{Directed cluster-tilting subcategories}
\label{subsection:IntroDirected}

The focus of this article is on Krull-Schmidt 2-Calabi-Yau categories with a directed cluster-tilting subcategory.  Recall that a path in a Krull-Schmidt category is a sequence $X_0, X_1, \ldots, X_n$ of indecomposable objects such that for every $i \in \{0, 1, \ldots, n-1\}$ there is a nonzero noninvertible morphism $X_i \to X_{i+1}$, and that a Krull-Schmidt category is called directed if there is no path from an indecomposable object to itself (see Definition \ref{definition:Directed}).  We look at directed cluster-tilting subcategories as a generalization of acyclic cluster-tilting objects.  Indeed, if $\TT = \add(T)$ for a $T \in \CC$, then $\TT$ is directed if and only if the quiver of $T$ is acyclic. 

Theorem \ref{theorem:TriangleFreeIntroduction} is a form of the criterion found in \cite{KellerReiten08}, but uses directedness instead of acyclicity of the (Gabriel) quiver of a cluster-tilting object.  We note that we do not require the full force of directedness in the proof of Theorem \ref{theorem:TriangleFreeIntroduction}; we only need to exclude paths of the form $X_0,X_1,X_2,X_0$.  That all paths are then excluded, is a consequence of this theorem.

We then look at which directed categories $\TT$ can actually occur as a cluster-tilting subcategory.  The question is fully answered by the following result (Theorem \ref{theorem:AllThreadQuiversOccur} in the text):

\begin{theorem}\label{theorem:IntroductionAllThreadQuiversOccur}
Let $\TT$ be a category such that $\mod \TT$ is hereditary and satisfies Serre duality.  The category $\Db(\mod \TT) / (\bS \circ [-2])$ is a Krull-Schmidt 2-Calabi-Yau category which admits a cluster-tilting subcategory equivalent to $\TT$.
\end{theorem}

The categories $\TT$ such that $\mod \TT$ is hereditary {with Serre duality} have been classified in \cite{BergVanRoosmalen14} using thread quivers (see Corollary \ref{corollary:Trianglefree}, the formulation of Theorem \ref{theorem:AllThreadQuiversOccur} uses this classification).  Thread quivers are combinatorial objects that describe such a category $\TT$ as an amalgam of a (possibly infinite) locally finite quiver without oriented paths and certain linearly ordered sets. This description will allow us to prove certain properties for $\TT$ by proving them separately for quivers and linearly ordered sets.

There is, however, a subtle but important difference from the case where a cluster-tilting object exists.  If an algebraic triangulated 2-Calabi-Yau category $\CC$ admits a directed cluster-tilting object, it is shown in \cite{KellerReiten08} that $\CC$ is triangle-equivalent to $(\Db \mod \Lambda) / (\bS \circ [-2])$.  One interpretation of \cite{KellerReiten08} is that this orbit category is the ``unique'' 2-Calabi-Yau categorification of the corresponding acyclic cluster algebra.  We do not know whether such a uniqueness result holds if $\CC$ only admits a directed cluster-tilting subcategory. We work around the problem in~\S\ref{section:Comparing} using Iyama-Yoshino reductions.

\subsection{Iyama-Yoshino reductions}

In \S\ref{section:Reduction}, we recall some results about the Iyama-Yoshino reduction (see \cite{IyamaYoshino08}, this reduction is also called the Calabi-Yau reduction).  For this introduction, let $\ZZ \subseteq \CC$ be rigid and functorially finite in $\CC$.  We consider the $\Ext^1$-perpendicular subcategory $\ZZ^{\perp_1} = \{C \in \CC \mid \Ext^1(\ZZ,C) = 0\}.$  It is shown in \cite{IyamaYoshino08} that the category $\CC_{[\ZZ]} = \ZZ^{\perp_1} / [\ZZ]$ is a triangulated 2-Calabi-Yau category, and in \cite{AmiotOppermann12} that $\CC_{[\ZZ]}$ is algebraic if $\CC$ is algebraic.

The following theorem (Theorem \ref{theorem:HereditaryBongartzComplement} in the text) shows that Iyama-Yoshino reductions do not take us out of our overarching framework, namely that of categories with a directed cluster-tilting subcategory.

\begin{theorem}\label{theorem:HereditaryBongartzIntroduction}
Let $\CC$ be a 2-Calabi-Yau Krull-Schmidt category.  Let $\ZZ$ be a functorially finite rigid subcategory of $\CC$.  {If $\CC$ has a directed cluster-tilting subcategory, then the reduction $\CC_{[\ZZ]}$ also has} a {directed} cluster-tilting subcategory.
\end{theorem}

The proof {of the} theorem is based on a version of the Bongartz completion for cluster-tilting subcategories (Theorem \ref{theorem:BongartzComplement} in the text, following \cite{Jasso13} where similar statements have been shown for cluster-tilting objects instead of the more general cluster-tilting subcategories):

\begin{theorem}\label{theorem:IntroductionBongartzComplement}
Let $\CC$ be a Krull-Schmidt 2-Calabi-Yau triangulated category, and let $\ZZ$ be a functorially finite rigid subcategory of $\CC$.  If $\CC$ has a cluster-tilting subcategory, then $\CC$ has a cluster-tilting subcategory $\RR$ which contains $\ZZ$.
\end{theorem}

Since Theorem \ref{theorem:IntroductionBongartzComplement} does not rely on the category admitting a directed cluster-tilting subcategory, it might be of independent interest.

Intuitively, our strategy to study $\CC$ is to ``approximate'' it by Iyama-Yoshino reductions $\CC_{[\ZZ]}$ which are small enough to have an acyclic cluster-tilting object. The exact implementation of this idea becomes somewhat technical due to the fact that $\CC_{[\ZZ]}$ is neither a subcategory nor a quotient category, but a combination of both.

In order to use $\CC_{[\ZZ]}$ to study properties of objects in $\CC$ we will use the following definition (see Definition \ref{definition:LivesIn}): an object $C \in \CC$ \emph{lives in} $\CC_{[\ZZ]}$ if $C \in \ZZ^{\perp_1}$ and $[\ZZ](C,C) = 0$.  Specifically, we have that $\Hom_\CC(C,C) \cong \Hom_{\CC_{[\ZZ]}}(C,C)$.

The following theorem (Theorem \ref{theorem:Covering} in the text) allows us to reduce questions about $\CC$ to better understood 2-Calabi-Yau categories.

\begin{theorem}\label{theorem:CoveringIntroduction}
Let $\CC$ be a Krull-Schmidt 2-Calabi-Yau category with a {directed} cluster-tilting subcategory $\TT$.  For any object $X \in \CC$, there is a functorially finite and rigid full subcategory $\ZZ \subseteq \CC$ such that
\begin{enumerate}
\item $X$ lives in $\CC_{[\ZZ]}$, and
\item there is an object $T \in \TT$ which lives in $\CC_{[\ZZ]}$ and is a cluster-tilting object in $\CC_{[\ZZ]}$.
\end{enumerate}
If $\CC$ is algebraic, then $\CC_{[\ZZ]} \cong \CC_{\End T}$.
\end{theorem}

Here, $\CC_{\End T} = (\Db \mod \End T)/ \bS \circ [-2]$.  Note that $\End T$ is a hereditary algebra (since $\mod \TT$ is hereditary and $T \in \TT$, see Proposition \ref{proposition:ShIsLocal} in the text).  We remark that $\ZZ$ is (in general) not contained in $\TT$.

\subsection{Cluster structures}

In \S\ref{section:ClusterStructures}, we recall the notion of a cluster structure from \cite{BuanIyamaReitenScott09} (see Definition \ref{definition:ClusterStructure} in the text) and prove that the cluster-tilting subcategories give a cluster structure on $\CC$ (see Theorem \ref{theorem:IntroductionClusterStructure}).  The proof uses the fact that $\CC_{\rep Q}$ has a cluster structure for any acyclic finite quiver $Q$ and uses Theorem \ref{theorem:HereditaryBongartzIntroduction} to reduce to this case.  In more detail, it is shown in \cite[Theorem II.2.1]{BuanIyamaReitenScott09} that $\CC$ has a cluster structure if and only if the Auslander-Reiten quiver of the cluster-tilting objects do not have loops or 2-cycles.  We show that if a cluster-tilting subcategory $\UU$ of $\CC$ has a loop or a 2-cycle, then there is a functorially finite rigid subcategory $\WW$ of $\CC$ such that $\CC_{[\WW]}$ also admits a cluster-tilting category with a loop or a 2-cycle.  However, one can choose $\WW$ to be large enough so that $\CC_{[\WW]}$ admits a cluster-tilting object.  It then follows from Theorem \ref{theorem:IntroductionBongartzComplement} and the criteria in \cite{KellerReiten08} that $\CC_{[\WW]} \cong \CC_Q$, for a finite acyclic quiver $Q$.  Since no cluster-tilting object in $\CC_Q$ has a loop or a 2-cycle (Example \ref{example:ClusterStructureFinite}), we are done.

Having a cluster structure allows one to look at cluster maps, i.e. decategorifications from the cluster category to a cluster algebra.  The Caldero-Chapoton cluster map (see \cite{CalderoChapoton06}) has been generalised to categories admitting cluster-tilting subcategories in \cite{JorgensenPalu13}.  However, the cluster map in \cite{JorgensenPalu13} is only defined on rigid objects which are reachable from the initial cluster-tilting subcategory.  In our setting, it is not realistic to expect all cluster-tilting subcategories to be reachable from any given one (thus, for the exchange graph to be connected).

Instead, we will consider a weaker question: is there a cluster-tilting subcategory $\UU$ of $\CC$ such that every rigid object is reachable from $\UU$?  One obvious condition (in general) is that for every rigid object $X \in \CC$, one should have $\Ext^1(U,X) = 0$ for all but finitely many $U \in \ind \UU$.  We show that (in our setting), the previous statement is also sufficient.  We have the following statement (Corollary \ref{corollary:Bongartz}) in the text:

\begin{theorem}\label{theorem:IntroductionReachability}
Let $\CC$ be an algebraic 2-Calabi-Yau category with a cluster-tilting subcategory $\UU \subseteq \CC$.  Assume that $\CC$ admits a directed cluster-tilting subcategory.
\begin{enumerate}
\item Let $X \in \CC$ be a rigid object such that $\Ext^1(U,X) = 0$ for all but finitely many $U \in \ind \UU$.  There is a cluster-tilting subcategory $\VV$ of $\CC$, reachable from $\UU$, with $X \in \VV$.
\item A cluster-tilting subcategory $\VV$ is reachable from $\UU$ if and only if $\UU$ and $\VV$ differ at only finitely many indecomposable objects.  In this case $|(\ind \UU) \setminus (\ind \VV)| = |(\ind \VV) \setminus (\ind \UU)|$.
\end{enumerate}
\end{theorem}

Thus, in order to prove Theorem \ref{theorem:IntroductionCountableReachable}, we need to show that there exists a cluster-tilting subcategory $\UU$ of $\CC$ satisfying the following property: for every rigid $X \in \CC$ we have $\Ext^1(U,X) = 0$ for all but finitely many $U \in \ind \UU$.  This is equivalent to $\UU$ being \emph{locally bounded}, i.e. for each $U \in \UU$, there are, up to isomorphism, only finitely many indecomposable objects $X \in \UU$ such that $\Hom(U,X) \not= 0$ or $\Hom(X,U) \not= 0$.  To find such a locally bounded cluster-tilting subcategory, we will use the classification of the directed cluster-tilting subcategories by thread quivers and handle two cases separately: one where the cluster-tilting subcategory is a linearly ordered set and one where the cluster-tilting subcategory is an (infinite) quiver.

\subsection{\texorpdfstring{Cluster-categories of type $A$ via triangulations}{Cluster-categories of type A via triangulations}}

In \S\ref{section:Triangulations} and \S\ref{section:CominatoricsOfTypeA}, we discuss the case where the cluster-tilting subcategory is given by a thread, or, equivalently, where the cluster-tilting subcategory is of the form $k\LL$ where $\LL$ is a bounded linearly ordered locally discrete set (here, \emph{bounded} means that $\LL$ has both a minimum and a maximum, and \emph{locally discrete} means that $\LL$ has no accumulation points).  The main results are those where we link the cluster-tilting subcategories with triangulations of cyclically ordered sets, building upon the finite case (\cite{CalderoChapotonSchiffler06}), and similar results for the infinite cases (\cite{HolmJorgensen12,LiuPaquette15}, see also \cite{GrabowskiGratz14, Gratz15} and \cite{IgusaTodorov15}).  We do note that our definition of a triangulation differs from the ones in these references (see Definition \ref{definition:MainTriangulationProperties}): a triangulation in our sense need not be a maximal set of pairwise noncrossing arcs, and conversely, a maximal set of pairwise noncrossing arcs need not be a triangulation (see Example \ref{example:NoncrossingSets}).

Recall (see \S\ref{subsection:CyclicallyOrderedDefinitions}) that a cyclic order $R$ on a set $C$ is a ternary relation.  We think of the cyclically ordered set $(C,R)$ as a polygon where the vertices are given by $C$.  The statement $(a,b,c) \in R$ then becomes: in walking over the polygon $(C,R)$ in a chosen direction (clockwise or counterclockwise), $b$ lies between $a$ and $c$.  An arc in $(C,R)$ is a diagonal in the polygon, i.e. a pair of distinct points which are not adjacent.  Two arcs are noncrossing if they do not cross in the usual sense.  

Let $\LL$ be a bounded linearly ordered locally discrete set.  We associate to $(\LL, \leq)$ a cyclically ordered set $(\LL, \leq)_{\rm cyc} = (C,R)$ in the usual way (namely $(a,b,c) \in R \Leftrightarrow (a < b < c) \lor (b < c < a) \lor (c < a < b)$, see Example \ref{example:CyclicFromLinear}).

In Proposition \ref{proposition:ObjectsArcs} we show, for an infinite $\LL$, that there is a bijection $\Psi: \arc (\LL, \leq)_{\rm cyc} \to \ind \CC_\LL$ between the set of arcs on $(\LL, \leq)_{\rm cyc}$ and the set of (isomorphism classes of) indecomposable objects in $\CC_\LL = (\Db \mod k\LL) / (\bS \circ [-2]).$  Moreover, arcs $ab$ and $cd$ in $(\LL, \leq)_{\rm cyc}$ cross if and only if $\Ext^1(\Psi(ab), \Psi(cd)) \not= 0$.  In this way, we see that there is a map $\Phi$ from sets of pairwise noncrossing arcs on $(\LL, \leq)_{\rm cyc}$ to rigid additive subcategories of $\CC_\LL$.

In \S\ref{section:Triangulations}, we study sets $S$ of pairwise noncrossing arcs on $(\LL, \leq)_{\rm cyc}$.  We concentrate on the properties of $S$ being maximal, connected, locally finite, and being a triangulation (see Definition \ref{definition:MainTriangulationProperties} for definitions).  These four properties are not independent, and our main results in this section investigate their relations.  These results are summarized in the table in Examples \ref{example:NoncrossingSets} where all possibilities are listed.

In \S\ref{section:CominatoricsOfTypeA}, we study the rigid additive subcategories corresponding to these sets of pairwise noncrossing arcs under the function $\Phi$.  We link the properties of sets of pairwise noncrossing arcs in Definition \ref{definition:MainTriangulationProperties} to the properties of rigid additive subcategories in Proposition \ref{proposition:ClusterTiltingSomeProperties}.  Our main result is where we describe which sets of arcs correspond to cluster-tilting subcategories (see Theorem \ref{theorem:WhenClusterTilting}):

\begin{theorem}\label{theorem:WhenClusterTiltingIntroduction}
Let $\LL$ be an unbounded linearly ordered locally discrete set.  A set $S$ of pairwise noncrossing arcs on $(\LL, \leq)_{\rm cyc}$ is a connected triangulation if and only if $\Phi(S)$ is a cluster-tilting subcategory.
\end{theorem}

As a corollary (see Corollary \ref{corollary:WhenClusterTilting}), we find a description of the cluster-tilting subcategories in $\CC_\LL$ that does not refer to functorial finiteness: namely, a rigid subcategory is a cluster-tilting subcategory if and only if it is indecomposable (as an additive category, see Appendix \ref{section:IndecomposableCTS}), maximal rigid, and can be mutated at every indecomposable object (in the sense of Definition \ref{definition:AllowMutation}).

\subsection{Locally bounded cluster-tilting objects} We are now set to prove {Theorem~\ref{theorem:IntroductionCountableReachable}}.  Based on Theorem \ref{theorem:IntroductionReachability}, we only need to show that there is a locally bounded cluster-tilting subcategory $\UU$.  To construct $\UU$, we will use the description of the directed cluster-tilting subcategory $\TT$ of $\CC$ by thread quivers.  Recall from \cite{BergVanRoosmalen14} that a thread quiver consists of the following information:
\begin{itemize}
\item a quiver $Q_u=(Q_0,Q_1)$ where $Q_0$ is the set of vertices and $Q_1$ is the set of arrows (called the \emph{underlying quiver}),
\item a decomposition $Q_1 = Q_s \coprod Q_t$ of the set of arrows into \emph{standard arrows} (contained in $Q_s$) and thread arrows (contained in $Q_t$),
\item a linearly ordered set $\PP_t$ (possibly empty) for every thread arrow $t \in Q_t$.
\end{itemize}
With a thread quiver $Q$, we may associate a Krull-Schmidt category $kQ$ such that $\mod kQ$ is hereditary and $\Db \mod kQ$ has a Serre functor.  The category $kQ$ is constructed as a 2-pushout, and we obtain a fully faithful functor $kQ_u \to kQ$.  Using this embedding as an identification, $kQ_u$ is a functorially finite subcategory of $kQ$.  Moreover $kQ / [kQ_u] \simeq \oplus_{t \in Q_t} k\LL_t$, where $k\LL_t$ is a bounded locally discrete linearly ordered set obtained from $\PP_t$.

Let $Q$ be a thread quiver such that $\TT \simeq kQ$.  The subcategory $kQ_u \subseteq kQ \subset \CC$ is functorially finite, and we may consider the Iyama-Yoshino reduction $\CC_{[kQ_u]}$ which has a cluster-tilting subcategory equivalent to $kQ / [kQ_u] \simeq \oplus_{t \in Q_t} k\LL_t$.  A priori, we do not know whether this implies that $\CC_{[kQ_u]} \cong \oplus_{t \in Q_t} \CC_{\LL_t}$.

This is where we will use the results in \S\ref{section:Comparing}.  Theorem \ref{theorem:EquivalentClusterStructures} states that there is a function $F: \Ob \CC_{[kQ_u]} \to \Ob \oplus_{t \in Q_t} \CC_{\LL_t}$ such that $F$ induces an algebra-isomorphism on endomorphism rings and $F$ maps cluster-tilting subcategories to cluster-tilting subcategories (compatible with the mutation).  The cluster-tilting subcategories of $\oplus_{t \in Q_t} \CC_{\LL_t}$ are understood via triangulations (see \S\ref{section:Triangulations} and \S\ref{section:CominatoricsOfTypeA}); in particular, we know that $\oplus_{t \in Q_t} \CC_{\LL_t}$ has a locally bounded cluster-tilting subcategory (Proposition \ref{corollary:LocallyBoundedClusterTilting}) and hence, so does $\CC_{[kQ_u]}$.  The intricacies of the proof of {Theorem~\ref{theorem:IntroductionCountableReachable}} then relate to showing that the corresponding cluster-tilting subcategory in $\CC$ is still locally bounded.

\subsection*{Acknowledgments}  This work was partially funded by the Eduard \v{C}ech Institute under grant GA \v{C}R P201/12/G028.  The second author is currently a postdoctoral researcher at FWO (12.M33.16N).

\section{Preliminaries and basic results}

Let $k$ be an algebraically closed field. We will assume all categories and functors to be $k$-linear. Furthermore, we assume all categories to be essentially small (with the obvious exception of module categories of the form $\Mod \frak{a}$).

\subsection{Krull-Schmidt categories}

An additive category is called \emph{Krull-Schmidt} (or Krull-Remark-Schmidt) if idempotents split and all endomorphism rings are semi-perfect, or equivalently (see \cite[Theorem 27.6]{AndersonFuller92}), every object is a finite direct sum of objects with local endomorphism rings.  This decomposition is then unique up to permutation of the direct summands and isomorphism.  Since every finite-dimensional algebra is semi-perfect, a Hom-finite additive category is Krull-Schmidt if and only if idempotents split (see for example \cite[Corollary A.2]{ChenYeZhang08}).

If $\AA$ is an abelian Ext-finite category, then $\AA$ and $\Db \AA$ are Krull-Schmidt categories, as both $\AA$ and $\Db \AA$ are Hom-finite and idempotents split (idempotents split in $\AA$ because $\AA$ is abelian; that they split in $\Db \AA$ has been shown in \cite[2.10. Corollary]{BalmerSchlichting01}).

Let $\CC$ be an additive Krull-Schmidt category. By $\ind \CC$ we will denote the full subcategory spanned by a set of chosen representatives of isomorphism classes of indecomposable objects of $\CC$. Note that $\CC$ is fully determined by $\ind \CC$ in that we can recover $\CC$ back (up to $k$-linear equivalence) as the category of finitely generated projective $\ind \CC$-modules.  If $\CC'$ is a Krull-Schmidt subcategory of $\CC$, then we will assume that $\ind \CC' \subseteq \ind \CC$.

\subsection{Linearly ordered sets}
\label{subsection:Lin ord sets}

Let $(\PP, \leq)$ be a poset.  We say that $\PP$ is \emph{bounded below} if $\PP$ has a minimum and that $\PP$ is \emph{bounded above} if $\PP$ has a maximum.  If $\PP$ is both bounded above and bounded below, we say that $\PP$ is \emph{bounded}.  If $\PP$ is neither bounded below nor bounded above, then we say that $\PP$ is \emph{unbounded}.

Let $(\PP_1, \leq_1)$ and $(\PP_2, \leq_2)$ be posets.  We will write $(\PP_1 \cdot \PP_2,\leq)$ for the poset with underlying set $\PP_1 \coprod \PP_2$ and with a partial ordering $\leq$ given by
\begin{eqnarray*}
a \leq b \Leftrightarrow \left\{
\begin{array}{l}
\mbox{$a,b \in \PP_1$ and $a \leq_1 b$,} \\
\mbox{$a,b \in \PP_2$ and $a \leq_2 b$,} \\
\mbox{$a \in \PP_1$ and $b \in \PP_2$}
\end{array} \right.
\end{eqnarray*}

If $(\PP_1,\leq_1)$ and $(\PP_2,\leq_2)$ are posets, then we will write $(\PP_1 \stackrel{\rightarrow}{\times} \PP_2,\leq)$ for the partially ordered set with underlying set $\PP_1 \times \PP_2$ endowed with the lexicographical ordering.

A poset which is totally ordered is called an \emph{ordered set} or a \emph{linearly ordered set}.  In this paper, we prefer the term linearly ordered to distinguish it from the cyclically ordered set from \S\ref{section:Triangulations}.

A linearly ordered set is called locally discrete if every nonmaximal element $x$ has a unique successor $x+1$ and every nonmimimal element $x$ has a unique predecessor $x-1$.  An unbounded locally discrete linearly ordered set is isomorphic to $\PP \stackrel{\rightarrow}{\times} \bZ$ for a linearly ordered set $\PP$.  A bounded locally discrete ordered set is isomorphic to $\bN \cdot (\PP \stackrel{\rightarrow}{\times} \bZ) \cdot (-\bN)$.  Here, we allow $\PP$ to be the empty set.

We will often see a partially ordered set $\PP$ as a category in the usual way.  If $\PP$ is a poset, we write $k\PP$ for the additive closure of $k$-linearized poset category and we choose $\ind k\PP$ to be $\PP \subseteq \Ob k\PP$, thus the objects are formal direct sums of elements of $\PP$, the morphisms are given by (for $a,b \in \ind k\PP = \PP$)
\[\Hom_{k\LL} (a,b) = \begin{cases} k & a \leq b, \\ 0 & a > b, \end{cases}\]
and the composition is given by multiplication.

\subsection{Serre functors and Calabi-Yau categories}
\label{subsection:Serre and CY}

Let $\CC$ be a Hom-finite category.  A \emph{Serre functor} \cite{BondalKapranov89} on $\CC$ is an additive auto-equivalence $\bS\colon \CC \to \CC$ such that for every $X,Y \in \Ob \CC$ there are isomorphisms
$$\Hom(X,Y) \cong \Hom(Y,\bS X)^*$$
natural in $X$ and $Y$, and where $(-)^*$ is the vector-space dual.  If $\CC$ is triangulated, then $\bS$ is a triangulated functor (see \cite[Proposition 3.3]{BondalKapranov89}).

We say that $\CC$ has \emph{Serre duality} if $\CC$ admits a Serre functor.  If there is an equivalence of triangle functors $\bS \cong [n]$ for some $n \geq 0$, then we say that the category $\CC$ is \emph{$n$-Calabi-Yau}.

It has been shown in \cite[Theorem A]{ReVdB02} that $\Db \AA$ has a Serre functor if and only if $\Db \AA$ has Auslander-Reiten triangles.  Writing $\tilde{\tau}$ for $\bS[-1]$, every indecomposable object $A \in \Db \AA$ admits an Auslander-Reiten triangle $\tilde{\tau} A \to M \to A \to \bS A$.  Note that $\tilde{\tau} \cong [1]$ when $\CC$ is a 2-Calabi-Yau category.

An abelian category $\AA$ is said to satisfy Serre duality if the bounded derived category $\Db \AA$ has Serre duality.  If $\AA$ is hereditary and has Serre duality, then it follows from \cite[Theorem A]{ReVdB02} that $\AA$ has Auslander-Reiten sequences.  We will denote the Auslander-Reiten translation in $\AA$ by $\tau$, thus for every nonprojective indecomposable object $A \in \AA$ there is an Auslander-Reiten sequence $0 \to \tau A \to M \to A \to 0$.

The Auslander-Reiten translation in $\AA$ and the functor $\tilde{\tau}\colon \Db \AA \to \Db \AA$ coincide on nonprojective indecomposable objects of $\AA$, meaning that for every indecomposable nonprojective $A \in \AA$, we have $(\tau A)[0] \cong \tilde{\tau} (A[0])$.  Since most abelian categories we will work with are hereditary, we can write $\tau$ for $\tilde{\tau}$ without confusion. Serre duality then takes the following form, known as the \emph{Auslander-Reiten formula}:

\begin{lemma} \label{lemma:AR-formula}
Let $\AA$ be a hereditary abelian category with Serre duality. If $X,Y \in \ind \AA$, then $\Ext^1_\AA(X,Y) \cong \Hom_\AA(Y,\tau X)^*$.
\end{lemma}

\begin{proof}
If $X$ is nonprojective indecomposable, then
\[ \Ext^1_\AA(X,Y) \cong \Hom_{\Db \AA}(X[-1],Y) \cong \Hom_{\Db \AA}(Y,\bS X[-1])^* = \Hom_A(Y,\tau X)^*. \]
If on the other hand $X$ is projective, then both sides vanish since $\tau X = 0$.
\end{proof}

\subsection{Representations of preadditive categories}
\label{subsection:Reps of categories}

Let $\aa$ be a small preadditive category.  A right $\aa$-module is a contravariant functor from $\aa$ to $\Mod k$, the category of all vector spaces.  The category of all right $\aa$-modules is denoted by $\Mod \aa$.

Let $f\colon \frak{a} \to \frak{b}$ be a functor between small preadditive categories.  There is an obvious restriction functor
\[(-)_{\frak{a}}\colon \Mod \frak{b} \to \Mod \frak{a}\]
which sends $N$ to $N \circ f$.  This restriction functor has a left adjoint 
$$- \otimes_{\frak{a}} \frak{b}\colon \Mod \frak{a} \to \Mod \frak{b}$$
which is the right exact functor which sends the projective generators
$\frak{a}(-,A)$ in $\Mod \frak{a}$ to $\frak{b}(-,f(A))$ in $\Mod \frak{b}$.  As usual, if $f$ is fully faithful
we have $(N \otimes_{\frak{a}} \frak{b})_\frak{a}=N$.

Let $M$ be in $\Mod \frak{a}$.  We will say that $M$ is \emph{finitely generated} if $M$ is a quotient object of a finitely generated projective.  We say that $M$ is \emph{finitely presented} if $M$ has a presentation
\[P\to Q\to M\to 0\]
where $P,Q$ are finitely generated projectives.  We write $\mod \aa$ for the full subcatgeory of $\Mod \aa$ spanned by the finitely presented modules.  If $\mod \aa$ is an abelian category, we say that $\aa$ is \emph{coherent}.

Dually we will say that $M$ is \emph{finitely cogenerated} if it is contained in a finitely cogenerated injective (that is, in an injective envelope of a finitely generated semisimple module, \cite[Proposition 18.18]{AndersonFuller92}).  \emph{Finitely copresented} is defined in a similar way.  If $\aa$ is Hom-finite, it follows from the (equivalent) definition of finitely generated and cogenerated modules in~\cite[\S10]{AndersonFuller92} that the vector-space duality $(-)^*\colon \Mod \aa \to \Mod \aa^\circ$ sends finitely generated modules to finitely cogenerated ones, finitely presented modules to finitely copresented ones, and vice versa.

We will write $\modc \aa$ for the full subcategory of $\Mod \aa$ spanned by the objects which are both finitely presented and finitely copresented.

With every object $A$ of $\aa$, we may associate a \emph{standard projective} $\aa(-,A)$ and a \emph{standard injective} $\aa(A,-)^*$.  It is clear that every finitely generated projective is a direct summand of a finite sum of standard projectives.  If $\aa$ has finite direct sums and idempotents split in $\aa$, then every finitely generated projective is isomorphic to a standard projective. Dual notions hold for injective objects.

Assume that idempotents split in $\aa$.  We call $\aa$ \emph{locally bounded} provided it is Hom-finite and for each $A \in \aa$, there are up to isomorphism only finitely many indecomposable objects $X \in \aa$ such that $\aa(A,X) \ne 0$ or $\aa(X,A) \ne 0$ (see \cite[\S6]{LenzingReiten06}). One can show that a Hom-finite Krull-Schmidt additive category $\aa$ is locally bounded if and only if each standard projective and standard injective module is of finite length.

A small preadditive category $\frak{a}$ is called \emph{semi-hereditary} if the finitely presented objects $\mod \frak{a}$ in $\Mod \frak{a}$ form an abelian and hereditary category.  The following proposition (\cite{vanRoosmalen06}, see also \cite[Theorem 1.6]{AuslanderReiten75}) gives a convenient criterion for semi-heredity.

\begin{proposition}\label{proposition:ShIsLocal}
Let $\frak{a}$ be a small preadditive category; then $\frak{a}$ is semi-hereditary if and only if any full subcategory of $\frak{a}$ with a finite number of objects is semi-hereditary.
\end{proposition}

The following statement is standard (see, for example \cite[Proposition 1.4]{Auslander74}); we provide a proof in the needed generality for the convenience of the reader.
\begin{proposition}\label{proposition:CategoryOfProjectivesIsSemiHereditary}
Let $\AA$ be an essentially small abelian hereditary category.  The full subcategory $\PP$ of all projective objects in $\AA$ is semi-hereditary.
\end{proposition}

\begin{proof}
To show that $\PP$ is semi-hereditary, we need to show that, for all $P,Q \in \PP$, the kernel of a map $f\colon \PP(-,Q) \to \PP(-,P)$ is projective and splits off.  Via the Yoneda Lemma, the map $f$ corresponds to a map $f'\colon Q \to P$ in $\PP$.  Since $\AA$ is hereditary, the kernel $K$ of $f'$ is projective in $\AA$ (hence lies in $\PP$) and splits off.  Since the restriction $(-)_P\colon \AA \to \mod \PP\colon A \mapsto \Hom(-,A)|_\PP$ is exact, the kernel of $f$ is $\PP(-,K)$, and the embedding $\PP(-,K) \to \PP(-,Q)$ is split since the embedding $K \to Q$ is split.
\end{proof}

\begin{remark}\label{remark:modcHereditary}
If we assume that $\aa$ is Hom-finite, then $\aa$ is semi-hereditary if and only if $\aa^\circ$ is semi-hereditary.  In this case, we may infer that $\modc \aa$ is hereditary.
\end{remark}

We also recall the following result (see \cite[Proposition 2.2]{AuslanderReiten75}).
  
\begin{proposition}
\label{proposition:SemiHereditaryExact}
Let $\frak{b}\to \frak{c}$ be a full embedding of semi-hereditary categories.  Then the (fully faithful) functor $- \otimes_{\frak{b}} \frak{c}\colon \mod(\frak{b})\to \mod(\frak{c})$ is exact.
\end{proposition}

\subsection{\texorpdfstring{Dualizing $k$-varieties}{Dualizing k-varieties}}\label{subsection:Dualizing}

We recall some definitions from \cite{Auslander74, AuslanderReiten74}.  A Hom-finite additive $k$-linear category where idempotents split is called a \emph{finite $k$-variety}.  A finite $k$-variety is always Krull-Schmidt.

Denote by $\Modlfd \aa$ the abelian category of locally finite-dimensional right $\aa$-modules, thus the full subcategory of $\Mod \aa$ spanned by all contravariant functors from $\aa$ to $\mod k$.  Note that an additive $k$-linear category where idempotents split is a finite $k$-variety if and only if every standard projective and standard injective lies in $\Modlfd \aa$.

Let $\frak{a}$ be a finite $k$-variety, and denote by $D\colon \Modlfd \aa \to \Modlfd \aa^\circ$ the duality given by sending a module $M \in \Ob \Modlfd \aa$ to the dual $D(M)$ where $D(M)(x) = M(x)^*$ for all $x \in \frak{a}$ (here, $(-)^*$ is the vector-space dual). If this functor induces a duality $D\colon \mod \frak{a} \to \mod \frak{a}^\circ$ by restricting its domain to the finitely presented objects in $\Modlfd \frak{a}$, then we will say that $\frak{a}$ is a \emph{dualizing $k$-variety}.

It follows from \cite{AuslanderSmalo81} that any functorially finite full subcategory $\bb$ (see \S\ref{subsection:FunctoriallyFinite} for a definition) of a dualizing $k$-variety $\aa$ is again dualizing.  It is shown in \cite[Proposition 2.11]{IyamaYoshino08} that a triangulated finite $k$-variety is a dualizing $k$-variety if and only if it satisfies Serre duality.

When $\aa$ is a finite $k$-variety, then $\mod \aa$ is an abelian hereditary category with Serre duality if and only if $\aa$ is a semi-hereditary dualizing $k$-variety (see \cite[Corollary 4.9]{BergVanRoosmalen14}).

\subsection{Thread quivers}\label{subsection:ThreadQuivers}

Thread quivers were introduced in \cite{BergVanRoosmalen14} to classify semi-hereditary dualizing $k$-varieties.  A \emph{thread quiver} consists of the following information:
\begin{itemize}
\item A quiver $Q=(Q_0,Q_1)$ where $Q_0$ is the set of vertices and $Q_1$ is the set of arrows.
\item A decomposition $Q_1 = Q_s \coprod Q_t$.  Arrows in $Q_s$ are called \emph{standard arrows}, while arrows in $Q_t$ are called \emph{thread arrows}.
\item For every thread arrow $t \in Q_t$, there is an associated linearly ordered set $\PP_t$, possibly empty.
\end{itemize}

For a thread quiver $Q$, we denote by $Q_u = (Q_0,Q_1)$ the underlying ``regular'' quiver, thus where all the thread arrows are (unlabeled) standard arrows.  We denote the $k$-linear additive path category of $Q_u$ by $kQ_u$.

With a thread arrow $\xymatrix@1{t\colon x_t \ar@{..>}[r]^-{\PP_t} & y_t}$ in $Q$ we associate a bounded locally discrete linearly ordered poset $\LL_t$: namely $\LL_t = \bN \cdot (\PP_t \stackrel{\rightarrow}{\times}\bZ) \cdot -\bN$.

The poset $\LL_t$ is interpreted as a category in the usual sense, and $k\LL_t$ will denote the associated $k$-linear additive category as in~\S\ref{subsection:Lin ord sets}.  For a thread arrow $\xymatrix@1{t\colon x_t \ar@{..>}[r]^-{\PP_t} & y_t}$ of $Q$, consider the poset $2_t = (\{x_t,y_t\}, \leq)$ given by $x_t < y_t$.  There is an obvious fully faithful functor $2_t \to \LL_t$ mapping $x_t,y_t \in 2_t$ to the minimum and maximum in $\LL_t$, respectively. Let $f^t\colon k2_t \to k\LL_t$ be the $k$-linearized functor.

The semi-hereditary dualizing $k$-variety $kQ$ is then defined as a 2-pushout:
$$\xymatrix{
\bigoplus_{t \in Q_t} k2_t \ar[r]^-f \ar[d]_g & {k Q_u} \ar@{-->}[d]^{i} \\
\bigoplus_{t \in Q_t} k\LL_t \ar@{-->}[r]_-{j} & kQ 
}$$
in the 2-category of small finite $k$-varieties, effectively replacing the thread arrows in $Q$ with the corresponding linearly ordered posets.  We refer to \cite{BergVanRoosmalen14} for more information.

\begin{remark}\label{remark:NonuniqueThreadQuivers}
\begin{enumerate}
\item Different thread quivers $Q,Q'$ may give rise to equivalent $k$-varieties $kQ, kQ'$, \cite[Example 7.10]{BergVanRoosmalen14}.
\item The functors $f^t\colon k2_t \to kQ_u$ (for each $t \in Q_t$) and $f\colon \oplus_t k2_t \to kQ_u$ are faithful, but not full in general.
\item Similarly, the functors $j^t\colon k\LL_t \to kQ$ (for each $t \in Q_t$) and $j\colon \oplus_t k\LL_t \to kQ$ are faithful, but not necessarily full.
\item The functors $g^t\colon k2_t \to k\LL_t$ and $i\colon kQ_u \to kQ$ are fully faithful, as is the functor $g\colon \oplus_t k2_t \to \oplus_t k\LL_t$.
\item For each thread quiver $Q$, there is a thread quiver $Q'$ with $kQ \simeq kQ'$ and for which the functors $f^t\colon k2_t \to kQ'_u$ and $j^t\colon k\LL_t \to kQ'$ are fully faithful (see \cite[Lemma 7.13 and Proposition 7.16]{BergVanRoosmalen14}).  The thread quivers $Q$ and $Q'$ have the same set of thread arrows.
\end{enumerate}
\end{remark}

\begin{definition}
A quiver $Q$ is called \emph{strongly locally finite} if the indecomposable projective and injective representations have finite dimension as $k$-vector spaces.
A thread quiver $Q$ will be called \emph{strongly locally finite} if and only if the underlying quiver $Q_u$ is strongly locally finite.
\end{definition}

\begin{remark}
A quiver $Q$ is strongly locally finite if and only if $kQ$ is locally bounded.  However, if a thread quiver $Q$ has a thread arrow, then $kQ$ is not locally bounded, even if the underlying quiver is strongly locally finite. 
\end{remark}

The following result is shown in \cite[\S 7]{BergVanRoosmalen14}.

\begin{theorem}\label{theorem:ClassificationIntroduction}
Let $Q$ be a thread quiver.  The category $kQ$ is a semi-hereditary dualizing $k$-variety if and only if $Q$ is strongly locally finite. Furthermore, every semi-hereditary dualizing $k$-variety $\TT$ is equivalent to a category $kQ$ where $Q$ is a strongly locally finite thread quiver.
\end{theorem}

\begin{example}\label{Example:ThreadsNotHereditary}
Let $\PP$ be a linearly ordered set and let $Q$ be the thread quiver $\xymatrix@1{ \ar@{..>}[r]^{\PP} &}$, so that $kQ$ is the $k$-linear poset category of the linearly ordered set $\bN \cdot (\PP \stackrel{\rightarrow}{\times}\bZ) \cdot -\bN$.  It follows from Proposition~\ref{proposition:ShIsLocal} that $\mod kQ$ is hereditary.  We will show that the category $\Mod kQ$ is not hereditary.

We consider the projective functor $P_{-0} = kQ(-, -0) \in \Mod kQ$ associated with the maximal element $-0$ in $\bN \cdot (\PP \stackrel{\rightarrow}{\times}\bZ) \cdot -\bN$.  Let $M$ be the submodule of $P_{-0}$ given by
$$M(i) = \begin{cases} k & \mbox{if $i \in \bN$,} \\ 0 & \mbox{if $i \not\in \bN.$} \end{cases}$$

We claim that $M$ is not projective, so that $\Mod kQ$ has global dimension at least two.

By the Yoneda Lemma, we have $\Hom(kQ(-,i), M) \cong M(i)$, for all $i \in \bN \cdot (\PP \stackrel{\rightarrow}{\times}\bZ) \cdot -\bN$.  For each $i \in \bN$, let $f_i \in \Hom(kQ(-,i), M)$ be the morphism associated to $1 \in M(i) = k$.

It is now straightforward to verify that the morphism $\sum_{i \in \bN} f_i\colon \oplus_{i \in \bN} kQ(-,i) \to M$ is a non-split epimorphism.  This shows that $M$ is not projective.  In particular, the module $P_{-0} / M$ has projective dimension at least two (it can be shown that it is exactly two here).
\end{example}

\begin{proposition}\label{proposition:ModNotHereditary}
Let $Q$ be a thread quiver.  We have that $\Mod kQ$ is hereditary if and only if $Q$ has no thread arrows.  
\end{proposition}

\begin{proof}
If $Q$ has no thread arrows, then it is shown in \cite[8.2 Proposition]{GabrielRoiter92} that $\Mod kQ$ is hereditary.  Thus, assume that $Q$ has at least one thread arrow $t \in Q_t$.  As in Remark \ref{remark:NonuniqueThreadQuivers}, we may assume that the functor $j^t\colon k\LL_t \to kQ$ is fully faithful.

Seeking a contradiction, we assume that $\Mod kQ$ is hereditary.  Consider a presentation
$$\oplus_m k\LL_t(-,B_m) \stackrel{f}{\rightarrow} \oplus_n k\LL_t(-,A_n) \to N \to 0$$
of an object $N \in \Mod k\LL_t$.  Applying the functor $- \otimes_{k\LL_t} kQ$ yields the exact sequence
$$\oplus_m kQ(-,j^t B_m) \stackrel{f \otimes kQ}{\rightarrow} \oplus_n kQ(-,j^t A_n) \to N \otimes_{k\LL_t} kQ \to 0.$$
Using the heredity of $\Mod kQ$, we know that the map $K = \ker(f \otimes_{k\LL_t} kQ) \to \oplus_m kQ(-,j^t B_m)$ is a split monomorphism.  Applying the restriction functor $(-)_{k\LL_t}\colon \Mod kQ \to \Mod k\LL_t$, we find that $(K)_{k\LL_t} = \ker f \to \oplus_m kQ(j^t -,j^t B_m) = \oplus_m k\LL_t(-,B_m)$ is a split monomorphism as well.  This implies that $N$ has projective dimension at most one.  However, in Example \ref{Example:ThreadsNotHereditary} we have seen that the global dimension of $k\LL_t$ is at least two.  This is the required contradiction, and we may conclude that $\Mod kQ$ is not hereditary.
\end{proof}

\subsection{Orbit categories}\label{subsection:ClusterConstruction}

Let $\AA$ be an abelian hereditary category with Serre duality; we write $\bS\colon \Db \AA \to \Db \AA$ for the Serre functor.  The orbit category $\CC_\AA = \Db \AA / (\bS \circ [-2])$ is defined by
\begin{eqnarray*}
\Ob \CC_\AA &=& \Ob \Db \AA \\
\Hom_{\CC_\AA}(X,Y) &=& \oplus_{n \in \bZ} \Hom_{\Db \AA}(X,\bS^n Y [-2n]),
\end{eqnarray*}
where the composition is given in the obvious way.  We will refer to $\CC_\AA$ as the cluster category of $\AA$.  It was shown in \cite[Theorem 6]{Keller05} that $\CC_\AA$ can be endowed with the structure of a triangulated category such that the natural functor $\Db \AA \to \CC_\AA$ is triangulated.  Furthermore, $\CC_\AA$ is Hom-finite and is a 2-Calabi-Yau category.

\begin{remark} \label{remark:EquivalentCluster}
Derived equivalent categories give rise to equivalent cluster categories.
\end{remark}

When $\AA$ is of the form $\mod kQ$ where $Q$ is a finite acyclic quiver (or more generally, a strongly locally finite thread quiver, see \S\ref{subsection:ThreadQuivers}), then we will write $\CC_Q$ for $\CC_{\mod kQ}$.

Later we will also use the following well-known proposition.

\begin{proposition} \label{proposition:Cluster without proj-inj}
Let $\AA$ be a hereditary abelian Ext-finite category with Serre duality.  If $\AA$ has no nonzero projective or injective objects, then the composition $\AA \to \Db \AA \to \CC_\AA$ induces a bijection between the isomorphism classes in $\AA$ and the isomorphism classes in $\CC_\AA$.
\end{proposition}

\begin{proof}
Since $\AA$ has no nonzero projective or injective objects, we know that $\tau\colon \Db \AA \to \Db \AA$ restricts to an autoequivalence on $\AA$.  This shows that the $(\bS \circ [-2])$-orbit of an indecomposable object in $\Db \AA$ has exactly one representative in $\AA$ (this uses the natural equivalence $\bS \circ[-2] \cong \tau \circ [1]$ and that $\tau$ induces an autoequivalence on $\AA$).
\end{proof}

\subsection{Functorially finite subcategories and approximations}
\label{subsection:FunctoriallyFinite}

Let $\TT$ be a full subcategory of a {Hom-finite} category $\CC$.  We will say that $\TT$ is \emph{contravariantly finite} in $\CC$ if for every object $X \in \CC$, there is a map $T_0 \to X$ with $T_0 \in \TT$ through which every map $T'_0 \to X$ with $T'_0 \in \TT$ factors. Formulated differently, for each $X \in \CC$, the $\TT$-module $\Hom_\CC(-,X)$ is finitely generated.  The map $T_0 \to X$ is called a \emph{right $\TT$-approximation} of $X$.  A \emph{minimal right $\TT$-approximation} of $X$ is a right $\TT$-approximation such that in addition the map is right minimal.

Dually, we will say that $\TT$ is \emph{covariantly finite} in $\CC$ if for every object $X \in \CC$, there is a map $X \to T_0$ with $T_0 \in \TT$ through which every map $X \to T'_0$ with $T'_0 \in \TT$ factors.  Since $\CC$ is assumed to be Hom-finite, this can be formulated by saying that the $\TT$-module $\Hom_\CC(X,-)^*$ is finitely cogenerated.  The map $X \to T_0$ is called a \emph{left $\TT$-approximation} of $X$ and will be called a \emph{minimal left $\TT$-approximation} if it is left minimal.

When $\TT$ is both covariantly and contravariantly finite in $\CC$, we will say that $\TT$ is \emph{functorially finite} in $\CC$.

\subsection{Cluster-tilting subcategories} \label{subsection:ClusterTilting}

Let $\XX$ be a full subcategory of a 2-Calabi-Yau triangulated category $\CC$.  For each $n \in \bZ$, we will denote by $\XX^{\perp_n}$ and ${}^{\perp_n}\XX$ the following full subcategories of~$\CC$:
\[\XX^{\perp_n} = \{C \in \CC \mid \Hom_\CC (X,C[n]) = 0\} \mbox{ and }{}^{\perp_n} \XX = \{C \in \CC \mid \Hom_\CC (C,X[n]) = 0\}.\]

We will also use the following notation:
\begin{align*}
\XX^\perp &= \bigcap_{n \in \bZ} \XX^{\perp_n} = \{C \in \CC \mid \forall n \in \bZ: \Hom_\CC (X,C[n]) = 0 \} \\
{}^\perp \XX &= \bigcap_{n \in \bZ} {}^{\perp_n} \XX = \{C \in \CC \mid \forall n \in \bZ: \Hom_\CC (C,X[n]) = 0 \}.
\end{align*}

A full subcategory $\TT \subseteq \CC$ is called a \emph{rigid subcategory} if {$\Ext^1_\CC(T,T') = 0$} for all $T,T' \in \TT$.  We say a full subcategory is a \emph{maximal rigid} subcategory or a \emph{weak cluster-tilting subcategory} if it is not properly contained in another rigid subcategory.  Formulated differently, the subcategory $\TT \subseteq \CC$ is maximal rigid if
$$\TT = \{ X \in \CC \mid \Ext^1(T \oplus X, T' \oplus X) = 0 \textrm{ for all } T,T' \in \TT \}. $$

A \emph{cluster-tilting subcategory} of $\CC$ is a functorially finite subcategory $\TT \subseteq \CC$ such that
$$\TT = {}^{\perp_1} \TT = \TT^{\perp_1}.$$

\begin{remark} \label{remark:DefClusterTilting}
\begin{enumerate}
\item It is clear that a cluster-tilting subcategory is maximal rigid.
\item\label{enumerate:MaxRigidVsClusterTilting} It follows from \cite{BuanMarshVatne10, BurbanIyamaKellerReiten08} that not every functorially finite and maximal rigid subcategory is a cluster-tilting subcategory. However, if $\CC$ has a cluster-tilting subcategory or if every indecomposable object of $\CC$ is rigid, then the cluster-tilting subcategories are exactly the functorially finite maximal rigid subcategories (the second claim is easy, the first claim is \cite[Theorem 2.6]{ZhouZhu11}).
\item The category $\CC$ is a dualizing $k$-variety by \cite[Proposition 2.11]{IyamaYoshino08} and hence so is a cluster-tilting subcategory $\TT$ by \cite{AuslanderSmalo81}.  In particular, $\mod \TT$ is abelian.
\item It is shown in \cite[Theorem 2.4]{DehyKeller08} that for any two cluster-tilting subcategories, $\TT$ and $\TT'$ of $\CC$, we have $\lvert\ind \TT\rvert = \lvert\ind \TT'\rvert$.
\end{enumerate}
\end{remark}

Let $\TT$ be a cluster-tilting subcategory of $\CC$.  For an object $X \in \CC$, consider a triangle $T_1 \to T_0 \to X \to T_1[1]$ where $T_0 \to X$ is a minimal right $\TT$-approximation of $X$.  Using the fact that $\TT$ is cluster-tilting, one can then show that $T_1 \in \TT$ (see \cite[Proposition 2.1(b)]{KellerReiten07} or \cite[Theorem 3.1]{IyamaYoshino08}).  A triangle of this form will be called a \emph{right $\TT$-approximation triangle} of $X$.  We define a \emph{left $\TT$-approximation triangle} in a similar way.  Note that a right $\TT$-approximation triangle of $X$ also gives a left $\TT[1]$-approximation of $X$.

The existence of a right $\TT$-approximation for $X$ implies that $\Hom_\CC(-,X) \in \mod \TT$.  Likewise, the existence of a left $\TT$-approximation triangle for $X[-2]$ implies that $\Hom_\CC(-,X) \cong \Hom_\CC(X[-2],-)^*$ is finitely cogenerated.  Moreover, it has been shown in \cite[Proposition 2.1]{KellerReiten07} (see also \cite[Theorem A]{BuanMarshReiten07}, \cite[Corollary 4.4]{KonigZhu08} and \cite[Corollary 6.5]{IyamaYoshino08}) that the correspondence $X \mapsto \Hom_\CC(-,X)$ defines an equivalence $\CC / [\TT[1]] \to \mod \TT$. {As shown in \cite[Proposition 4.3(a)]{AdachiIyamaReiten14} and \cite[Proposition 4.7]{KonigZhu08}, this equivalence nicely interacts with the Auslander-Reiten translation and almost split sequences. For easy reference, we shall summarize these facts in the following proposition. 

\begin{proposition}\label{proposition:AbelianQuotient}
Let $\CC$ be a Hom-finite Krull-Schmidt triangulated 2-Calabi-Yau category.

\begin{enumerate}
\item\label{enumerate:AbelianQuotient1} For any cluster-tilting subcategory $\TT \subseteq \CC$, we have that $\TT$ is coherent (thus, $\mod \TT$ is abelian) and the correspondence $X \mapsto \Hom_\CC(-,X)|_\TT$ defines an equivalence $H_\TT\colon \CC / [\TT[1]] \to \mod \TT$.

\item\label{enumerate:AbelianQuotient2} Under the latter equivalence, the category of projectives of $\mod \TT$ corresponds to $\TT \subseteq \CC / [\TT[1]]$ and the category of injectives of $\mod \TT$ corresponds to $\bS \TT = \TT[2] \subseteq \CC / [\TT[1]]$.

\item\label{enumerate:AbelianQuotient3} Let $X \in \CC$ be indecomposable.  If $X \not\in \TT[1]$, then $H_\TT(X[1]) \cong \tau H_\TT(X)$ in $\mod\TT$.  If, moreover, $X \not\in \TT$, then an Auslander-Reiten triangle $\tau X \to M \to X \to \tau X[1]$ in $\CC$ induces the almost split sequence $0 \to H_\TT (\tau X) \to H_\TT (M) \to H_\TT(X) \to 0$ in $\mod\TT$.
\end{enumerate}
\end{proposition}

Regarding homological properties of $\TT$, it has been shown in \cite{KellerReiten07} that $\mod \TT$ is 1-Gorenstein and stably Calabi-Yau. It will be useful to record the following relation of the Ext groups in $\CC$ and in $\mod\TT$.

\begin{lemma} \label{lemma:Ext-via-H}
Let $\CC$ be a Hom-finite Krull-Schmidt triangulated 2-Calabi-Yau category with a cluster-tilting subcategory $\TT$.  If $X,Y \in \CC$ have no indecomposable summands in $\TT[1]$, then
\[ \big(\Hom_\TT(H_\TT(Y),\tau H_\TT(X))\big)^* \cong [\TT[1]](X,Y[1]) \subseteq \Ext^1_\CC(X,Y). \]
If, moreover, $\mod\TT$ is hereditary, then
\[ \Ext^1_\TT(H_\TT(X),H_\TT(Y)) \cong [\TT[1]](X,Y[1]) \subseteq \Ext^1_\CC(X,Y). \]
\end{lemma}

\begin{proof}
The first part was proved in~\cite[Lemma 3.3]{Palu08} (see also \cite[Proposition 4.3(b)]{AdachiIyamaReiten14}).  The second part follows from the first one by the Auslander-Reiten formula {(Lemma~\ref{lemma:AR-formula}).}
\end{proof}
}

Finally, we will use the following result (see \cite{BuanMarshReinekeReitenTodorov06} and \cite[Lemma 4.4]{LiuPaquette15} for representations of finite quivers and (infinite) strongly locally finite quivers, respectively).  We give a proof for the benefit of the reader.

\begin{proposition}\label{proposition:CTFromEnoughProj} Let $\AA$ be an abelian Ext-finite hereditary category with Serre duality.
\begin{enumerate}
\item\label{emunerate:FundamantalDomain} Let $A \in \CC_\AA$ be indecomposable.  There is a unique (up to isomorphism) lift $\tilde{A} \in \Db \AA$ of $A$ satisfying the condition that either $H^{-1}(\tilde{A})$ is nonzero and projective, or $H^{0}(\tilde{A}) \not= 0$.
\item Let $\PP \subseteq \AA$ be the category of projective objects.  If $\AA$ has enough projectives, then the image of $\PP$ under the composition $F\colon \AA \to \Db \AA \to \CC_\AA$ is a cluster-tilting subcategory.
\end{enumerate}
\end{proposition}

\begin{proof}
\begin{enumerate}
\item Let $A \in \CC_\AA$ be an indecomposable object, and let $\overline{A} \in \Db \AA$ be any lift of $A$.  Since $\overline{A}$ is indecomposable as well, and $\AA$ is hereditary, we know that there is a unique $i \in \bZ$ such that $H^i(\overline{A}) \not= 0$.

If $H^i(\overline{A})$ is projective, then $H^i(\bS\overline{A})$ is nonzero and injective in $\AA$, otherwise $H^{i-1}(\bS\overline{A})$ is nonzero (see \cite[Corollary I.3.4]{ReVdB02}).  Likewise, if $H^i(\overline{A})$ is injective, then $H^i(\bS^{-1}\overline{A})$ is nonzero and projective in $\AA$, otherwise $H^{i+1}(\bS^{-1}\overline{A})$ is nonzero.

From this we infer that the $\bS\circ [-2]$-orbit of $\overline{A}$ contains an object $\tilde{A}$ such that either $H^{0}(\tilde{A})$ is nonzero, or $H^{-1}(\tilde{A})$ is nonzero and projective.

\item Let $A \in \ind (F(\PP)^{\perp_1})$, so that $\Hom_{\CC_\AA}(F(\PP), A[1]) = 0$.  From the above argument, we deduce the existence of a lift $\widetilde{A[1]}$ of $A[1]$ such that either $H^{0}(\widetilde{A[1]})$ is nonzero, or $H^{-1}(\widetilde{A[1]})$ is nonzero and projective.  Since $\AA$ has enough projectives and $\Hom_{\CC_\AA}(F(\PP), A[1]) = 0$, we may exclude the former case.  Hence $\widetilde{A[1]} \in \PP[1]$ and $A \in F(\PP)$.

All that remains is to show that $F(\PP)$ is functorially finite.  Let $A \in \ind \CC_\AA$, and let $\tilde{A}$ be a lift as in (\ref{emunerate:FundamantalDomain}).  Since $\AA$ has enough projective objects, there is a right $\PP[0]$-approximation $P \to \tilde{A}$ (note that $P = 0$ if $H^{0}(\tilde{A}) = 0$, thus is $A \in \PP[1]$).  It is easy to check that $P \to \tilde{A}$ stays a right $\PP[0]$-approximation after applying $\Db \AA \to \CC_\AA$.  Showing that $F(\PP)$ is covariantly finite is similar.
\end{enumerate}
\end{proof}

\subsection{Torsion and mutation pairs} \label{subsection:pairs}

Let $\CC$ be a 2-Calabi-Yau triangulated category.  A pair of full subcategories $(\XX,\YY)$ of $\CC$ is called a \emph{torsion pair} if $\Hom(\XX,\YY) = 0$ and every object $C \in \CC$ lies in a triangle $X \to C \to Y \to X[1]$ where $X \in \XX$ and $Y \in \YY$.  In this case, $\XX$ is a contravariantly finite and extension-closed subcategory of $\CC$ and $\YY$ is a covariantly finite and extension-closed subcategory of $\CC$. In fact, these properties characterize the classes in a torsion pair:

\begin{lemma}\label{lemma:TorsionPairs} \cite[Proposition 2.3]{IyamaYoshino08}
Let $\XX$ be a contravariantly finite and extension-closed subcategory of $\CC$. Then $(\XX,\XX^{\perp_0})$ is a torsion pair. Dually, if $\YY$ is covariantly finite and extension-closed in $\CC$, then $({}^{\perp_0} \YY, \YY)$ is a torsion pair.
\end{lemma}

If $(\XX,\YY)$ is a torsion pair, we will say that $(\XX[1],\YY)$ is a \emph{cotorsion pair}.

Let $\ZZ$ be a full rigid subcategory of $\CC$.  For a subcategory $\XX \subseteq \CC$, we consider
\begin{enumerate}
\item the subcategory $\mu_\ZZ(\XX)$ consisting of all objects $M \in \CC$ such that there is a triangle
$$M \to Z_X \stackrel{f}{\rightarrow} X \to M[1]$$
where $f\colon Z_X \to X$ is a right $\ZZ$-approximation of $X \in \XX$, and
\item the subcategory $\mu^{-}_\ZZ(\XX)$ consisting of all objects $M' \in \CC$ such that there is a triangle
$$X \stackrel{g}{\rightarrow} Z^X \to M' \to X[1]$$
where $g\colon X \to Z^X$ is a left $\ZZ$-approximation of $X \in \XX$.
\end{enumerate}

\begin{definition}
We say that $(\XX, \YY)$ is a \emph{$\ZZ$-mutation pair} if and only if
$$\mbox{$\ZZ \subseteq \XX \subseteq \mu_\ZZ(\YY)$ and $\ZZ \subseteq \YY \subseteq \mu^{-}_\ZZ(\XX)$.}$$
\end{definition}

\begin{remark}
It has been shown in \cite[Proposition 2.6(2)]{IyamaYoshino08} that $\XX = \mu_\ZZ(\YY)$ and $\YY = \mu^{-}_\ZZ(\XX)$ for a $\ZZ$-mutation pair $(\XX, \YY)$.
\end{remark}

\begin{example}\label{example:RigidMutationPair}
When $\ZZ$ is a functorially finite rigid subcategory of $\CC$, then $(\ZZ^{\perp_1}, {^{\perp_1}\ZZ})$ is a $\ZZ$-mutation pair (this is straightforward to check, or can be obtained as a corollary from \cite[Proposition 2.7]{IyamaYoshino08}).

More generally, \cite[Proposition 2.7]{IyamaYoshino08} says that given any additive full subcategory $\XX \subseteq \ZZ^{\perp_1}$ which contains $\ZZ$, the $\ZZ$-mutation pair $(\ZZ^{\perp_1}, {^{\perp_1}\ZZ})$ restricts to a $\ZZ$-mutation pair $(\XX, \YY)$, where $\YY = \mu^{-}_\ZZ(\XX)$.
\end{example}

Following \cite[Definition 5.2]{IyamaYoshino08}, we say that a functorially finite subcategory $\DD$ of $\CC$ is an \emph{almost complete cluster-tilting subcategory} if there is a cluster-tilting subcategory $\XX$ with $\DD \subseteq \XX$ such that $\ind \XX \setminus \ind \DD$ consists of a single element.  We will use the following theorem (see \cite[Theorem 5.3]{IyamaYoshino08}).

\begin{theorem}\label{theorem:IyamaYoshinoMutation}
Let $\CC$ be a Krull-Schmidt triangulated 2-Calabi-Yau category.  Any almost complete cluster-tilting subcategory $\DD$ of $\CC$ is
contained in exactly two cluster-tilting subcategories, $\XX$ and $\YY$ of $\CC$.  Both $(\XX,\YY)$ and $(\YY,\XX)$ form $\DD$-mutation pairs.
\end{theorem}

The following proposition (see \cite[Lemma 4.1]{JorgensenPalu13}) allows us to find almost complete cluster-tilting subcategories.

\begin{proposition}\label{proposition:JorgensenPaluCofinite}
Let $\CC$ be a Krull-Schmidt triangulated 2-Calabi-Yau category.  Let $\XX \subseteq \CC$ be a cluster-tilting subcategory and let $\DD \subseteq \XX$.  If $\ind \XX \setminus \ind \DD$ is finite, then $\DD$ is functorially finite in $\CC$.
\end{proposition}

\begin{corollary}\label{corollary:Mutation}
Let $\CC$ be a Krull-Schmidt triangulated 2-Calabi-Yau category and let $\XX$ be a cluster-tilting subcategory of $\CC$.  For each $X \in \ind \XX$, there is a $Y \in \ind \CC$, nonisomorphic to $X$, such that $\YY = \add[(\ind \XX \setminus \{X\}) \cup \{Y\}]$ is a cluster-tilting subcategory.  Moreover (writing $\DD = \add(\ind \XX \setminus \{X\})$), for each $X,Y$ as above, there are triangles
$$X \stackrel{f}{\rightarrow} D \stackrel{g}{\rightarrow} Y \mbox{ and } Y \stackrel{s}{\rightarrow} D' \stackrel{t}{\rightarrow} X$$
where $f$ and $s$ are minimal left $\DD$-approximations, and $g$ and $t$ are minimal right $\DD$-approximations.
\end{corollary}

\begin{proof}
This is an application of Theorem \ref{theorem:IyamaYoshinoMutation} and Proposition \ref{proposition:JorgensenPaluCofinite}.
\end{proof}

\begin{definition}\label{definition:ExchangeGraph}
Let $\CC$ be a Krull-Schmidt triangulated 2-Calabi-Yau category.  The \emph{exchange graph} of $\CC$ is the graph whose vertices are the cluster-tilting subcategories, and where there is an edge $\xymatrix@1 {\XX \ar@{-}[r] & \YY}$ if and only if $\XX$ and $\YY$ are as in Theorem \ref{theorem:IyamaYoshinoMutation}, meaning that there is an almost complete cluster-tilting subcategory $\DD$ such that both $(\XX,\YY)$ and $(\YY,\XX)$ are $\DD$-mutation pairs.
\end{definition}

\begin{remark}
There is an edge $\xymatrix@1 {\XX \ar@{-}[r] & \YY}$ in the exchange graph of $\CC$ if and only if $\XX \cap \YY$ is an almost complete cluster-tilting subcategory.  It follows from Corollary \ref{corollary:Mutation} that the edges incident with a cluster-tilting subcategory $\XX$ are naturally in bijection with the objects of $\ind \XX$.  The exchange graph of $\CC$ does not need to be connected.
\end{remark}
\section{Directed cluster-tilting subcategories}\label{section:TriangleFree}

Let $\CC$ be a Krull-Schmidt 2-Calabi-Yau triangulated category.  Assume that $\CC$ has a cluster-tilting object $T \in \CC$.  One of the main results of \cite{KellerReiten08} is that the quiver of $T$ is acyclic if and only if the algebra $\End(T)$ is hereditary.  The main result in this section (Theorem \ref{theorem:Trianglefree}) generalizes this result to more general cluster-tilting subcategories.

The following definitions are related to the concept of acyclic quivers.

\begin{definition}\label{definition:Directed}
Let $\CC$ be a Krull-Schmidt category.  A \emph{path} of length $n$ in $\CC$ is a sequence $X_0, X_1, \ldots, X_n$ of indecomposable objects such that for every $i \in \{0, 1, \ldots, n-1\}$ there is a nonzero noninvertible morphism $X_i \to X_{i+1}$.  A \emph{cycle} in $\CC$ is a path $X_0, X_1, \ldots, X_n$ with $n \geq 1$ and $X_0 \cong X_n$.  We say that $\CC$ is \emph{directed} if it contains no cycles.
\end{definition}

The following theorem strengthens~\cite[Corollary 2.1]{KellerReiten08}, and it refines Proposition~\ref{proposition:ShIsLocal} for the special case when $\aa = \TT$ is a cluster-tilting subcategory of $\CC$.

\begin{theorem}\label{theorem:Trianglefree}
Let $\CC$ be a Krull-Schmidt 2-Calabi-Yau triangulated category with a cluster-tilting subcategory $\TT$.  If $\TT$ has no cycles of length three, then $\mod \TT$ is hereditary.
\end{theorem}

\begin{proof}
Seeking a contradiction, assume that $\mod \TT$ is not hereditary.  Let $X \in \mod \TT$ be an object with projective dimension at least two and consider the beginning of a minimal projective resolution
$$(-,T_0) \to (-,T_1) \to (-,T_2) \to X \to 0$$
where $T_0 \not= 0$.  Let $f: T_0 \to T_1$ and $g:T_1 \to T_2$ be the corresponding maps in $\TT$ via the Yoneda Lemma.  Note that every map $h:T \to T_1$ with $g \circ h = 0$ factors through $f$, i.e. the map $f$ is a weak kernel of $g$ in $\TT$.

Let $T'_0 \in \ind\TT$ be an indecomposable direct summand of $T_0$.  Using Corollary \ref{corollary:Mutation}, there are triangles
$$\mbox{$T'_0 \stackrel{l}{\rightarrow} D \stackrel{r}{\rightarrow} (T'_0)^* \to T'_0[1]$ and $(T'_0)^* \stackrel{l^*}{\rightarrow} E \stackrel{r^*}{\rightarrow} T'_0 \to (T'_0)^*[1],$}$$
with $(T'_0)^*$ indecomposable, $D,E \in \add(\ind \TT \setminus \{T'_0\})$, and where $l: T'_0 \to D$ is a left $\add(T \setminus \{T'_0\})$-approximation of $T'_0$.

Next we find that the morphism $f': T'_0 \to T_1$ factors through the left $\add(\ind\TT \setminus \{T'_0\})$-approximation of $T'_0$ as
$$f': T'_0 \stackrel{l}{\rightarrow} D \stackrel{j}{\rightarrow} T_1.$$
To see this, we decompose $T_1$ as $T_1 = (T'_0)^{\oplus n} \oplus T'_1$ where $T'_1$ has no summand isomorphic to $T'_0$. Since the composition $T'_0 \to T_0 \to T_1$ is a radical map (because the map $(-,T_0) \to (-,T_1)$ is part of a minimal projective resolution), $f'$ must factor as $T'_0 \stackrel{f''}\to T'_1 \to T_1$ where $T'_1 \to T_1$ is the split inclusion (otherwise we would have a cycle $T'_0 \to T'_0 \to T'_0 \to T'_0$). Now $T'_1 \in \add(\ind \TT \setminus \{T'_0\})$ and so $f''$ factors through the left approximation $l$, as claimed.

We will show that $g \circ j: D \to T_1 \to T_2$ is nonzero.  Seeking a contradiction, assume that $g \circ j = 0$.  This implies that $j: D \to T_1$ factors through $f:T_0 \to T_1$ since $f$ is a weak kernel of $g$, and hence yields a factorization of $f': T'_0 \to T_1$ as $T'_0 \to D \to T_0 \to T_1$.  The minimality of the projective resolution of $X$ yields that the composition  $T'_0 \to D \to T_0$ is a split monomorphism.  In particular, the map $T'_0 \to D$ is a split monomorphism which is a contradiction since $D \in \add(\ind \TT \setminus \{T'_0\})$.  We have shown that $g \circ j \not= 0$.

We will now show that $l^* \circ r: D \to (T'_0)^* \to E$ is a nonzero radical map. Clearly, $l^* \circ r$ is a radical map since both $l^*$ and $r$ are. Note that the composition $g \circ (j \circ l): T'_0 \to D \to T_1 \to T_2$ is zero, so that $g \circ j$ factors as below:
\[\xymatrix{T'_0 \ar[r]^{l} & D \ar[r]^{r} \ar[d]_{g \circ j}& (T'_0)^* \ar[r] \ar@{-->}[dl]^{h} & T'_0[1] \\ & T_2}\]
Since $g \circ j$ is nonzero, we know that map $h: (T'_0)^* \to T_2$ is nonzero.  Furthermore, we see from the exchange triangle $(T'_0)^* \stackrel{l^*}{\rightarrow} E \stackrel{r^*}{\rightarrow} T'_0 \to (T'_0)^*[1]$ and $\Ext^1(T'_0,T_2) = 0$ (since $T'_0, T_2 \in \TT$) that $h$ factors through the map $l^*\colon (T'_0)^* \to E$:
\[\xymatrix{& & {T'_0[-1]} \ar[d] & \\
T'_0 \ar[r]^{l} & D \ar[r]^{r} \ar[d]_{g \circ j}& (T'_0)^* \ar[r] \ar[dl]^{h} \ar[d]^{l^*} & T'_0[1] \\
& T_2 & E \ar@{-->}[l]^{e} \ar[d]^{r^*} \\ && T'_0}\]
Since $e \circ l^* \circ r = h \circ r = g \circ j \not= 0$, we infer that $l^* \circ r$ is nonzero.

We then find nonzero direct summands $D'$ and $E'$ of $D$ and $E$ respectively such that $\rad(D',E') \not= 0$. Thus, we have found a sequence of nonzero nonisomorphisms $T'_0 \to D' \to E' \to T'_0$ in $\TT$, contradicting the hypothesis that $\CC$ is has no cycles of length three.
\end{proof}

\begin{remark}
It is shown in \cite[Theorem 5.7]{BertaniOklandOpperman11} that there exists a Krull-Schmidt 2-Calabi-Yau triangulated category with a cluster-tilting object $T = T_1 \oplus T_2 \oplus T_3 \oplus T_4$ whose quiver is given by
\[\xymatrix@1{T_1 & T_2 \ar[l] \ar@<2pt>[r]^\alpha & T_3 \ar@<2pt>[l]^\beta \ar[r] & T_4.}\]
The generating relations are $\beta\alpha\beta$ and $\alpha\beta\alpha$.  Hence, $\add T$ is not directed; there is for instance a cycle (in the sense of Definition~\ref{definition:Directed})
\[\xymatrix{T_2 \ar[r]^\alpha & T_3 \ar[r]^{\alpha\beta} & T_3 \ar[r]^\beta & T_2.}\]
Yet there exist no cycles $A \to B \to C \to A$ where $A,B,C$ are pair-wise nonisomorphic.  This shows that Theorem~\ref{theorem:Trianglefree} does not hold if one only excludes the existence of cycles $A \to B \to C \to A$ where $A,B,C$ are pair-wise nonisomorphic.
\end{remark}

\begin{corollary}\label{corollary:Trianglefree}
Let $\CC$ be a Krull-Schmidt 2-Calabi-Yau category.  Assume that $\CC$ has a cluster-tilting subcategory $\TT$.  The following are equivalent:
\begin{enumerate}
\item\label{enumerate:SemiHereditary} $\TT$ is semi-hereditary,
\item\label{enumerate:Directed} $\TT$ is directed,
\item\label{enumerate:Trianglefree} $\TT$ has no cycles of length three,
\item\label{enumerate:ThreadQuiver} $\TT \cong kQ$ for a strongly locally finite thread quiver $Q$.
\end{enumerate}
If $\Mod \TT$ is hereditary, then the previous points are equivalent to
\begin{enumerate}
\item[(5)] the Auslander-Reiten quiver of $\TT$ is acyclic.
\end{enumerate}
\end{corollary}

\begin{proof}
Let $\TT$ be a Hom-finite additive semi-hereditary category.  It follows from \cite[Proposition 1.1]{AuslanderReiten75} that every morphism in $\TT$ between indecomposable objects is a monomorphism.  Hence, a Hom-finite semi-hereditary category is directed.

It is obvious that if $\TT$ is directed, then $\TT$ has no cycles of length three.

If $\TT$ has no cycles of length three, then it follows from Theorem \ref{theorem:Trianglefree} that $\TT$ is semi-hereditary.

Note that $\CC$ is a dualizing $k$-variety due to \cite[Proposition 2.11]{IyamaYoshino08}, and $\TT$ is a dualizing $k$-variety due to \cite[Proposition 2.10(1)]{IyamaYoshino08}.  The equivalence (\ref{enumerate:SemiHereditary}) $\Leftrightarrow$ (\ref{enumerate:ThreadQuiver}) follows from \cite[Theorem 7.21]{BergVanRoosmalen14}.

If $\Mod \TT$ is hereditary, then Proposition \ref{proposition:ModNotHereditary} shows that $\TT \cong kQ$ where $Q$ is a quiver (or a thread quiver without thread arrows).  Here, $Q$ is the Auslander-Reiten quiver of $\TT$.  We know that the category $\TT$ is directed if and only if $Q$ is acyclic.  This establishes the final equivalence.
\end{proof}

The condition that $\Mod \TT$ is hereditary cannot be removed from Corollary \ref{corollary:Trianglefree}, as is illustrated by the following example. 

\begin{example}\label{example:ClusterD}
Let $\AA$ be the category $\repc D_\bZ$ discussed in \cite[\S 4.2]{vanRoosmalen06}, i.e. let $D_\bZ$ be the poset whose underlying set is $\bZ \coprod \{Q_1, Q_2\}$ and whose order relation is given by
\[ X < Y \Leftrightarrow \begin{cases} \mbox{$X,Y \in \bZ$ and $X <_\bZ Y$, or} \\ \mbox{$X \in \bZ$ and $Y \in \{Q_1,Q_2\}.$} \end{cases} \]
The indecomposable objects are given by
\begin{align*}
A_{i,j} &= \coker((kD_\bZ)(-,i-1) \to (kD_\bZ)(-,j)), \\
A^{1}_{i} &= \coker((kD_\bZ)(-,i-1) \to (kD_\bZ)(-,Q_{1})), \\
A^{2}_{i} &= \coker((kD_\bZ)(-,i-1) \to (kD_\bZ)(-,Q_{2})), \\
B_{i,j} &= \coker((kD_\bZ)(-,i-1) \oplus (k\DL)(-,j-1) \to (kD_\bZ)(Q_{1},-) \oplus (kD_\bZ)(Q_{2},-)),
\end{align*}
where $i \leq j$ in $\bZ$ (for $B_{i,j}$, we also assume that $i < j$).  We know (see \cite[Proposition 4.8]{vanRoosmalen06}) that $\repc D_\bZ$ is a hereditary category with Serre duality and without any nonzero projective and injective objects; the Auslander-Reiten translate is given by:
\begin{align*}
\tau A_{i,j} &\cong A_{i-1,j-1}, \\
\tau A^{1}_{i} &\cong A^2_{i-1}, \\
\tau A^{2}_{i} &\cong A^1_{i-1}, \\
\tau B_{i,j} &\cong B_{i-1,j-1}.
\end{align*}

We write $\CC$ for the category $\Db (\repc D_\bZ) / \bS[-2]$.  It follows from \S\ref{subsection:ClusterConstruction} that
\begin{align*}
\Ob \CC &= \Ob \repc D_\bZ, \\
\Hom_\CC(X,Y) &= \Hom_{\AA} (X,Y) \oplus \Ext^1_{\Db \AA}(X, \tau^{-1} Y).
\end{align*}

Consider the full subcategory $\TT = \add(\{A^1_i\}_{i \in \bZ})$ of $\CC$; it can be readily confirmed that $\TT$ is a maximal rigid subcategory.  For the indecomposable objects of $\CC$, we have the following right $\TT$-approximation triangles:
\[ \xymatrix@R=3pt{
A^1_i \ar[r] & A^{1}_{j+1} \ar[r] & A_{i,j} \ar[r] & A^1_i[1] \\
0 \ar[r] & A^1_i \ar[r] & A^1_i \ar[r] & 0 \\
A^1_{i+1} \ar[r] & 0 \ar[r] & A^2_i \ar[r] & A^1_{i+1}[1] \\
A^1_{j+1} \ar[r] & A^1_{i} \ar[r] & B_{i,j} \ar[r] & A^1_{j+1}[1]
} \]
This shows that $\TT$ is a cluster-tilting subcategory of $\CC$.  In $\CC$, there is a sequence of nonzero morphisms {$A^{1}_0 \to A^1_1 = A^2_0[-1] \to A^1_{-1} \to A^{1}_0$}; thus $\TT$ is not directed (as there is a cycle of length three).  However, the Auslander-Reiten quiver of $\TT$ is given by a linearly ordered quiver of type $A_\infty^\infty$:
\[ \xymatrix@1{
\cdots \ar[r] & \cdot \ar[r] & \cdot \ar[r] & \cdot \ar[r] & \cdots
} \]
Hence, the Auslander-Reiten quiver of $\TT$ is acyclic.
\end{example}

It is shown in Corollary \ref{corollary:Trianglefree} that thread quivers classify the possible directed cluster-tilting subcategories.  The following theorem shows that every strongly locally finite thread quiver occurs as a cluster-tilting subcategory of some Krull-Schmidt 2-Calabi-Yau category.

\begin{theorem}\label{theorem:AllThreadQuiversOccur}
Let $Q$ be a strongly locally finite thread quiver.  The category $\Db(\mod kQ) / (\bS \circ [-2])$ is an algebraic Krull-Schmidt 2-Calabi-Yau category which admits a cluster-tilting subcategory equivalent to $kQ$.
\end{theorem}

\begin{proof}
It follows from \cite[Theorem 1]{BergVanRoosmalen14} that $\mod kQ$ is a hereditary Ext-finite category with Serre duality and enough
projectives; the category of projectives is equivalent to $kQ$ (via the Yoneda embedding $kQ \to \mod kQ: A \mapsto (-,A)$).  It follows from \cite[Theorem 6]{Keller05} that $\CC = \Db(\mod kQ) / (\bS \circ [-2])$ is an algebraic triangulated 2-Calabi-Yau category and from \cite[Proposition 1.2]{BuanMarshReinekeReitenTodorov06} that $\CC$ is Krull-Schmidt (the proof carries over to this setting).  It then follows from Proposition \ref{proposition:CTFromEnoughProj} that $\CC$ admits a cluster-tilting subcategory equivalent to $kQ$.
\end{proof}

\section{Iyama-Yoshino reductions}\label{section:Reduction}

In this section, we recall the definition of and some results about the Iyama-Yoshino reduction (see \cite{IyamaYoshino08}, also called the Calabi-Yau reduction).  Our main results in this section are Theorems \ref{theorem:Covering} and \ref{theorem:HereditaryBongartzComplement}, describing these reductions for Krull-Schmidt 2-Calabi-Yau categories with a directed cluster-tilting suibcategory.

\subsection{Preliminaries on Iyama-Yoshino reductions}\label{Subsection:IyamaYoshinoReduction}

Let $\CC$ be a Krull-Schmidt 2-Calabi-Yau category.  Let $\ZZ$ be a full rigid functorially finite subcategory of $\CC$.  We consider the category $\CC_{[\ZZ]} = \ZZ^{\perp_1} / [\ZZ]$.

We may endow $\CC_{[\ZZ]}$ with an auto-equivalence $\langle 1 \rangle: \CC_{[\ZZ]} \to \CC_{[\ZZ]}$ given in the following way.  Let $f: X \to Y$ be a morphism in $\CC_{[\ZZ]}$ and let $\tilde{f}: X \to Y$ be a lift in $\ZZ^{\perp_1} \subseteq \CC$.  Consider the commutative square in $\CC$:
\[\xymatrix{X \ar[r] \ar[d]^{\tilde{f}} & X' \ar[d]^{\tilde{g}} \\ Y \ar[r] &Y'}\]
where the maps $X \to X'$ and $Y \to Y'$ are minimal left $\ZZ$-approximations.  We may extend this commutative square to a morphism of triangles in $\CC$:
  \[\xymatrix{X \ar[r] \ar[d]^f & X' \ar[d]^g \ar[r] & X \langle 1 \rangle \ar[r] \ar@{-->}[d]^{\tilde{h}} & X[1] \ar[d]^{f[1]}\\
	Y \ar[r] &Y' \ar[r] & Y \langle 1 \rangle \ar[r] & Y[1]}\]
One can check that $X \langle 1 \rangle, Y \langle 1 \rangle \in \ZZ^{\perp_1}$.  The image of $\tilde{h}: X \langle 1 \rangle \to Y \langle 1 \rangle$ under the functor $\ZZ^{\perp_1} \to \ZZ^{\perp_1} / [\ZZ]$ is $f \langle 1 \rangle$.  It is shown in \cite[Proposition 2.6]{IyamaYoshino08} that this correspondence is indeed a well-defined equivalence $\CC_{[\ZZ]} \to \CC_{[\ZZ]}$.

Let $A \to B \to C \to A[1]$ be a triangle in $\CC$, and assume that $A,B,C \in \ZZ^{\perp_1}$.  Under the quotient functor $\ZZ^{\perp_1} \to \ZZ^{\perp_1} / [\ZZ] = \CC_{[\ZZ]}$, this triangle induces a complex $A \to B \to C \to A \langle 1 \rangle$.  We can impose the structure of a triangulated category on $\CC_{[\ZZ]}$ by taking as the class of triangles all complexes obtained in this way (see \cite[Theorem 4.2]{IyamaYoshino08}).

If $\CC$ is an algebraic triangulated category, then $\CC_{[\ZZ]}$ with the above induced triangulated structure is also an algebraic triangulated category (see \cite[Proposition 4.3]{AmiotOppermann12}).

Finally, we remark that $\CC_{[\ZZ]}$ is a 2-Calabi-Yau category (see \cite[Theorem 4.7]{IyamaYoshino08}).

\subsection{Reducing the cluster-tilting subcategory to a cluster-tilting object}

Let $\CC$ be a Krull-Schmidt 2-Calabi-Yau category with a directed cluster-tilting subcategory $\TT$.  Our goal in this section is to prove Theorem \ref{theorem:Covering} below, where we say that $\CC$ can be understood ``locally.''  By this, we mean that any question about a finite set of objects in $\CC$ can be reduced to a question pertaining to finitely many objects in a 2-Calabi-Yau category with a cluster-tilting object.  To make this more precise, we will need the following definition.

\begin{definition}\label{definition:LivesIn}Let $\ZZ$ be a functorially finite and extension-closed full subcategory of $\CC$.  We say that an object $X \in \CC$ \emph{lives in} the reduced category $\CC_{[\ZZ]}$ if $X \in \ZZ^{\perp_1}$ and $[\ZZ](X,X) = 0$.  Similarly, we say that a full subcategory $\DD \subseteq \CC$ lives in $\CC_{[\ZZ]}$ if every $X \in \DD$ lives in $\CC_{[\ZZ]}$.
\end{definition}

\begin{remark}
Thus an object $X \in \CC$ lives in the reduced category $\CC_{[\ZZ]}$ if $\Hom_\CC(X,X) = \Hom_{\CC_{[\ZZ]}}(X,X)$.  If $X \oplus X[n]$ lives in $\CC_{[\ZZ]}$, then $\Ext^n_\CC(X,X) = \Ext^n_{\CC_{[\ZZ]}}(X,X)$.
\end{remark}

\begin{theorem}\label{theorem:Covering}
Let $\CC$ be a Krull-Schmidt 2-Calabi-Yau category with a directed cluster-tilting subcategory $\TT$.  For any object $X \in \CC$, there is a functorially finite and rigid full subcategory $\ZZ \subseteq \CC$ such that
\begin{enumerate}
\item $X$ lives in $\CC_{[\ZZ]}$, and
\item there is an object $T \in \TT$ which lives in $\CC_{[\ZZ]}$ and is a cluster-tilting object in $\CC_{[\ZZ]}$.
\end{enumerate}
If $\CC$ is algebraic, then $\CC_{[\ZZ]} \cong \CC_{\End T}$.
\end{theorem}

\begin{proof}
This follows directly from Lemma \ref{lemma:Covering} below.
\end{proof}

We will use the following notation.

\begin{notation}\label{notation:DT}
Let $T \in \TT$ and $X \in \CC$.  Let $T_1 \to T_0 \to X \oplus X[1] \to T_1[1]$ be a right $\TT$-approximation triangle of $X \oplus X[1]$ (thus, $T_0,T_1 \in \TT$), and let $T'_1 \to T'_0 \to T_1[2] \to T'_1[1]$ be a right $\TT$-approximation triangle of $T_1[2]$.

We will write $\DD_T$ for the full subcategory of $\CC$ consisting of all $X \in \CC$ such that $T_1 \oplus T_0 \oplus T'_1 \oplus T'_0 \in \add(T)$.
\end{notation}

\begin{lemma}\label{lemma:AboutDT}
\begin{enumerate}
\item For each $T \in \TT$, the subcategory $\DD_T \subseteq \CC$ is closed under direct sums, summands, and isomorphisms.
\item If $\add T \subseteq \add T'$, then $\DD_{T} \subseteq \DD_{T'}$.
\end{enumerate}
\end{lemma}

\begin{proof}
This is straightforward.
\end{proof}

\begin{proposition}\label{proposition:ColimitOfDT}
Let $\CC$ be a Krull-Schmidt 2-Calabi-Yau category with a directed cluster-tilting subcategory $\TT$.  We have
$$\CC \cong 2\colim_{T \in \TT} \DD_T.$$
\end{proposition}

\begin{proof}
It is clear that, for every $C \in \CC$, there is a $T \in \TT$ such that $C \in \DD_T \subseteq \CC$.  The statement follows from \cite[Proposition A.3.6]{Waschkies04}.
\end{proof}

\begin{lemma}\label{lemma:Covering}
For each $T \in \TT$, there is a functorially finite and rigid full subcategory $\ZZ \subseteq \CC$ such that
\begin{enumerate}
\item every object of $\DD_T$ lives in $\CC_{[\ZZ]}$,
\item $T$ lives in $\CC_{[\ZZ]}$, and
\item $T$ is a cluster-tilting object in $\CC_{[\ZZ]}$.
\end{enumerate}
If $\CC$ is algebraic, then $\CC_{[\ZZ]} \cong \CC_{\End T}$.
\end{lemma}

\begin{proof}
To avoid confusing notation, we will write $\FF$ for $\add(T)$.  Consider the full subcategory $\UU$ of $\CC$ such that $(\TT,\UU)$ is an $\FF$-mutation pair, thus $\UU = \mu_\FF(\TT)$), or put differently, $\UU$ is the full subcategory spanned by all $U \in \CC$ such that there is a triangle
\begin{equation}\label{equation:Ztriangle}
T' \to F \to U \to T'[1]
\end{equation}
where $T' \in \TT, F \in \FF$ and the map $T' \to F$ is a left $\FF$-approximation of $T'$.  We wish to show that $\ZZ = \add(\ind \UU \setminus \ind \FF)$ is the category in the statement of the lemma.  Note that $(\ind \ZZ) \cap (\ind \TT) = \emptyset$.

It follows from \cite[Theorem 5.1]{IyamaYoshino08} that $\UU$ is a cluster-tilting subcategory of $\CC$, so that $\Hom(\UU, \UU[1]) = 0$ and hence $\Hom(\ZZ, \ZZ[1]) = 0$.  Furthermore, Proposition \ref{proposition:JorgensenPaluCofinite} yields that $\ZZ$ is a functorially finite subcategory of $\CC$.  We may thus consider the Iyama-Yoshino reduction $\CC_{[\ZZ]}$ of $\CC$ with respect to $\ZZ$.

We will show that no nonzero morphism in $\FF$ factors through $\ZZ$; this would then establish the second statement of the lemma.  Let $f:F_1 \to F_2$ be a morphism in $\FF$ which factors through $\ZZ$.  By Proposition \ref{proposition:AbelianQuotient}, we know that $\CC / (\TT[1]) \cong \mod \TT$ (the category of projectives in $\CC / (\TT[1])$ is $\TT$), and by Theorem \ref{theorem:Trianglefree}, we know this category is hereditary.  The morphism $f:F_1 \to F_2$ is hence a morphism between projective objects.  Since $(\ind \ZZ) \cap (\ind \TT) = \emptyset$, no (nonzero) object in $\ZZ$ is projective in $\CC / \t \TT$ and we know that $f$ can only factor through an object of $\ZZ$ if $f = 0$.  This establishes the second statement of the lemma.

By \cite[Theorem 4.9]{IyamaYoshino08}, we know that $\FF \subseteq \ZZ^{\perp_1}$ will correspond to a cluster-tilting subcategory of $\CC_{[\ZZ]}$.  Since $\FF \subseteq \TT$ and $\ind \FF$ is finite, this proves the third statement.

We now turn our attention to the first part of the lemma.  Thus, let $X \in \DD_T$; we want to show that $X$ lives in $\CC_{[\ZZ]}$.  We will write $Y = X \oplus X[1]$.

First, we prove the claim that $Y \in \ZZ^{\perp_1}$.  Let $Z \in \ZZ$.  By applying $\Hom(Z,-)$ to the right $\TT$-approximation triangle for $Y$, we find the exact sequence
$$(Z, T_0[1]) \to (Z,Y[1]) \to (Z,T_1[2]).$$
Note that $(Z,T_0[1]) = 0$ since $Z, T_0 \in \UU$.  Thus, to show that $\Hom(Z,Y[1]) = 0$, it suffices to show that $(Z,T_1[2]) = 0$.

We apply $\Hom(Z,-)$ to the right $\TT$-approximation triangle of $T_1[2]$ to find an exact sequence
$$(Z,T'_0) \to (Z, T_1[2]) \to (Z,T'_1[1]).$$
Since $Z, T'_1 \in \UU$, we know that $(Z,T'_1[1])=0$.

By Proposition \ref{proposition:AbelianQuotient}, we know that the category $\CC / [\tau \TT] \cong \mod \TT$ and hence is hereditary by Theorem~\ref{theorem:Trianglefree}.  Moreover, the projective objects are given by $\TT$.  Since $\Hom(-,Z)|_\TT$ has no projective summands in $\mod\TT$ (given that $\ind \ZZ \cap \ind \TT = \emptyset$), we know that $(\CC / [\tau \TT])(Z, T'_0) = 0$.  Thus, every map in $\CC(Z,T'_0)$ factors though $\tau \TT$.  Consider the following diagram where the rows are triangles and where the map $Z \to T_1[2]$ is any map
\[\xymatrix{
T \ar[r] & F_T  \ar[r] & Z \ar[r] \ar[d] & T[1] \\
T'_1 \ar[r] & T'_0  \ar[r] & T_1[2] \ar[r] & T'_1[1]}\]
Since $Z, T'_1 \in \UU$, we know that the composition $Z \to T_1[2] \to T'_1[1]$ is zero.  Hence, the map $Z \to T_1[2]$ factors as $Z \to T'_0 \to T_1[2]$.  We have already established that every map $Z \to T'_0$ factors through $\tau \TT$ so that the composition $F_T \to Z \to T_0$ is zero.  We see that the map $Z \to T'_0$ factors as $Z \to T[1] \to T'_0$.

Summarizing, the map $Z \to T_1[2]$ factors as $Z \to T[1] \to T'_0 \to T_1[2]$.  But since $\Hom(T[1], T_1[2]) = 0$, we see that the original map $Z \to T_1[2]$ is zero as well.  We have shown that $\Hom(Z,T_1[2]) = 0$ and hence that $\Hom(Z,Y[1]) = 0$, for all $Z \in \ZZ$.

It follows from \cite[Lemma 4.8]{IyamaYoshino08} that the passage $\ZZ^{\perp_1} \to \ZZ^{\perp_1} / [\ZZ]$ preserves $\Ext^1(X,\t X) \cong \Hom(X,X)^*$.  This shows that $X \in \CC$ lives in $\ZZ^{\perp_1} / [\ZZ]$.

When $\CC$ is algebraic, it follows from \cite[Proposition 4.3]{AmiotOppermann12} that $\ZZ^{\perp_1} / [\ZZ]$ is algebraic, so that \cite{KellerReiten08} implies that $\ZZ^{\perp_1} / [\ZZ] \cong \CC_{\End T}$.  This proves the final statement.
\end{proof}

\begin{example}
Let $\CC$ be the 2-Calabi-Yau Krull-Schmidt category with cluster-tilting subcategory $\TT$ from Example \ref{example:ClusterD}.  It follows from Theorem \ref{theorem:Covering} that every object $X \in \CC$ lives in a subquotient that is equivalent to a cluster category of type $D$.  
\end{example}

\subsection{On the Bongartz complement}

As usual, let $\CC$ be a Krull-Schmidt 2-Calabi-Yau category with a directed cluster-tilting subcategory $\TT$.
In this section let $\WW$ be any fixed functorially finite rigid subcategory of $\CC$.  By Theorem \ref{theorem:BongartzComplement}, we know that there is a cluster-tilting subcategory $\RR$ of $\CC$ which contains $\WW$.  The main result in this section is that we can choose $\RR$ such that $\RR$ is directed in $\CC_{[\WW]}$ (see Theorem \ref{theorem:HereditaryBongartzComplement} below).

We follow the notation from Appendix \ref{section:Bongartz}; in particular $\FF$ and $\RR$ are as in Construction \ref{construction:Bongartz}.  That is, for each $T \in \TT$ we consider a right $\WW$-approximation $f\colon W \to T[1]$ and the triangle
\begin{equation}\label{equation:Bongartz}
\xymatrix@1{
T \ar[r]^t & F_T \ar[r] & W \ar[r]^{f} & T[1].
}
\end{equation}
We then let $\FF = \add \{F_T\}_{T \in \TT}$ and $\RR = \add (\FF \cup \WW)$, and we will in fact show that this particular $\RR$ is directed in $\CC_{[\WW]}$.

\begin{remark}
The proof presented here follows \cite[\S 4.2]{Jasso13}.
\end{remark}

We consider the functor
\begin{align*}
H_\TT: \CC &\to \mod \TT \\
C & \mapsto \Hom(-,C)|_\TT.
\end{align*}
Recall from Proposition \ref{proposition:AbelianQuotient} that $H_\TT$ induces an equivalence $\CC / [\TT[1]] \to \mod \TT$. {To simplify the notation, we often write $\overline{C}$ instead of $H_\TT(C)$ for the image of $C \in \CC$ in $\CC / [\TT[1]] \cong \mod \TT$. Similarly if $\RR'\subseteq \CC$, we write $\overline{\RR'}$ instead of $H_\TT(\RR')$ for the essential image of $H_\TT|_{\RR'}$ Furthermore, we recall (see again Proposition~\ref{proposition:AbelianQuotient}) that for every indecomposable $C \in \CC$ with no indecomposable summands in $\TT[1]$ we have $H_\TT(C[1]) \cong \tau H_\TT(C)$.}

Let $\GG$ be a full additive subcategory of an abelian category $\AA$.  We will write $\Fac_\AA \GG$ for the full subcategory of $\AA$ whose objects are quotients of objects of $\GG$.  We will write $\Fac \GG$ if the ambient category $\AA$ is clear.

\begin{lemma}\label{lemma:TrianglefreeReductions}
Let $\RR'$ be any functorially finite and rigid subcategory of $\CC$.
\begin{enumerate}
\item\label{enumerate:TrianglefreeReductions1} {$\overline{\RR'}$} is a functorially finite and rigid subcategory of $\mod \TT$,
\item\label{enumerate:TrianglefreeReductions2} the category {$\Fac_{\mod \TT} \overline{\RR'}$} is a torsion class in $\mod \TT$,
\item\label{enumerate:TrianglefreeReductions3} {if $\RR'$ is a cluster-tilting subcategory of $\CC$, then $\overline{\RR'}$ is precisely the category of Ext-projective objects in $\Fac_{\mod \TT} \overline{\RR'}$.}
\end{enumerate}
\end{lemma}

{
\begin{proof}
\begin{enumerate}
\item Let $M \in \mod \TT$, and consider a lift $\tilde{M} \in \CC$ (thus, $H_\TT(\tilde{M}) \cong M$).  Since $\RR'$ is functorially finite in $\CC$, there is a right $\RR'$-approximation $R \to \tilde{M}$.  It is now easy to check that $\overline{R} \to M$ is a right $\overline{\RR'}$-approximation of $M$.  The existence of left $\overline{\RR'}$-approximations can be checked in a similar way.
The rigidity of $\overline{\RR'}$ follows from \cite[Theorem 4.9]{KonigZhu08} (or from Lemma~\ref{lemma:Ext-via-H}).

\item Clearly $\Fac \overline{\RR'}$ is closed under quotient objects.  Since $\mod \TT$ is hereditary, we know that $\Ext^1_\TT(-,-)$ is right exact.  From $\overline{\RR'}$ being rigid, we infer that $\Ext^1(\overline{\RR'}, \Fac \overline{\RR'}) = 0$.  Using this, it is easy to see that $\Fac \overline{\RR'}$ is closed under extensions in $\mod\TT$.

To show that $\Fac \overline{\RR'}$ is a torsion class, we need to show that the embedding $\Fac \overline{\RR'} \to \mod \TT$ has a right adjoint $t\colon \mod \TT \to \Fac \overline{\RR'}$.  Since $\Fac \overline{\RR'}$ is closed under quotient objects, it suffices to show that each $X \in \mod \TT$ has a largest subobject lying in $\Fac \overline{\RR'}$.  Indeed, this subobject is then $t(X)$.  Here, $t(X)$ is the image of a right $\overline{\RR'}$-approximation $f:R \to X$.

\item Since $\Ext^1_\TT(\overline{\RR'}, \Fac \overline{\RR'}) = 0$ by the proof of part (\ref{enumerate:TrianglefreeReductions2}), every object in $\overline{\RR'}$ is Ext-projective in $\overline{\RR'}$.

Consider conversely an Ext-projective $P$ of $\Fac \overline{\RR'}$, and let $\tilde{P}$ be a lift of $P$ in $\CC$ (thus, $\tilde{P} \in \CC$ and $H_\TT(\tilde{P}) \cong P$); we choose a lift $\tilde{P}$ which has no indecomposable summands in $\TT[1]$.  We will show that $\tilde{P} \in \RR'$ by showing that $\Ext^1_\CC(\tilde{P},\RR') = 0.$

First, consider an object $R \in \RR'$ such that $R$ has no direct summands in $\TT[1]$.  As $P$ is Ext-projective in $\Fac \overline{\RR}$, we have $\Ext^1_\TT (P,\overline{R}) = 0$.  We obtain from Proposition~\ref{proposition:AbelianQuotient}(\ref{enumerate:AbelianQuotient3}) and Lemma~\ref{lemma:Ext-via-H} that $\Hom_\TT(\overline{R},\tau P) = \Hom_{\CC/[\TT[1]]}(R,\tilde{P}[1]) = 0$ and $[\TT[1]](R,\tilde{P}[1]) = 0$, respectively. Hence, $\Ext^1_\CC(R,\tilde{P}) = 0$, as required.

Next, consider an object $R \in \RR' \cap \TT[1]$.  In this case, $\overline{R[1]}$ is injective in $\CC / [\TT[1]]$.  As $P \in \Fac \overline{\RR'}$ and $\Hom_\TT(\overline{\RR'}, \overline{R[1]}) \cong \Hom_{\CC / \TT[1]} (\RR', R[1])= 0$, we find that $\Hom_\TT(P, \overline{R[1]}) = 0$.  Since $R \in \TT[1]$, we find that $[\TT[1]](-, R[1]) = 0$ and thus $[\TT[1]](\tilde{P}, R[1]) = 0$.  This implies that $\Hom_\CC(\tilde{P}, R[1]) = 0$.

We have shown that $\Ext^1_\CC(\tilde{P}, \RR') = 0$ and hence $\tilde{P} \in \RR'$.  This finishes the proof.
\qedhere
\end{enumerate}
\end{proof}

Turning back to our functorially finite rigid subcategory $\WW \subseteq \CC$, we know that there is a torsion pair $(\Fac\overline{\WW}, \overline{\WW}^{\perp_0})$. We will denote by $t_\WW\colon \mod \TT \to \Fac \overline{\WW}$ the corresponding torsion radical.

\begin{lemma}\label{lemma:EndomorphismVSQuotient}
Let $\WW$ and $\FF$ be functorially finite and rigid subcategories of $\CC$.  Assume that $\RR' = \add(\WW \cup \FF)$ is also rigid.
Then, for any $R \in \RR'$, there is an algebra-isomorphism
\[{\frac{\Hom_\TT(\overline{R},\overline{R})}{[\overline{\WW}](\overline{R},\overline{R})} \cong \Hom_\TT(\overline{R}/t_\WW(\overline{R}), \overline{R}/t_\WW(\overline{R})).}\]
\end{lemma}

\begin{proof}
\begin{enumerate}
\item For any $e \in \End_\TT(\overline{R})$, there is a commutative diagram
\[
\xymatrix{
0 \ar[r] & t_\WW(\overline{R}) \ar[r] \ar@{-->}[d]^{t_\WW(e)} & \overline{R} \ar[r] \ar[d]^{e} & \overline{R}/t_\WW(\overline{R}) \ar@{-->}[d]^{e'} \ar[r] & 0 \\
0 \ar[r] & t_\WW(\overline{R}) \ar[r] & \overline{R} \ar[r] & \overline{R}/t_\WW(\overline{R}) \ar[r]  & 0
} \]
The assignment $e \mapsto e'$ defines an algebra-homomorphism $\varphi\colon \End_\TT(\overline{R}) \to \End_\TT(\overline{R}/t_\WW(\overline{R}))$, whose kernel is given by exactly those $e \in \End_\TT(\overline{R})$ which factor through $\Fac \overline{\WW}$.  As $\overline{R}$ is Ext-projective in $\Fac \overline{\RR'}$ (see Lemma \ref{lemma:TrianglefreeReductions}) and $\Fac \overline{\WW} \subseteq \Fac \overline{\RR'}$, we find that $e$ factors through an object in $\Fac \overline{\WW}$ if and only if it factors through an object in $\overline{\WW}$.  This shows that the kernel of $\varphi\colon \End_\TT(\overline{R}) \to \End_\TT(\overline{R}/t_\WW(\overline{R}))$ is $[\overline{\WW}](\overline{R},\overline{R})$.

To show that $\varphi$ is an epimorphism, we apply the functor $\Hom_\TT(\overline{R},-)$ to the short exact sequence $0 \to t_\WW(\overline{R}) \to \overline{R} \stackrel{p}{\rightarrow} \overline{R}/t_\WW(\overline{R}) \to 0.$  Using the fact that $\Ext^1_\TT(\overline{R}, t_\WW(\overline{R})) = 0$ (since $\Ext^1_\TT(\overline{\RR'}, \Fac \overline{\RR'}) = 0$), we see that any morphism $\overline{R} \to \overline{R}/t_\WW(\overline{R})$ lifts through the projection $p\colon \overline{R} \to \overline{R}/t_\WW(\overline{R}).$  Thus, any $e' \in \End_\TT(\overline{R}/t_\WW(\overline{R}))$ gives rise to a commutative square:
\[\xymatrix{
0 \ar[r] & t_\WW(\overline{R}) \ar[r] & \overline{R} \ar[r]^-{p} \ar@{-->}[d] & \overline{R}/t_\WW(\overline{R}) \ar[d]^{e'} \ar[r] & 0 \\
0 \ar[r] & t_\WW(\overline{R}) \ar[r] & \overline{R} \ar[r]^-{p} & \overline{R}/t_\WW(\overline{R}) \ar[r]  & 0
} \]
From this follows easily that $\varphi$ is surjective.  Hence, we have shown that $\varphi$ induces an algebra-isomorphism $\End_\TT(\overline{R})/[\overline{\WW}](\overline{R},\overline{R}) \to \End_\TT(\overline{R}/t_\WW(\overline{R}))$, as required.
\qedhere
\end{enumerate}
\end{proof}

Aside from $(\Fac\overline{\WW}, \overline{\WW}^{\perp_0})$, Lemma~\ref{lemma:TrianglefreeReductions}(\ref{enumerate:TrianglefreeReductions2}) also provides us with the torsion pair $(\Fac\overline{\RR}, \overline{\RR}^{\perp_0})$ coming from the Bongartz completion $\RR$ of $\WW$.  We will be particularly interested in the intersection $\AA = \Fac\overline{\RR} \cap \overline{\WW}^{\perp_0}$ of the torsion and the torsion-free class from the respective pairs.  This category can be viewed as a generalized version of the $\tau$-rigid reduction from~\cite{Jasso13} where we also allow objects from $\TT[1]$ in $\WW$.

\begin{lemma} \label{lemma:TauRigidReduction}
Let $\TT \subseteq \CC$ be a directed cluster-tilting subcategory, $\WW \subseteq \CC$ be a functorially finite rigid subcategory and $\RR$ be the Bongartz completion of $\WW$ by $\TT$.  Then $\AA = \Fac\overline{\RR} \cap \overline{\WW}^{\perp_0}$ is closed under kernels, cokernels and extensions in $\mod\TT$.  In particular, $\AA$ is a hereditary abelian category itself.
\end{lemma}

\begin{proof}
Let $\WW' = \add(\ind\WW \setminus \ind\TT[1])$	and $\WW'' = \WW \cap \TT[1]$.  In particular $\overline{\WW''[-1]}$ consists of projective $\TT$-modules.  In fact, the pair $(\overline{\WW'}, \overline{\WW''[-1]})$ of subcategories of $\mod\TT$ is essentially nothing else than a support $\tau$-tilting pair in the sense of~\cite[Definition 0.3]{AdachiIyamaReiten14}.

Note that it suffices to prove that if we let $\QQ = \overline{\WW'} \cup \overline{\WW''[-1]}$, then we have
\begin{equation} \label{equation:Perpendicular}
\AA = \Ker\Hom_\TT(\QQ,-) \cap \Ker\Ext^1_\TT(\QQ,-) \quad (\subseteq \mod\TT).
\end{equation}
Indeed, the closure properties of $\AA$ then follow at once from \cite[Proposition 1.1]{GeigleLenzing91} since $\AA$ is then right perpendicular to $\QQ$ in the terminology of~\cite{GeigleLenzing91}.

Thus, let us focus on equality~\eqref{equation:Perpendicular}. Fix $X \in \mod\TT$ and let $\tilde{X}$ be a lift of $X$ to $\CC$, i.e.\ $X = H_\TT(\tilde{X})$. Moreover, we may choose $\tilde{X}$ such that it has no summand in $\TT[1]$.  Further, observe that if we invoke Lemma~\ref{lemma:Ext-via-H} and the Yoneda lemma, the vanishing conditions
\[ \Ext^1_\TT(\overline{\WW'}, X) = 0 = \Hom_\TT(\overline{\WW''[-1]}, X) \]
translate to $[\TT](\WW'[-1],\tilde{X}) = 0 = \Hom_\CC(\WW''[-1],\tilde{X})$, which can be equivalently restated as $[\TT](\WW'[-1],\tilde{X}) = 0$.

Suppose now that we chose $X \in \overline{\RR}$. Since $\RR$ is rigid, we clearly have $[\TT](\WW'[-1],\tilde{X}) = 0$. Thus,
\[ \Ext^1_\TT(\overline{\WW'},X) = 0 = \Hom_\TT(\overline{\WW''[-1]},X). \] 
and, since $\Ext^1_\TT(-,-)$ is right exact, the same equalities hold in fact for any $X \in \AA$.  Since by definition also $\Hom_\TT(\overline{\WW'},X) = 0$ for each $X \in \AA$, we have just proven the inclusion $\AA \subseteq \Ker\Hom_\TT(\QQ,-) \cap \Ker\Ext^1_\TT(\QQ,-)$.

Suppose conversely that $\Hom_\TT(\QQ,X) = 0 = \Ext^1_\TT(\QQ,X)$, so that in particular $[\TT](\WW'[-1],\tilde{X}) = 0$. If $T_1 \to T_0 \stackrel{p}\to \tilde{X} \to T_1[1]$ is the $\TT$-approximation triangle and $W[-1] \stackrel{f[-1]}\to T_0 \to F_{T_0} \to W$ is the rotation of triangle~\eqref{equation:Bongartz} then $p \circ f[-1]$ vanishes by the assumption.  Hence $p$ factors through a morphism $q\colon F_{T_0} \to \tilde{X}$.  Since $H_\TT(p)$ is an epimorphism in $\mod\TT \cong \CC/[\TT[1]]$ (see \cite[Theorem 2.3]{KonigZhu08}), so is $H_\TT(q)$ and so $X \in \Fac\overline{\RR}$.  Combining this with the assumption that $\Hom_\TT(\overline{\WW'},X) = 0$ (since $\Hom_\TT(\QQ,X) = 0$), we infer that $X \in \AA$.  Therefore $\Ker\Hom_\TT(\QQ,-) \cap \Ker\Ext^1_\TT(\QQ,-) \subseteq \AA$ holds as well.
\end{proof}

The following lemma allows us to relate objects of the Bongartz complement $\FF$ to projective objects in $\AA$.

\begin{lemma}\label{lemma:TorsionfreeIsProjective}
Let $\WW \subseteq \CC$ be a functorially finite rigid subcategory and let $\FF$ and $\RR = \add(\WW \cup \FF)$ be as above.  Further, let $\AA = \Fac\overline{\RR} \cap \overline{\WW}^{\perp_0}$.  Then for each $F \in \overline{\FF}$, the object $F/t_\WW(F)$ is a projective object in the abelian category $\AA$.
\end{lemma}

\begin{proof}
Note that $F/t_\WW(F) \in \AA$, directly from the definitions. Let $X \in \AA$ be another object. The long exact sequence coming from the application of $\Hom_\TT(-,X)$ to the short exact sequence $0 \to t_\WW(F) \to F \to F/t_\WW(F) \to 0$ shows that $\Ext^1_\AA(F/t_\WW(F),X) = \Ext^1_\TT(F/t_\WW(F),X) = 0$ if and only if $\Ext^1_\TT(F,X) = 0$. However, $\Ext^1_\TT(F,X)$ vanishes by Lemma~\ref{lemma:TrianglefreeReductions}(\ref{enumerate:TrianglefreeReductions3}) since $F \in \overline{\RR}$ and $X \in \AA \subseteq \Fac\overline{\RR}$.
\end{proof}

Now we are in a position to state and prove the second main result of this section.

\begin{theorem}\label{theorem:HereditaryBongartzComplement}
Let $\CC$ be a 2-Calabi-Yau triangulated category which admits a directed cluster-tilting subcategory.  Let $\WW$ be a functorially finite rigid subcategory of $\CC$.  The reduction $\CC_{[\WW]}$ also admits a directed cluster-tilting subcategory.
\end{theorem}

\begin{proof}
Let $\TT$ be a directed cluster-tilting subcategory of $\CC$.  We want to construct a directed cluster-tilting subcategory in $\CC_{[\WW]}$.  To this end, let $\RR = \add(\FF \cup \WW)$ be the Bongartz completion of $\WW$ by $\TT$, as in Construction \ref{construction:Bongartz}. By Proposition~\ref{proposition:Bongartz}, we know that $\RR$ is a cluster-tilting subcategory of $\CC$, and by \cite[Theorem 4.9(1)]{IyamaYoshino08}, we know that $\FF$ becomes a cluster-tilting subcategory of $\CC_{[\WW]}$.  We will show that $\FF \subseteq \CC$ induces in fact a directed cluster-tilting subcategory of $\CC_{[\WW]}$.  Put differently, we want to show that $\End_{\CC_{[\WW]}}(F)$ is a hereditary algebra, for any $F \in \FF$ (see Proposition \ref{proposition:ShIsLocal}).

It follows from Theorem~\ref{theorem:Trianglefree} that $\CC/{[\TT[1]]} \cong \mod \TT$ is hereditary.  As before, consider the torsion pairs $(\Fac\overline{\WW}, \overline{\WW}^{\perp_0})$ and $(\Fac\overline{\RR}, \overline{\RR}^{\perp_0})$ in $\mod \TT$ and the hereditary abelian subcategory $\AA = \Fac\overline{\RR} \cap \overline{\WW}^{\perp_0}$ of $\mod \TT$. We again write $t_\WW\colon \mod \TT \to \Fac \WW$ for the torsion radical for $\Fac\overline{\WW}$.

Lemma \ref{lemma:EndomorphismVSQuotient} yields that
\[\End_{\CC / [\TT[1]]}(\overline{F}/t_W(\overline{F})) \cong \frac{\Hom_{\CC / [\TT[1]]}(\overline{F},\overline{F})}{[\overline{\WW}](\overline{F},\overline{F})}.\]
Lemma \ref{lemma:TorsionfreeIsProjective} implies that this algebra is hereditary, as it is the endomorphism ring of the projective object $\overline{F}/t_\WW(\overline{F})$ in the hereditary category $\AA$ (see Proposition \ref{proposition:CategoryOfProjectivesIsSemiHereditary}).  We will now prove that $\End_{\CC / [\TT[1]]}(F/t_W(F)) \cong \Hom_{\CC_{[\WW]}}(F,F)$. 

For this, consider any $T \in \TT$ and the associated triangle $T \to F_T \to W \to T[1].$
From this triangle follows that any morphism $e \in \End_\CC(F_T)$ that factors through $\TT[1]$ also factors through $\WW$, so that $[\TT[1]](F_T,F_T) \subseteq [\WW](F_T,F_T)$.  We find:
\begin{align*}
\frac{\Hom_{\CC / [\TT[1]]}(\overline{F},\overline{F})}{[H_\TT(\WW)](\overline{F},\overline{F})} &\cong
\frac{\frac{\Hom_{\CC}(F,F)}{[\TT[1](F,F)]}}{\frac{[\WW](F,F)}{[\TT[1]](F,F)}} \\
&\cong \frac{\Hom_{\CC}(F,F)}{[\WW](F,F)}\\
&\cong \Hom_{\CC_{[\WW]}}(F,F).
\end{align*}
Thus, $\End_{\CC_{[\WW]}}(F)$ is a hereditary algebra, as required.
\end{proof}

We have the following corollary to Theorem~\ref{theorem:HereditaryBongartzComplement}.

\begin{corollary}\label{corollary:BongartzQuiver}
Let $\CC$ be an algebraic 2-Calabi-Yau category with a directed cluster-tilting subcategory. Let $\UU$ be an arbitrary (not necessarily directed) cluster-tilting subcategory of $\CC$, and let $\WW$ be a full additive Krull-Schmidt subcategory of $\UU$ such that $\lvert\ind \UU \setminus \ind \WW\rvert < \infty$.  The reduction $\CC_{[\WW]}$ is of the form $\CC_Q$, for some quiver $Q$.
\end{corollary}

\begin{proof}
It follows from Proposition \ref{proposition:JorgensenPaluCofinite} that $\WW$ is functorially finite in $\CC$, so that we can indeed consider the reduction $\CC_{[\WW]}$, and by Theorem~\ref{theorem:HereditaryBongartzComplement}, we know that $\CC_{[\WW]}$ has a directed cluster-tilting subcategory.  By \cite[Theorem 2.4]{DehyKeller08}, we know that, for a cluster-tilting subcategory $\VV$ of $\CC_{[\VV]}$, we have $\vert\ind \VV\rvert = \lvert\ind \UU \setminus \ind \WW\rvert < \infty,$ hence $\VV = \add V$, for some (cluster-tilting) object $V \in \CC_{[\WW]}.$  The statement now follows from the characterization in \cite{KellerReiten08}. 
\end{proof}
}

\section{Cluster structures}\label{section:ClusterStructures}

Let $\CC$ be a 2-Calabi-Yau category with a directed cluster-tilting subcategory $\TT \subseteq \CC$.  We will show in Theorem \ref{theorem:ClusterStructure} that $\TT$ has a cluster structure in the sense of \cite{BuanIyamaReitenScott09}.  We start by recalling the definition of a cluster structure.

\begin{definition}\label{definition:ClusterStructure}
Let $\CC$ be a Hom-finite Krull-Schmidt triangulated category.  Assume that we have a collection of subsets of $\ind \CC$ called \emph{clusters}.  Any element of $\ind \CC$ which occurs in a cluster is called a \emph{cluster variable}.

We say that the clusters give a \emph{cluster structure} on $\CC$ if the following conditions hold.
\begin{enumerate}
\item\label{enumerate:ClusterStructure1} Let $T \subseteq \ind \CC$ be a cluster.  Then, for each $M \in T$, there is a unique cluster variable $M^* \in \ind \CC$, nonisomorphic to $M$, such that $(T \setminus \{M\}) \cup \{M^*\}$ is a cluster.
\item\label{enumerate:ClusterStructure2} Let $T \subseteq \ind \CC$ be a cluster.  For each $M,M^*$ as above, there are triangles
$$M \stackrel{f}{\rightarrow} B \stackrel{g}{\rightarrow} M^* \mbox{ and } M^* \stackrel{s}{\rightarrow} B' \stackrel{t}{\rightarrow} M$$
where $f$ and $s$ are minimal left $(T \setminus \{M\})$-approximations, and $g$ and $t$ are minimal right $(T \setminus \{M\})$-approximations.
\item\label{enumerate:ClusterStructure3} The Auslander-Reiten quiver of $T$ has no loops and no 2-cycles.
\item\label{enumerate:ClusterStructure4} Let $M$ be a cluster variable in a cluster $T$, and denote by $T^*$ the unique cluster obtained by replacing $M$ with $M^*$.  Then the Auslander-Reiten quiver of $T^*$ is obtained by a Fomin-Zelevinski mutation at the vertex corresponding to $M$ in the Auslander-Reiten quiver of~$T$.
\end{enumerate}

If (\ref{enumerate:ClusterStructure1}) and (\ref{enumerate:ClusterStructure2}) hold, we say that the collection of clusters gives a \emph{weak cluster structure} on $\CC$.
\end{definition}

\begin{remark}
The definition given in \cite{BuanIyamaReitenScott09} is more general than the one we give here, as we do not include any cluster coefficients.
\end{remark}

We recall the following theorem (see \cite[Theorem II.2.1]{BuanIyamaReitenScott09}).

\begin{theorem}\label{theorem:BuanIyamaReitenScott}
Let $\CC$ be a triangulated 2-CY category having a cluster-tilting subcategory.
\begin{enumerate}
  \item The cluster-tilting subcategories determine a weak cluster structure on $\CC$.
  \item If there are no loops or 2-cycles in the Auslander-Reiten quivers of the cluster-tilting subcategories, then they determine a cluster structure on $\CC$.
\end{enumerate}
\end{theorem}

\begin{example}\label{example:ClusterStructureFinite}
For a finite acyclic quiver $Q$, the category $\CC_Q$ has a cluster structure where the clusters are given by the cluster-tilting objects.  Indeed, it has been shown in \cite{BuanMarshReinekeReitenTodorov06} that $\CC_Q$ has a weak cluster structure and in \cite[Proposition 3.2]{BuanMarshReiten08} that the quivers of the cluster-tilting objects have no loops or 2-cycles.
\end{example}

\begin{lemma}\label{lemma:QuotientsOfRadicalMorphisms}
Let $\WW \subset \UU$ and let $f \in \Hom_\UU(X,Y)$.  If either $X$ or $Y$ lies in $\ind (\UU \setminus \WW)$, then $f \in \rad^1_{\UU}(X,Y) \Leftrightarrow \overline{f} \in \rad^1_{\UU / [\WW]}(X,Y)$.
\end{lemma}

\begin{proof}
We will assume that $X \in \ind (\UU \setminus \WW)$.  The case where $Y \in \ind (\UU \setminus \WW)$ is similar.

Let $f \in \rad^1_{\UU}(X,Y)$.  This means that for every $Z \in \ind \UU$ and any $a \in \Hom_\UU(Z,X), b \in \Hom_\UU(Y,Z)$, we have that ${b \circ f \circ a} \in \Hom_{\UU}(Z,Z)$ is not invertible, hence $b \circ f \circ a$ is nilpotent (since $\Hom_{\UU}(Z,Z)$ is an artinian local ring).  This implies that $\overline{b \circ f \circ a}$ is nilpotent in $\Hom_{\UU/[\WW]}(Z,Z) = \Hom_{\UU}(Z,Z) / [\WW](Z,Z)$, so that we may conclude that $\overline{f} \in \rad^1_{\UU / [\WW]}(X,Y)$.

For the other direction, we assume that $\overline{f} \in \rad^1_{\UU / [\WW]}(X,Y)$.  We want to show that $f \in \rad^1_\UU(X,Y)$, or thus that for any $Z \in \ind \UU$ and $a \in \Hom_\UU(Z,X), b \in \Hom_\UU(Y,Z)$, the composition $b \circ f \circ a \in \Hom_\UU(Z,Z)$ is not invertible.

If $Z \in \ind \WW$, then $a: Z \to X$ is not an isomorphism, and thus neither is the composition $b \circ f \circ a$.  We therefore need only consider the case where $Z \not\in \ind \WW$.  Since $\overline{f} \in \rad^1_{\UU / [\WW]}(X,Y)$, we know that $\overline{b \circ f \circ a}$ is nilpotent, or thus that ${(b \circ f \circ a)}^n \in [\WW](Z,Z)$ for some $n \in \bN$.  However, this implies that ${(b \circ f \circ a)}^n$ is not an isomorphism, and hence neither is $b \circ f \circ a$.
\end{proof}

\begin{theorem}\label{theorem:ClusterStructure}
Let $\CC$ be an algebraic Krull-Schmidt 2-Calabi-Yau triangulated category.  If $\CC$ has a directed cluster-tilting subcategory, then the cluster-tilting subcategories give a cluster structure on $\CC$.
\end{theorem}

\begin{proof}
By Theorem \ref{theorem:BuanIyamaReitenScott}, we need only show that the Auslander-Reiten quivers of the cluster-tilting subcategories of $\CC$ have no loops or 2-cycles.  Let $\UU$ be a cluster-tilting subcategory of $\CC$.  Seeking a contradiction, assume that $\UU$ has a 2-cycle $X \rightleftarrows Y$; the case where $\UU$ has a loop can be handled in a similar way.

Let $\WW = \add(\ind \UU \setminus \{X,Y\})$.  Note that $\add(X \oplus Y)$ is a cluster-tilting object in $\CC_{[\WW]}$ by \cite[Theorem 4.9]{IyamaYoshino08}.  

Let $f \in \Irr_\UU(X,Y) = \rad_\UU^1(X,Y) \setminus \rad_\UU^2(X,Y)$.  Lemma \ref{lemma:QuotientsOfRadicalMorphisms} yields that $\overline{f} \in \Irr_{\UU/[\WW]}(X,Y) = \rad_{\UU/[\WW]}^1(X,Y) \setminus \rad_{\UU/[\WW]}^2(X,Y)$.  Hence, there is an arrow $X \to Y$ in the Auslander-Reiten quiver of $\UU / [\WW]$.  Similarly, we find an arrow $Y \to X$ in the Auslander-Reiten quiver of $\UU / [\WW]$.

However, it follows from Corollary \ref{corollary:BongartzQuiver} that $\CC_{[\WW]}$ is of the form $\CC_Q$, for a (finite) quiver $Q$, and it follows from \cite[Proposition 3.2]{BuanMarshReiten08} that the cluster-tilting object $X \oplus Y$ of $\CC_Q$ has no 2-cycle.  We have obtained the required contradiction.
\end{proof}

As a corollary, we obtain a positive answer to a conjecture in \cite{LiuPaquette15}.

\begin{corollary}
Let $Q$ be a strongly locally finite thread quiver.  The cluster-tilting subcategories give a cluster structure on $\CC_Q = \Db (\mod kQ) / (\bS \circ [-2])$.
\end{corollary}

\begin{proof}
This follows from Theorems \ref{theorem:AllThreadQuiversOccur} and \ref{theorem:ClusterStructure}.
\end{proof}

\begin{definition}\label{definition:Reachability}
We say that a cluster-tilting subcategory $\VV$ is \emph{reachable by mutation} (or just \emph{reachable}) from $\UU$ if there is a finite sequence of cluster-tilting subcategories $\UU = \UU_0, \UU_1, \ldots, \UU_n = \VV$ such that $\UU_i \cap \UU_{i-1}$ are almost cluster-tilting subcategories for $1 \leq i \leq n$ {(thus $|\ind \UU_i \setminus \ind \UU_{i-1}|=1$)}.  We say that a rigid object $E \in \CC$ is \emph{reachable by mutation} (or just \emph{reachable}) from $\UU$ if there is a cluster-tilting subcategory $\VV$, reachable from $\UU$, with $E \in \VV$.
\end{definition}

\begin{remark}
{A cluster-tilting object $\UU$ is reachable from a cluster-tilting object $\VV$ if and only if $\UU$ and $\VV$ are in the same connected component of the exchange graph of $\CC$ (see Definition \ref{definition:ExchangeGraph}).}
\end{remark}

\begin{corollary}\label{corollary:Bongartz}
Let $\CC$ be an algebraic 2-Calabi-Yau category with a cluster-tilting subcategory $\UU \subseteq \CC$.  Assume that $\CC$ admits a directed cluster-tilting subcategory.
\begin{enumerate}
\item Let $X \in \CC$ be a rigid object such that $\Ext^1(U,X) = 0$ for all but finitely many $U \in \ind \UU$.  There is a cluster-tilting subcategory $\VV$ of $\CC$, reachable from $\UU$, with $X \in \VV$.
\item A cluster-tilting subcategory $\VV$ is reachable from $\UU$ if and only if $\UU$ and $\VV$ differ at only finitely many indecomposable objects.  In this case $|(\ind \UU) \setminus (\ind \VV)| = |(\ind \VV) \setminus (\ind \UU)|$.
\end{enumerate}
\end{corollary}

\begin{proof}
\begin{enumerate}
\item Let $\WW$ be the full subcategory of $\UU$, consisting of all objects $W \in \UU$ for which $\Ext^1(W,X) = 0$.  Note that $X \in \WW^{\perp_1}$.  It follows from Corollary \ref{corollary:BongartzQuiver} that $\CC_{[\WW]} \cong \CC_Q$, for a finite quiver $Q$, and from Corollary \ref{corollary:BongartzObjects} that $\CC_{[\WW]}$ has a cluster-tilting object $R$ of which $X$ is a direct summand.  The result now follows from \cite{BuanMarshReinekeReitenTodorov06} {(see also \cite[Theorem 19]{Hubery11})}.

\item It is clear that if $\VV$ is reachable from $\UU$, then $\UU$ and $\VV$ can only differ at finitely many indecomposable objects.  For the other direction, assume that $\UU$ and $\VV$ differ at only finitely many indecomposable objects.  Let $\WW = \UU \cap \VV$.  As above, we may use Corollary \ref{corollary:BongartzQuiver} to see that $\CC_{[\WW]} \cong \CC_Q$, for a finite quiver $Q$.  Note that both $\UU$ and $\WW$ become cluster-tilting objects in $\CC_{[\WW]}$.  Again, the result now follows from the connectedness of the exchange graph of $\CC_{[\WW]} \cong \CC_Q$ (see \cite{BuanMarshReinekeReitenTodorov06} or \cite[Theorem 19]{Hubery11}).

For the last statement, note that $|\ind(\UU / [\WW])| = |\ind(\UU) \setminus \ind \WW)|$ and $|\ind(\VV / [\WW])| = |\ind(\VV) \setminus \ind \WW)|$.  Since $\UU / [\WW]$ and $\VV / [\WW]$ are cluster-tilting subcategories of $\CC_{[\WW]}$, the statement follows from \cite[Theorem 2.4]{DehyKeller08} (see also \cite[Corollary 4.5]{AdachiIyamaReiten14} or \cite[Corollary 3.7]{ZhouZhu11}).
\qedhere
\end{enumerate}
\end{proof}

\section{Triangulations of cyclic orders}\label{section:Triangulations}

In this section, we generalize the notion of a triangulation of an $n$-gon to a triangulation of a locally discrete cyclically ordered set.  In the first part, we give a short account about cyclically ordered sets.  In the second part, we define triangulations, maximality, connectedness, and local finiteness and introduce the notion of exchanging arcs. We then proceed to study some connections between these notions.

\subsection{Definitions}\label{subsection:CyclicallyOrderedDefinitions}

Let $C$ be a set and let $R \subseteq C^3$ be a ternary relation.  We say that $(C,R)$ is a \emph{cyclically ordered set} or a \emph{cyclic order} (see \cite{Novak82}) if the following conditions are satisfied: 
\begin{description}
\item[Cyclicity] $\forall a,b,c \in C: (a,b,c) \in R \Rightarrow (b,c,a) \in R,$
\item[Asymmetry] $\forall a,b,c \in C: (a,b,c) \in R \Rightarrow (c,b,a) \not\in R,$
\item[Transitivity] $\forall a,b,c,d \in C: (a,b,c) \in R \land (a,c,d) \in R \Rightarrow (a,b,d) \in R$,
\item[Totality] $\forall a,b,c \in C: \mbox{if $a,b,c$ are distinct, then $(a,b,c) \in R \lor (c,b,a) \in R$.}$
\end{description}
We will also write $R(a,b,c)$ for $(a,b,c) \in R$.

\begin{remark}\label{remark:Diagonal}
Note that $R(a,b,c)$ implies that $a,b,c$ are distinct.  Indeed, assume that $a=c$ (by cyclicity, we may reduce to this case).  Asymmetry then implies that $(a,b,c) \not\in R.$  Thus, for $R \not= \emptyset$, the set $C$ needs at least three elements.
\end{remark}

\begin{remark}
If $(C,R)$ is a cyclic order, then so is $(C,R^\circ)$ where $R^\circ = \{(c,b,a) \in C^3 \mid (a,b,c) \in R \}.$
\end{remark}

\begin{remark}\label{remark:CyclicSuborder}
Let $(C,R)$ be a cyclic order.  For any $D \subseteq C$, we have that $(D, R|_D)$, where $R|_D = R \cap D^3$, is a cyclic order.
\end{remark}

\begin{remark}
In contrast to our conventions for linearly ordered sets (see \S\ref{subsection:Lin ord sets}), we do not interpret a cyclically ordered set as a category. For such definitions, we refer to, for example, \cite{Drinfeld04, IgusaTodorov15, vanRoosmalen12b}.
\end{remark}

A \emph{morphism} $f\colon (C,R_C) \to (D,R_D)$ of cyclic orders is a map $f\colon C \to D$ such that $\forall a,b,c \in C: R_D(f(a),f(b),f(c)) \Rightarrow R_C(a,b,c)$ (see \cite{NovakNovotny85}).

We will write $\COrd$ for the category of cyclic orders with the above morphisms.

\begin{remark}\label{remark:MonoAndIso}
\begin{enumerate}
\item {A cyclic order $R$ as defined above is the cyclic analogue of a \emph{strict} linear ordering (Example~\ref{example:CyclicFromLinear} will make this more precise). One can characterize morphisms $f\colon (\LL,\le) \to (\LL',\le)$ of linearly ordered sets in an analogous way using the strict version of the ordering: $\forall a,b \in \LL: f(a)<f(b) \Rightarrow a<b$.} 

\item A morphism $f\colon (C,R_C) \to (D,R_D)$ is a monomorphism if and only if $f\colon C \to D$ is an injection.  Indeed, if $f$ is an injection, then it is clear that $f$ is a monomorphism.  For the other direction, assume that $f$ is not an injection, so that there are $a,b \in C$ such that $a \not= b$ and $f(a) = f(b)$.  Consider the morphisms {$g_a \colon (\{\ast\}, \emptyset) \to (C,R)$ and $g_b\colon (\{\ast\}, \emptyset) \to (C,R)$} given by $g_a(\ast) = a$ and $g_b(\ast) = b$.  We have $f \circ g_a = f \circ g_b$ while $g_a \not= g_b$.  This shows that $f$ is not a monomorphism.
\item A morphism $f\colon (C,R_C) \to (D,R_D)$ is an isomorphism if and only if $f\colon C \to D$ is a bijection.
\end{enumerate}
When $f$ is a monomorphism or an isomorphism, we have $\forall a,b,c \in \LL: R_D(f(a),f(b),f(c)) \Leftrightarrow R_C(a,b,c)$.
\end{remark}

\begin{remark}
The cyclic orders we consider here are slightly different from the ones considered in \cite{DyckerhoffKapranov15}: one can pass between them by adding or removing the triples $(a,b,c)$ for which $|\{a,b,c\}| \not= 3$ to or from $R$ (compare Remark \ref{remark:Diagonal} with \cite[Definition II.25(C5)]{DyckerhoffKapranov15}).  However, the morphisms considered in \cite{DyckerhoffKapranov15} carry more information than the morphisms we consider.  Consequently, Connes' cyclic category $\Lambda$ (see \cite{Connes94}) is not a full subcategory of $\COrd$ (compare with \cite{DyckerhoffKapranov15}).  However, the notions of a monomorphism and an isomorphism coincide between our approach and the approach in \cite{DyckerhoffKapranov15}.
\end{remark}

\begin{example}\label{example:CyclicOrders}
Let $(\LL, \leq)$ be a linearly ordered set and let $\Sigma\colon (\LL, \leq) \to (\LL, \leq)$ be an automorphism.  We require $\Sigma$ to be \emph{unbounded}, meaning that for every $a,b \in \LL$, there are $n,m \in \bZ$ such that $\Sigma^{n} a < b < \Sigma^{m} a$.  For an element $a \in \LL$, we write $\overline{a}$ for the $\Sigma$-orbit of $a$ in $\LL$.  The set $\LL / \Sigma$ has the following cyclic ordering $R$:
$$ (\overline{a}, \overline{b}, \overline{c}) \in R \Leftrightarrow \exists k,l \in \bZ: a < \Sigma^k b < \Sigma^l c < \Sigma(a).$$
More explicitly, by taking
\begin{enumerate}
\item $(\LL, \leq) = (\bR, \leq)$ and $\Sigma(x) = x+1$, we find a cyclic ordering on the circle $S^1 \cong \bR / \bZ$,
\item $(\LL, \leq) = (\bZ, \leq)$ and $\Sigma(x) = x+n$ (for $n \geq 2$), we find a cyclic ordering on $\bZ_n$.
\end{enumerate}
Conversely, every cyclically ordered set can be obtained in this way.  Indeed, let $(C, R)$ be any cyclically ordered set and $c \in C$.  We consider the linear ordering $\leq_c$ on $C$ given as:
\[ a <_c b \Leftrightarrow \mbox{$(a = c) \lor R(c,a,b)$,} \]
for $a \not= b$.  We set $\LL = \bZ \stackrel{\rightarrow}{\times} C$ and $\Sigma: \LL \to \LL: (n,a) \mapsto (n+1,a)$.  In this case, the map $\LL / \Sigma \to C\colon \overline{(n,a)} \to a$ is an isomorphism of cyclically ordered sets.
\end{example}

\begin{example}\label{example:CyclicFromLinear}
With any linearly ordered set $(\LL, \leq)$, we may associate a cyclic order $(\LL, R) = (\LL, \leq)_{\rm cyc}$ as follows:
\[ R(a,b,c) \Leftrightarrow (a < b < c) \lor (b < c < a) \lor (c < a < b). \]
This correspondence defines a functor $(-)_{\rm cyc}$ from the category $\LOrd$ of linearly ordered sets to the category $\COrd$ of cyclically ordered sets.  This functor is essentially surjective (see \cite[3.11 Lemma]{Novak82}).
\end{example}

\begin{definition}\label{definition:RotationallyEquivalent}
A linearly ordered set $(\LL, \leq)$ such that $(\LL, \leq)_{\rm cyc} \cong (C, R)$ is called a \emph{cut} of $(C, R)$.  Two linearly ordered sets, {$(\LL, \leq)$ and $(\LL', \leq)$}, are called \emph{cyclically equivalent} or \emph{rotationally equivalent} if $(\LL, \leq)_{\rm cyc} \cong (\LL', \leq)_{\rm cyc}$, thus if they are cuts of the same cyclically ordered set.
\end{definition}

\begin{proposition}\label{proposition:CriterionCyclicallyEquivalent}
The linearly ordered sets $(\LL, \leq)$ and $(\LL', \leq)$ are cyclically equivalent if and only if there are linearly ordered sets $(\LL_1,\leq)$ and $(\LL_2, \leq)$ such that $\LL \cong \LL_1 \cdot \LL_2$ and $\LL' \cong \LL_2 \cdot \LL_1$.
\end{proposition}

\begin{proof}
See \cite[3.6 Theorem]{Novak84}.
\end{proof}

\begin{definition}
Let $(C,R)$ be a cyclically ordered set.  For $a,b \in C$, we define the open interval $(a,b) \subseteq C$ as
$$(a,b) = \{x \in \LL \mid R(a,x,b) \}.$$
We define the closed interval $[a,b]$ as $\{a,b\} \cup (a,b)$.  Half-open intervals are defined similarly.  Note that, due to asymmetry, we have $(a,b) \cap (b,a) = \emptyset.$
\end{definition}

\begin{proposition}\label{proposition:CyclicDecomposition}
Let $(C,R)$ be a cyclically ordered set.  For each $a,b \in C$ (with $a \not= b$), there is a decomposition $C = (a,b) \coprod [b,a]$.
\end{proposition}

\begin{proof}
Using asymmetry, we easily see that
$$(a,b) \cap [b,a] = (a,b) \cap \big( (b,a) \cup \{a,b\} \big) = \big( (a,b) \cap (b,a) \big) \cap \big( (a,b) \cap \{a,b\} \big) = \emptyset.$$
Now, let $x \in \LL$.  If $x \in \{a,b\}$, then $x \in [b,a]$.  We may thus assume that $x \not\in\{a,b\}$.  Using totality, we have either $R(a,b,x)$ (and thus $x \in (b,a) \subseteq [b,a]$) or $R(a,x,b)$ (and thus $x \in (a,b)$).  In both cases, we have $x \in (a,b) \cup [b,a]$, as required.
\end{proof}

\begin{proposition}\label{proposition:FiniteCycles}
Let $n \geq 2$.  Any cyclically ordered set $(C,R)$ with $\# C = n$ is isomorphic to $\bZ_n$.
\end{proposition}

\begin{proof}
Let $(C, R_C)$ and $(D, R_D)$ be cyclically ordered sets with $\# C = \# D = n \in \bN$.  By \cite[3.11 Lemma]{Novak82}, we know that there are linear orderings on $C$ and $D$ such that $(C, R_C) \cong (C,\leq)_{\rm cyc}$ and $(D, R_D) \cong (D,\leq)_{\rm cyc}$.  Since $(C, \leq)$ and $(D, \leq)$ are linearly ordered sets with $\# C = \# D = n \in \bN$, we know that $(C, \leq) \cong (D, \leq)$.  This implies that $(C,\leq)_{\rm cyc} \cong (D,\leq)_{\rm cyc}$, and hence $(C, R_C) \cong (D, R_D)$.

Since $\bZ_n$ is a cyclically ordered set (see Example \ref{example:CyclicOrders}) with $n$ elements, we know that $(C,R)$ is isomorphic to $\bZ_n$.
\end{proof}

\begin{definition}
Let $(C,R)$ be a cyclically ordered set.  An \emph{$n$-cycle} of $C$ is a subobject of $C$ in the category $\COrd$ with $n$ elements.  By Proposition \ref{proposition:FiniteCycles}, we know that (up to isomorphism), an $n$-cycle is a monomorphism $f:\bZ_n \to C$.  We will denote this subobject by $\langle f(\overline{0}), f(\overline{1}), \ldots, f(\overline{n-1})\rangle$
\end{definition}

\begin{remark}
Every subset $\{a_0, a_1, \ldots, a_{n-1}\} \subseteq C$ gives rise to an $n$-cycle.  This follows from Remark \ref{remark:CyclicSuborder} and Proposition \ref{proposition:FiniteCycles}.
\end{remark}

\subsection{Locally discrete cyclically ordered sets}

We will say that the cyclic order $(C,R)$ is \emph{locally discrete} if and only if for every $a \in C$, there are elements $a-1,a+1 \in C \setminus \{a\}$ such that $\forall b \in C: (a,b,a+1) \not\in R \land (a-1,b,a) \not\in R$.  Note that, by totality, $a-1$ and $a+1$ are uniquely defined by this property.  We say that $a-1$ is a \emph{direct predecessor} of $a$ and $a+1$ is a \emph{direct successor} of $a$, both are \emph{neighbors} of $a$.  We define, for all $n > 0$:
\begin{eqnarray*}
a-(n+1) = (a-n)-1, \\
a+(n+1) = (a+n)+1.
\end{eqnarray*}

\begin{lemma}
For all $n,m \in \bZ$ and all $a \in C$, we have $(a+n)+m = a + (n+m)$.
\end{lemma}

\begin{proof}
If $m$ and $n$ are both positive or both negative, the statement follows easily from induction.  The other cases follow by induction and from $(a+1)-1 = a = (a-1)+1$, which follows directly from the definition.
\end{proof}

\begin{proposition}
Let $(\LL, \leq)$ be a linearly ordered set.  The cyclic order $(\LL, \leq)_{\rm cyc} = (\LL, R)$ is locally discrete if and only if
\begin{enumerate}
\item $(\LL, \leq)$ is a locally discrete linearly ordered set with no minimum and no maximum, or
\item $(\LL, \leq)$ is a locally discrete linearly ordered set with both a minimum and a maximum.
\end{enumerate}
If $(C,R)$ is a locally discrete cyclically ordered set, then every cut is of this form.
\end{proposition}

\begin{proof}
This follows from \cite[3.5 Lemma]{Novak84}.
\end{proof}

\begin{proposition}\label{proposition:SubobjectsOrderings}
Let $(C, R)$ be a cyclically ordered set and let $c \in C$.  The correspondence $(D, \leq_c) \mapsto (D,R|_D)$ is a bijection between the subobjects of $(C, \leq_c)$ and the subobjects of $(C, R)$.
\end{proposition}

\begin{proof}
{This is} straightforward.
\end{proof}

\begin{proposition}\label{proposition:Countable}
Let $(\LL, \leq)$ be a linearly ordered set.  
\begin{enumerate}
\item If $\LL$ is locally discrete and $\LL$ is a subposet of {$(\bR, \leq)$}, then $\LL$ is countable.
\item If $\LL$ is countable, then $\LL$ is a subposet of {$(\bQ,\leq)$}.
\end{enumerate}
\end{proposition}

\begin{proof}
The second statement is standard and easy to prove; we will only prove the first statement.

First, assume that $\LL$ is a subposet of $\bR$; we want to show that $\LL$ is countable.  To this end, we will construct, for each integer $n > 0$, a countable subset $\LL_n \subseteq \LL$, and then show that $\cup_{n > 0} \LL_n = \LL$.

So, fix an integer $n> 0$.  We define
$$\LL_n = \left\{l \in \LL \mid \exist \lambda \in \bZ: \LL \cap \left[\frac{\lambda}{n}, \frac{\lambda+1}{n}\right) = \{l\} \right\}.$$
Note that for each $l \in \LL_n$, there is a unique $\lambda \in \bZ$ such that $l \in  [\frac{\lambda}{n}, \frac{\lambda+1}{n})$.  This gives an injection $\LL_n \to \bZ$ and shows that $\LL_n$ is countable.

We are left with showing that $\cup_{n > 0} \LL_n = \LL$.  Since $\bR = \cup_{\lambda \in \bZ} [\frac{\lambda}{n}, \frac{\lambda+1}{n})$, we know that for each $l \in \LL$, there is a $\lambda \in \bZ$ such that $l \in [\frac{\lambda}{n}, \frac{\lambda+1}{n})$.  Since $\LL$ is locally discrete, we know that $l$ is the unique element in $\LL \cap [\frac{\lambda}{n}, \frac{\lambda+1}{n})$ if and only if no neighbors of $l$ lie in $[\frac{\lambda}{n}, \frac{\lambda+1}{n})$.  Since $l$ has at most two neighbors, this will be the case when $n$ is large enough.  Hence, $\cup_{n > 0} \LL_n = \LL$ and $\LL$ is countable.
\end{proof}

The following corollary among others says that, at least as long as we are concerned with countable cyclically ordered sets, we can draw them on a circle. This gives for instance a precise meaning to Figures~\ref{figure:Crossing}, \ref{figure:NoncrossingSetsOfArcs} and~\ref{figure:LocallyFiniteTriangle} below.

\begin{corollary}\label{corollary:Countable}
Let $(C, R)$ be a cyclically ordered set.  
\begin{enumerate}
\item If $C$ is locally discrete and a cyclically ordered subset of $S^1$, then $C$ is countable.
\item If $C$ is countable, then $C$ is a cyclically ordered subset of $\bQ / \bZ \subset S^1$.
\end{enumerate}
\end{corollary}

\begin{proof}
This follows directly from Propositions \ref{proposition:SubobjectsOrderings} and \ref{proposition:Countable}.
\end{proof}

\subsection{Arcs and triangulations}

We will now look at triangulations of a locally discrete cyclically ordered set $(C, R)$.

\begin{definition}
\begin{enumerate}
\item An edge of $C$ is a pair $\{a,b\} \subseteq C$ such that $a$ and $b$ are neighbors.
\item An \emph{arc} or a \emph{diagonal} of $C$ is a pair $\{a,b\} \subseteq C$ such that $a \not= b$ and $\{a,b\}$ is not an edge.  The set of all arcs in $(C,R)$ is denoted by $\arc(C,R)$ or just $\arc(C)$ if there is no cause for confusion.
\end{enumerate}
We will often write $ab$ for the edge or the arc $\{a,b\}$; note that $ab = ba$.
\end{definition}

\begin{remark}
An arc in the cyclically ordered set $(C,R)$ is a subset $\{a,b\} \subseteq C$ such that there are elements $c,d \in \LL$ with $R(a,c,b)$ and $R(b,d,a)$.  Put differently, $ab$ is an arc in $C$ if and only if there are $c,d \in C$ such that $\langle a,c,b,d \rangle$ is a 4-cycle.
\end{remark}

The following definition and subsequent remarks are illustrated in Figure~\ref{figure:Crossing}.

\begin{figure}
  \centering
	
  \begin{tikzpicture}
    \TikzCross{0}{0}
		\TikzAdjacentPair{5}{0}
    \TikzDisjointPair{10}{0}
  \end{tikzpicture}

\caption{A pair of crossing arcs, adjacent arcs and noncrossing nonadjacent arcs, respectively.}
\label{figure:Crossing}
\end{figure}
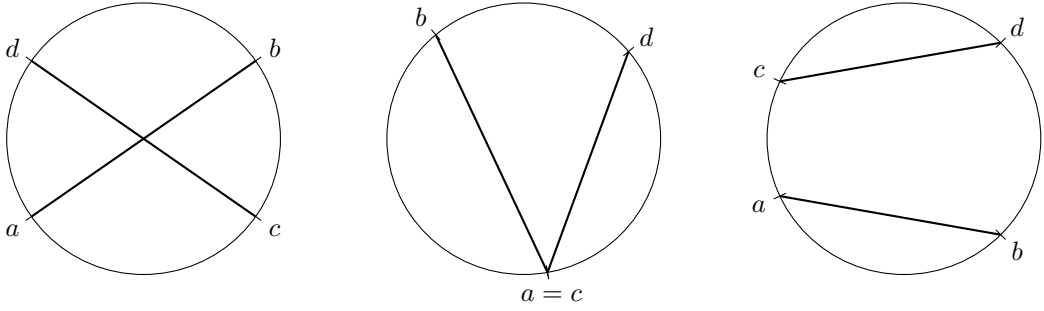

\begin{definition}
Let $ab, cd \in \arc(C)$.  We say that $ab$ and $cd$ are \emph{crossing} if and only if $R(a,c,b) \land R(b,d,a)$ or $R(a,d,b) \land R(b,c,a)$.  Two arcs which are not crossing are said to be \emph{noncrossing}.  A set of arcs such that every two arcs are noncrossing is called a \emph{set of pairwise noncrossing arcs}.

Two distinct arcs $ab$ and $cd$ are called \emph{adjacent} if $\{a,b\} \cap \{c,d\} \not= \emptyset$.
\end{definition}

\begin{remark}
The arcs $ab$ and $cd$ cross if and only if either $acbd$ or $adbc$ is a 4-cycle.
\end{remark}

\begin{remark}
If $ab$ and $cd$ are crossing arcs, then $\{a,b\} \cap \{c,d\} = \emptyset$.
\end{remark}

\begin{definition}
Let $(C,R)$ be a locally discrete cyclically ordered set and let $S \subseteq \arc(C)$ be a set of pairwise noncrossing arcs on $C$.  An $n$-cycle $f: \bZ_n \to C$ is called an \emph{$n$-gon} in $S$ if and only if for all $i \in \bZ_n$, we have that $\{f(i), f(i+1)\}$ is either an edge in $C$ or an arc in $S$.  A 3-gon in $C$ is also called a \emph{triangle}.

Similar to the notation of edges, we will write $a_0 a_1 \ldots a_{n-1}$ for an $n$-gon $\langle a_0,a_1,\ldots, a_{n-1} \rangle$.
\end{definition}

\begin{remark}
\begin{enumerate}
\item The 2-gons in $S$ are exactly the edges in $C$ and the arcs in $S$.
\item By asymmetry, if $abc$ is a triangle in $S$, then $acb$ is not a triangle in $S$.
\end{enumerate}
\end{remark}

\begin{lemma}\label{lemma:Noncrossing}
Let $ab, cd \in \arc(C)$.
\begin{enumerate}
\item The arcs $ab, cd$ cross if and only if either $c \in (a,b)$ and $d \in (b,a)$ or $d \in (a,b)$ and ${c} \in (b,a)$
\item the arcs $ab, cd$ do not cross if and only if either $c,d \in [a,b]$ or $c,d \in [b,a]$.
\end{enumerate}
\end{lemma}

\begin{proof}
The first statement follows directly from the definition.  The second statement is the contrapositive of the first one and uses that $C = (a,b) \coprod [b,a]$ (Proposition \ref{proposition:CyclicDecomposition}).
\end{proof}

\begin{proposition}\label{proposition:AtMost2Triangles}
Let $(C,R)$ be a cyclically ordered set and let $S \subseteq \arc(C)$ be a set of pairwise noncrossing arcs.  An arc $ab \in S$ lies in at most one triangle of the form $axb$ and one triangle of the form $aby$.
\end{proposition}

\begin{proof}
Let $axb$ and $azb$ be triangles in $S$.  We want to show that $x=z$.  Seeking a contradiction, assume that $x \not= z$.  By totality (and possibly changing the roles of $x$ and $z$), we may assume that $R(a,x,z)$, so that there is a 4-cycle $\langle a,x,z,b \rangle$ in $C$.  This means that $az$ and $xb$ are not edges and hence they are crossing arcs.  This is the required contradiction.
\end{proof}

\begin{definition}
Let $(C,R)$ be a cyclically ordered set and let $S \subseteq \arc(C)$ be a set of pairwise noncrossing arcs.  Let $a,b \in C$.  A \emph{path} in $S$ of length $n$ from $a$ to $b$ is a sequence $a = x_0, x_1, \ldots, x_n = b$ where $x_i x_{i+1} \in S$ for all $i \in \{0, 1, \ldots, n-1\}$.
\end{definition}

\begin{definition}\label{definition:MainTriangulationProperties}
\begin{enumerate}
\item A set $S$ of pairwise noncrossing arcs is said to be \emph{connected} if for every $a,b \in C$ incident with $S$, there is a path from $a$ to $b$.
\item A set $S$ of pairwise noncrossing arcs is \emph{maximal} if it is not a proper subset of a set $S$ of pairwise noncrossing arcs.
\item A set $S$ of pairwise noncrossing arcs is called a \emph{triangulation} if every arc in $S$ lies in exactly two triangles in $S$ and every edge lies in exactly one triangle in $S$.
\item A set $S$ of pairwise noncrossing arcs is \emph{locally finite} if every element of {$C$} is incident with finitely many arcs of $S$.
\end{enumerate}
We will say that $S$ is a \emph{connected triangulation} if $S$ is connected and a triangulation.
\end{definition}

\begin{remark}
\begin{enumerate}
\item An alternative definition of connectedness can be given as follows.  A set $S$ of pairwise noncrossing arcs is said to be connected if and only if for any arcs $ab, cd \in S$, there is a sequence of arcs $ab= x_0 y_0, x_1 y_1, \ldots, x_n y_n = cd$ where $x_i y_i$ and $x_{i+1} y_{i+1}$ are adjacent (for all $i \in \{0,1, \ldots, n-1\}$).
\item A triangulation need not be a maximal set of noncrossing arcs, nor is a maximal set of noncrossing arcs always a triangulation {(see Examples~\ref{example:NoncrossingSets}(\ref{enumerate:Example3}) and~(\ref{enumerate:Example9}) below, respectively)}.
\end{enumerate}
\end{remark}

\begin{example}\label{example:LightConeArcs}
Let $a \in {C}$, and let
$$S_a = \big\{ \{ a,x \} \mid x \not\in \{x-1,x,x+1\} \big\}.$$
It is easily checked that $S_a$ is a connected triangulation on {$C$}.
\end{example}

\begin{figure}
  \centering
	
  \begin{tikzpicture}
    \TikzNonmaxBrokenFountain{0}{0}
    \TikzMaxBrokenFountain{7}{0}
  \end{tikzpicture}

  \vspace{1em}		

  \begin{tikzpicture}
    \TikzFountainAndZigZag{0}{0}
	  \TikzDoubleZigZag{7}{0}
  \end{tikzpicture}

  \vspace{1em}
  
	\begin{tikzpicture}
	  \TikzSpaceShip{0}{0}
	  \TikzFountain{7}{0}
  \end{tikzpicture}

  \vspace{1em}
  
	\begin{tikzpicture}
	  \TikzBouncingGame{7}{0}
  \end{tikzpicture}

\caption{Noncrossing set of arcs $S_1$, $S_4$, $S_5$, $S_6$, $S_9$, $S_{10}$ and $S_{11}$ from Example~\ref{example:NoncrossingSets} (in this order). In the cases where $\LL = \{0,1\} \stackrel{\rightarrow}{\times} \bZ$, we denote the elements $(0,n) \in \LL$ simply by $n$ and the elements $(1,n) \in \LL$ by $n^*$.}
\label{figure:NoncrossingSetsOfArcs}
\end{figure}

\begin{example}\label{example:NoncrossingSets}
We give examples of sets of pairwise noncrossing arcs that illustrate some connections between the concepts of Definition \ref{definition:MainTriangulationProperties} (see table below). {We will show in Theorems~\ref{theorem:ConnectedTriangulationsAreMaximal} and~\ref{theorem:LocallyFiniteTriangulation} that these examples cover all possible combinations of the four properties (connectedness, maximality, being a triangulation and being locally finite) which can occur. The cyclically ordered sets $C$ will be given by a suitable cut $\LL$ and can be reconstructed via Example~\ref{example:CyclicFromLinear}. Seven of the sets of arcs,} $S_1$, $S_4$, $S_5$, $S_6$, $S_9$, $S_{10}$ and $S_{11}$, are depicted in Figure~\ref{figure:NoncrossingSetsOfArcs}, while the other four are just subsets of one of these seven (namely $S_2 \subseteq S_1$, $S_3 \subseteq S_9$ and $S_7,S_8 \subseteq S_{10}$).
\begin{enumerate}
\item\label{enumerate:Example1} Let $\LL = \bZ$ and let {$S_1 = \big\{ \{n, -1\} \mid n < -2 \big\} \cup \big\{ \{1,m\} \mid m > 2 \big\}$}.
\item\label{enumerate:Example2} Let $\LL = \bZ$ and let {$S_2 = \big\{ \{-3,-1\}, \{1,3\} \big\}$}.
\item\label{enumerate:Example3} Let $\LL = \{0,1\} \stackrel{\rightarrow}{\times} \bZ$ and let
$$S_3 = \big\{ \{(0,n), (0,0)\} \mid n \not\in \{-1,0,1\} \big\} \cup \big\{ \{(1,n), (1,0)\} \mid n \not\in \{-1,0,1\} \big\} $$
\item\label{enumerate:Example4} Let $\LL = \bZ$ and let $S_4 = \big\{ \{ n,0 \} \mid n < -1 \big\} \cup \big\{ \{ 1,m \} \mid m > 2 \big\}.$
\item\label{enumerate:Example5} Let $\LL = \{0,1\} \stackrel{\rightarrow}{\times} \bZ$ and 
\begin{eqnarray*}
S_5 &=& \big\{ \{(0,n), (0,0)\} \mid n \not\in \{-1,0,1\} \big\} \\
&&\cup \big\{ \{(1,n),(1,-n) \} \mid n \geq 1 \big \} \cup \big\{ \{(1,n+1),(1,-n) \} \mid n \geq 1 \big\}.
\end{eqnarray*}
\item\label{enumerate:Example6} Let $\LL = \{0,1\} \stackrel{\rightarrow}{\times} \bZ$ and 
\begin{eqnarray*}
S_6 &=& \big\{ \{(0,n),(0,-n) \} \mid n \geq 1 \big \} \cup \big\{ \{(0,n+1),(0,-n) \} \mid n \geq 1 \big\} \\
&&\cup \big\{ \{(1,n),(1,-n) \} \mid n \geq 1 \big \} \cup \big\{ \{(1,n+1),(1,-n) \} \mid n \geq 1 \big\}.
\end{eqnarray*}
\item\label{enumerate:Example7} Let $\LL = \bZ$ and $S_7 = \big\{ \{0,n\} \mid n \not\in \{-1,0,1,2\} \big\}.$
\item\label{enumerate:Example8} Let $\LL = \bZ$ and let $S_8 = \big\{ \{0,2\} \big\}$.
\item\label{enumerate:Example9} Let $\LL = \{0,1\} \stackrel{\rightarrow}{\times} \bZ$ and let
\begin{eqnarray*}
S_9 &=& \big\{ \{ (0,n), (0,0) \} \mid n \not\in \{-1,0,1\} \} \cup \big\{ \{ (1,n), (1,0) \} \mid n \not\in \{-1,0,1\} \} \\
&& \cup \{((0,0),(1,0))\},
\end{eqnarray*}
thus $S_{9} = S_4 \cup \{((0,0),(1,0))\}$.
\item\label{enumerate:Example10} Let $\LL = \bZ$ and $S_{10} = \big\{ \{0,n\} \mid n \not\in \{-1,0,1\} \big\}.$
\item\label{enumerate:Example11} Let $\LL = \{0,1\} \stackrel{\rightarrow}{\times} \bZ$ and {$S_{11} = \big\{ \{(0,n),(1,-n)\} \mid n \in \bZ \big\} \cup \big\{ \{(0,n),(1,-n+1) \} \mid n \in \bZ \big\}$}.
\end{enumerate}

The properties of these examples can be {summarized} in the following table: \\
\begin{center}
\begin{tabular}{ r | c  c  c  c  c  c  c  c c c c}
{} & 1 & 2 & 3 & 4 & 5 & 6 & 7 & 8 & 9 & 10 & 11 \\
\hline
Connected & $\times$ & $\times$ & $\times$ & $\times$ & $\times$ & $\times$ & \checkmark & \checkmark & \checkmark & \checkmark & \checkmark \\
Maximal & $\times$ & $\times$ & $\times$ & \checkmark & \checkmark & \checkmark & $\times$ & $\times$ & \checkmark & \checkmark & \checkmark \\
Triangulation & $\times$ & $\times$ & \checkmark & $\times$ & \checkmark & \checkmark & $\times$ & $\times$ & $\times$ & \checkmark & \checkmark \\
Locally finite & $\times$ & \checkmark & $\times$ & $\times$ & $\times$ & \checkmark & $\times$ & \checkmark & $\times$ & $\times$ & \checkmark  
\end{tabular}
\end{center}

Note that all the cyclically ordered sets in this example are countable, so they are subobjects of the unit circle (as cyclically ordered sets).
\end{example}

We will {also need the following lemma and proposition}.

\begin{lemma}\label{lemma:Maximal}
Let $S$ be a set of pairwise noncrossing arcs in $C$ and $|C| \geq 4$.  Assume that $S$ is maximal or a triangulation.  If $a \in C$ is not incident with $S$, then $\{ a+1,a-1 \} \in S$.
\end{lemma}

\begin{proof}
Note that since $|C| \geq 4$, we know that $a+1$ and $a-1$ are not neighbors and thus $\{a+1,a-1\} \in \arc(C)$.

First, assume that $S$ is maximal.  If no arc of $S$ is incident with $a$, then $\{a-1,a+1 \} \in \arc(C)$ crosses no arcs in $S$.  It follows from the maximality of $S$ that $\{ a-1,a+1 \} \in S$.

Next, assume that $S$ is a triangulation.  Let $\langle x, a, a+1 \rangle$ be the triangle containing the edge $\{a, a+1\}$.  If $a \in {C}$ is not incident with $S$, then $xa \not\in S$.  This implies that $xa$ is an edge, and thus $x = a-1$.  This shows that $\{ a+1,a-1 \} \in S$.
\end{proof}

\begin{proposition}\label{proposition:CuttingUpS}
Let $S$ be a set of pairwise noncrossing arcs for $C$ and let $ab \in S$.
\begin{enumerate}
\item\label{enumerate:IyamaYoshinoMaximal} $S$ is maximal if and only $S|_{[a,b]}$ and $S|_{[b,a]}$ are maximal in $[a,b]$ and $[b,a]$, respectively.
\item\label{enumerate:IyamaYoshinoTriangulated} $S$ is a triangulation if and only if $S|_{[a,b]}$ and $S|_{[b,a]}$ are triangulations of $[a,b]$ and $[b,a]$, respectively.
\item\label{enumerate:PathPassesEndpoints} Any path from $x \in [a,b]$ to $y \in [b,a]$ passes through either $a$ or $b$.
\item\label{enumerate:IyamaYoshinoMaximalConnected} If $S|_{[a,b]}$ and $S|_{[b,a]}$ are maximal and connected in $[a,b]$ and $[b,a]$, respectively, then $S$ is maximal and connected.
\item $S$ is locally finite if and only if $S|_{[a,b]}$ and $S|_{[b,a]}$ are locally finite.
\end{enumerate}
\end{proposition}

\begin{proof}
\begin{enumerate}
  \item {This is} straightforward.
	\item Assume that $S$ is a triangulation; we will show that $S|_{[a,b]}$ is a triangulation for $[a,b]$.

By Proposition \ref{proposition:AtMost2Triangles}, we know that an arc $cd \in S|_{[a,b]}$ lies in at most two triangles in $S|_{[a,b]}$, so it suffices to find those two triangles.  We know that there are two triangles in $S$ which contain $cd$ and, using the fact that the arcs of $S$ do not cross $ab$, we can deduce that both triangles are triangles in $S|_{[a,b]}$.

The case where $cd$ is an edge in $[a,b]$, different from the edge $ab$ itself, is similar.  If $cd = ab$ (say, $a=c$ and $b=d$), we know that there are two triangles in $S$ which contain $cd$.  By Proposition \ref{proposition:AtMost2Triangles}, one of the triangles is $cxd$ for some $x \in [a,b] = [c,d]$.  This shows that $S|_{[a,b]}$ is a triangulation of $[a,b]$.

Showing that $S|_{[b,a]}$ is a triangulation of $[b,a]$ is similar.

For the other direction, assume that both $S|_{[a,b]}$ and $S|_{[b,a]}$ are triangulations of $[a,b]$ and $[b,a]$, respectively.  We want to show that $S$ is a triangulation of $\LL$.  First, we check that $ab$ lies in two triangles.  Since both $S|_{[a,b]}$ and $S|_{[a,b]}$ are triangulations and $ab$ is an edge in both $[a,b]$ and $[b,a]$, we know that there is a unique triangle in $S|_{[a,b]}$ containing $ab$ and a unique triangle in $S|_{[b,a]}$ containing $ab$.  These are two different triangles in $S$.

Consider now an edge $cd \in S$, different from $ab$.  Without loss of generality, by Lemma \ref{lemma:Noncrossing}, we may assume that $c,d \in [a,b]$.  Since $S|_{[a,b]}$ is a triangulation of $[a,b]$, we know that $cd$ lies in two triangles in $S|_{[a,b]}$.  Hence, $cd$ lies in at least two triangles in $S$.  Proposition \ref{proposition:AtMost2Triangles} then shows that $cd$ lies in exactly two triangles in $S$. 

  \item This is obvious.

  \item We know by (\ref{enumerate:IyamaYoshinoMaximal}) that $S$ is maximal.  In order to show that $S$ is connected, it suffices to show that each $x \in \LL$ incident with $S$, admits a path to $a$.  The case where $x \in \{a,b\}$ is trivial, we will exclude this case.
	
	We will assume that $x \in (a,b)$; the case where $x \in (b,a)$ is similar.  We know that $S|_{[a,b]}$ is maximal in $[a,b]$, hence $a$ or $b$ is incident with $S|_{[a,b]}$ (Lemma \ref{lemma:Maximal}).  Since $S|_{[a,b]}$ is connected, we find that there is a path from $x$ to either $a$ or $b$ in $(a,b)$.  This yields a path from $x$ to $a$ in $S$.

  \item This is obvious.
\end{enumerate}
\end{proof}

\begin{remark}
A counterexample to the converse of Proposition \ref{proposition:CuttingUpS}(\ref{enumerate:IyamaYoshinoMaximalConnected}) is given by Example \ref{example:NoncrossingSets}(\ref{enumerate:Example9}) where one chooses $\{a,b\} = \{(0,0),(1,0)\}$.
\end{remark}

\subsection{Maximal and connected triangulations}

\begin{theorem}\label{theorem:ConnectedTriangulationsAreMaximal}
If $S$ is a connected triangulation, then $S$ is maximal.
\end{theorem}

\begin{proof}
Let $ab \in \arc C$ such that $ab$ intersects with no arcs of $S$; we want to show that $ab \in S$.  We start by showing that $S$ has arcs incident with $a$ and arcs incident with $b$.  Assume that $a$ is not incident with $S$.  It follows from Lemma \ref{lemma:Maximal} that $\{a-1, a+1\} \in S$.  However, $ab$ crosses $\{a-1, a+1\}$.  Since $ab$ does not cross any arcs of $S$, we know that $a$ is incident with $S$.  Similarly, one shows that $b$ is incident with $S$.

For the next step, we will use connectedness.  We consider a path
$$a = x_0, x_1, x_2, \ldots, x_n = b$$
in $S$ of length $n$ from $a$ to $b$.  We will furthermore assume that $n$ is minimal with this property, so there is no shorter path from $a$ to $b$.  We will show that $n = 1$, which implies that $ab \in S$.  Seeking a contradiction, we will assume that $n \not= 1$ and thus $x_1 \not= b$.

By totality, we know that either $R(a,b, x_1)$ or $R(a, x_1, b)$. Without loss of generality, we may assume the former, so that $\langle a, b, x_1 \rangle$ is a 3-cycle in {$C$}. Since $S$ is a triangulation and $ax_1 \in S$, we know that there is a triangle $\langle a, c, x_1 \rangle$ in $S$.  By totality, we have either the 4-cycle $\langle a, c, b, x_1 \rangle$ or $\langle a, b, c, x_1 \rangle$.  Since $cx_1$ does not cross the arc $ab$, we find the latter.  Note that $ac \in S$.

It follows from Proposition \ref{proposition:CuttingUpS} that the path $x_1, x_2, \ldots, x_n = b$ passes either $a$ or $c$.  Minimality of $n$ implies that the path passes $c$, and hence $c = x_k$ for some $1 < k \leq n$.  We now have a path $a = x_0, x_k, x_{k+1}, \ldots, x_n = b$, contradicting the minimality of $n$.  This concludes the proof.
\end{proof}

\begin{proposition}
Let $S$ be a set of pairwise noncrossing arcs for {$(C,R)$} and let $ab \in S$.  The set $S \subseteq \arc (C)$ is a connected triangulation if and only $S|_{[a, b]}$ and $S|_{[b, a]}$ are connected triangulations for $[a,b]$ and $[b,a]$, respectively.
\end{proposition}

\begin{proof}
  If $S|_{[a,b]}$ and $S|_{[b,a]}$ are connected triangulations for $[a,b]$ and $[b,a]$, respectively, then Theorem \ref{theorem:ConnectedTriangulationsAreMaximal}, together with Proposition \ref{proposition:CuttingUpS}(\ref{enumerate:IyamaYoshinoTriangulated} and \ref{enumerate:IyamaYoshinoMaximalConnected}) above, show that $S$ is a connected triangulation for $C$.
	
  For the other direction, assume that $S$ is a connected triangulation.  We know that $S|_{[a,b]}$ and $S|_{[b,a]}$ are triangulations of $[a,b]$ and $[b,a]$, respectively; we only need to show connectedness.  Let $x,y \in [a,b]$ be two elements incident with {$S|_{[a,b]}$}.
	
  Since $S$ is connected, there is a minimal path from $x$ to $y$ in $S$: $x = x_0, x_1, \ldots, x_n = y$.  The minimality and Proposition \ref{proposition:CuttingUpS}(\ref{enumerate:PathPassesEndpoints}) imply that $a$ and $b$ can each occur at most once in this path, and if both do occur, then they occur in adjacent positions.  Moreover, all of $x_0, x_1, \ldots, x_n = y$ belong to $[a,b]$.
	
  If $a$ and $b$ do not both occur, then the above path is a path in $S|_{[a,b]}$. Thus assume that $a$ and $b$ both occur in this path (necessarily in adjacent positions).  Using the fact that $S|_{[a,b]}$ is a triangulation and $ab$ is an edge in $S|_{[a,b]}$, we find a unique $c \in [a,b]$ such that $acb$ is a triangle in $S|_{[a,b]}$.

  There are two cases which could occur.  If $ac,cb \in \arc C$, then by inserting $c$ between $a$ and $b$ in the path from $x$ to $y$, we get a path from $x$ to $y$ in $S|_{[a,b]}$.  Otherwise, one of $ac$ and $bc$ is an edge; say it is $ac$ and $c = a+1$ in $C$.  In that case, since $b,c$ are the two neighbors of $a$ in the interval $[a,b]$ and $abc$ is a triangle in $S$, $a$ cannot be incident to any arc of $S|_{[a,b]}$.  Indeed, such an arc would cross $bc$.  In particular, $a$ is the end of the path, so that $a=x$ or $a=y$.  This is a contradiction, however, since we chose $x,y$ such that they were incident with $S_{[a,b]}$.  Hence, the second case cannot occur.
\end{proof}

{The rest of the subsection is devoted to a discussion of properties of connected triangulation which will in \S\ref{section:CominatoricsOfTypeA} lead to functorial finiteness of certain rigid subcategories of 2-Calabi-Yau triangulated categories.}

\begin{proposition}\label{proposition:TowardClusterTilting}
{Suppose that} $S \subseteq \arc C$ is a connected set of noncrossing arcs.
\begin{enumerate}
\item\label{enumerate:LocallyFiniteCrossings} If $S$ is maximal, then for each arc $ab \in \arc (C)$, there is a finite subset $F \subseteq C$ such that every arc of $S$ which intersects $ab$ is incident with $F$.
\item\label{enumerate:ClosestTriangle} Let $S$ be a triangulation.  For any arc $ab \not\in S$, there is a unique triangle $axy$ in $S$ such that $axby$ is a 4-cycle.
\end{enumerate}
\end{proposition}

\begin{proof}
\begin{enumerate}
\item We assume first that both $a$ and $b$ are incident with $S$.  Let $a= x_0, x_1, \ldots, x_n = b$ be a path from $a$ to $b$ in $S$.  We let $F = \{x_i\}_{0 < i < n}$.  Let $cd \in S$ be an arc which crosses $ab$.  It follows from Proposition \ref{proposition:CuttingUpS} that the path from $a$ to $b$ passes $c$ or $d$, hence $\{c,d\} \cap F \not= \emptyset$.  In this case, our choice of $F$ suffices.

We now turn our attention to the case where either $a$ or $b$ would not be incident with $S$.  We start with the following observation.  If an arc $cd$ crosses $ab$, then the arc $cd$ crosses $\{a+1,b\}$ or $a+1 \in \{c,d\}$.  Indeed, if $cd$ croses $ab$, then we may assume (up to renaming) that $c \in (a,b) = [a+1,b)$ and $d \in (b,a) \subset (b,a+1)$.  Thus, either $c = a+1$ or $cd$ crosses $\{a+1,b\}.$

If neither $a$ nor $b$ were incident with $S$, then by Lemma \ref{lemma:Maximal}, we know that $a+1$ and $b+1$ are incident with $S$.  As above, we may thus find a finite set $F'$ such that every arc $cd$ which crosses $\{a+1,b+1\}$ is incident with $F'$.  The set
$F$ from the statement of the proposition is the $F = F' \cup \{a+1,b+1\}$.

The other cases are handled similarly.

\item We start with uniqueness.  Let $axy$ and $ax'y'$ be triangles as in the statement of the proposition.  If $x \not= x'$ we may, without loss of generality, assume that $R(a,x,x')$.  We find a 5-cycle $\langle a, x, x', b, y \rangle$.  The arc $ax'$ intersects the arc $xy$.  This is a contradiction, and hence $x = x'$.  Similarly, we show that $y = y'$, establishing the uniqueness.

We proceed by showing that such a triangle $axy$ exists.  If $a$ is not incident with $S$, then Lemma \ref{lemma:Maximal} yields that $\{a-1, a+1\} \in S$ and hence the triangle $\langle a-1,a,a+1\rangle$ is the requested triangle.

We may thus assume that $a$ is incident with $S$.  Lemma \ref{lemma:Maximal} yields that either $\{b-1,b+1\} \in S$ or $b$ is incident with $S$.  We start by considering the latter case.  Thus, assume that $b$ is incident with $S$ and consider a path $a = l_0, l_1, \ldots, l_n = b$.  We will choose a path such that $n$ is minimal.  Since $ab \not\in S$, we know that $n \geq 2$.  Further, by minimality, we know that $a = l_i$ if and only if $i = 0$.

Since $S$ is a triangulation, we know that $al_1$ lies in two triangles in $S$: $a y l_1$ and $a l_1 z$ (see Proposition \ref{proposition:AtMost2Triangles}).  By totality, we have two possibilities: either $R(a, l_1, b)$ or $R(a,b,l_1)$.  Assume that $R(a,l_1,b)$; we claim that the triangle $a l_1 z$ is the triangle from the statement of the proposition.  We need to verify that $R(l_1,b,z)$.

By totality, we have either $R(l_1,b,z)$ or $R(l_1,z,b)$.  Seeking a contradiction, assume $R(l_1,z,b)$.  In this case, $l_1 \in [a,z]$ and $b \in [z,a]$ so that every path from $l_1$ to $b$ passes either $a$ or $z$ (see Proposition \ref{proposition:CuttingUpS}).  We consider the path $l_1, l_2 \ldots, l_n = b$ from $l_1$ to $b$.  We know that this path does not pass $a$ and hence, there is a $j$ such that $z = l_j$.  From the minimality, we infer that $z=l_1$.  However, since $a l_1 z$ is a triangle, we have $z \not= l_1$.  This is a contradiction, and hence $R(l_1,z,b)$.  We thus have a 4-cycle $\langle a, l_1, b, z \rangle$ as requested.

When $R(a,b,l_1)$ instead of $R(a, l_1, b)$, one proves that the triangle $a y l_1$ is the triangle from the statement of the proposition.  The proof is similar.

The only remaining case is the case where $b$ is not incident with $S$.  In this case, we have $\{b-1,b+1\} \in S$.  If $\{a, b-1\}, \{a,b+1\} \in S$, then the triangle is $\langle a, b-1, b+1\rangle$.  If either $\{a,b-1\} \not\in S$ {or} $\{a,b+1\} \not\in S$ (say the former), we may apply the previous part of the proof and find a triangle $axy$ such that $R(x,b-1,y)$.  It is then readily verified that $R(x,b,y)$ and thus that $axy$ is the requested triangle.
\qedhere
\end{enumerate}
\end{proof}

\begin{proposition}\label{proposition:ExtremalCuttingPoints}
Let $C,R$ be a locally discrete cyclically ordered set and let $S \subseteq \arc(C)$ be a connected triangulation.  Let $\langle a,x,b,y \rangle$ be a 4-cycle in $C$.  If $xy \in S$, then
\begin{enumerate}
\item\label{emumerate:AnticlockwiseTriangle} there is a triangle $xcd$ in $S$ such that $xc$ crosses $ab$ and $xd$ does not cross $ab$, and
\item\label{emumerate:ClockwiseTriangle} there is a triangle $xef$ in $S$ such that $xf$ crosses $ab$ and $xe$ does not cross $ab$.
\end{enumerate}
\end{proposition}

\begin{proof}
We only prove the first statement; the second one is similar.

We consider two cases.  First, assume that $ax \in S$.  Since {$S$} is a triangulation, we know that there is a triangle $xca$.  Let $d = a$; we claim that the triangle $xcd$ is the triangle from the statement of the proposition.  It is clear that $xd = xa$ does not cross $ab$.  We turn our attention to $xc$.  We add $c$ to the 4-cycle $\langle a,x,b,y \rangle$; since $ac$ does not cross $xy$, we find the 5-cycle $\langle a,x,b,y,c \rangle$.  This shows that $ab$ crosses $xc$.

The second case we consider is where $ax \not\in S$.  Here, we can use Proposition \ref{proposition:TowardClusterTilting}(\ref{enumerate:ClosestTriangle}) to find a triangle $xcd$ {in $S$} such that $\langle x,c,a,d \rangle$ is a 4-cycle.  Adding $b$ to this 4-cycle (using the fact that $b \in (x,a)$), gives either the 5-cycle $\langle x,b,c,a,d \rangle$ or the 5-cycle $\langle x,c,b,a,d \rangle$.  However, if $\langle x,c,b,a,d \rangle$ were a 5-cycle, then (using the fact that $y \in (b,a)$) $\langle x,c,b,y,a,d \rangle$ would be a 6-cycle and $xy,cd \in S$ would cross.  We conclude that $\langle x,b,c,a,d \rangle$ is a 5-cycle.  Hence, $xc$ crosses $ab$ and $xd$ does not cross $ab$.
\end{proof}

\begin{definition} \label{definition:RotationMap}
Let $(C,R)$ be a locally discrete cyclically ordered set.  We consider the function $\rho\colon \arc(C) \to \arc(C)$ {defined} by $\rho(\{a,b\}) = \{a-1,b-1\}$.
\end{definition}

\begin{corollary}\label{Corollary:TowardsFunctoriallyFiniteness}
Let $(C,R)$ be a locally discrete cyclically ordered set and let $S \subseteq \arc(C)$ be a connected triangulation.  For any arc $ab \in \arc(C)$, there is a finite set $\{x_i y_i\}_i \subseteq S$ such that every $x_i y_i$ crosses $ab$ and if an arc $xy \in S$ crosses $ab$, then $xy$ crosses $\{\rho(x_i y_i)\}_i$.
\end{corollary}

\begin{proof}
By Theorem \ref{theorem:ConnectedTriangulationsAreMaximal}, we know that $S$ is maximal, and hence Proposition \ref{proposition:TowardClusterTilting}(\ref{enumerate:LocallyFiniteCrossings}) yields the existence a finite set $\{x_i\}_i \subseteq C$ such that if an arc of $S$ crosses $ab$, then that arc is incident with $\{x_i\}_i$.  We choose a minimal such set.

From the minimality, it follows that for each $x_i$, there is a $z_i \in C$ such that $x_i z_i$ crosses $ab$.  We may apply Proposition \ref{proposition:ExtremalCuttingPoints}(\ref{emumerate:ClockwiseTriangle}) to find a triangle $x_i e_i f_i$ such that $x_i f_i$ crosses $ab$ and $x_i e_i$ does not cross $ab$.

We set $x_i y_i = x_i f_i$, and claim that the set $\{x_i y_i\}_i \subseteq S$ is the set from the proposition.  Note that $x_i y_i$ crosses $ab$, as required.  Next, we consider an arc $xy$ which crosses $ab$.  Hence, $\{x_i\}_i \subseteq C$ is incident with $xy$, say $x = x_i$. Possibly after renaming $a$ and $b$, we may assume that there is a 4-cycle $\langle a,x,b,y \rangle$.

If $y = y_i$, then $x_i y_i = xy$ and $xy$ crosses $\rho(x_i y_i)$ as required.  Thus, assume that $y \not= y_i$.  In this case, since $xy$ crosses $ab$, there is either a 5-cycle $\langle a,x,b,y,y_i \rangle$ or $\langle a,x,b,y_i,y \rangle$.  We will exclude the former possibility.  Recall that the arc $x e_i$ did not cross $ab$, implying that $e_i \in (x, b]$.  If $\langle a,x,b,y,y_i \rangle$ were a 5-cycle, then $e_i y_i = e_i f_i \in S$ would cross $xy \in S$, which is a contradiction.  Hence, $\langle a,x,b,y_i,y \rangle$ is a 5-cycle.

We find that $xy$ crosses $\rho(x_iy_i)$, as required.
\end{proof}

\subsection{Exchangeable and obtainable arcs}

We now turn our attention to exchangeable arcs.  Before giving the definition (see Definition \ref{definition:Flips} below), we need the following proposition.

\begin{proposition}\label{proposition:AtMostOneFlip}
Let $S$ be a maximal set of pairwise noncrossing arcs.  For any $ab \in S$, there is at most one maximal set of pairwise noncrossing arcs $T$ such that $S \not= T$ and $S \setminus \{ ab \} \subseteq T$.  Furthermore, for such a $T$, we have $S \setminus T = \{ab\}$ and $|T \setminus S| = 1$.
\end{proposition}

\begin{proof}
Let $S' = S \setminus \{ab\}$.  Let $x_1 y_1$ and $x_2 y_2$ be two arcs which are not in $S$ and which do not cross any arcs in $S'$.  By the maximality of $S$, the arcs $x_1 y_1$ and $x_2 y_2$ need to cross $ab$.  Up to renaming, we may assume (see Lemma \ref{lemma:Noncrossing}) that $x_1,x_2 \in (a,b)$ and $y_1,y_2 \in (b,a)$.  
We want to show that $x_1 = x_2$ and $y_1 = y_2$ so that $x_1 y_1 = x_2 y_2$.  Seeking a contradiction, assume that $x_1 \not= x_2$.  By totality, we may then assume that (up to renaming) $R(a,x_1,x_2)$.  This gives a 5-cycle $\langle a,x_1,x_2,b,y_2 \rangle$.

Hence, $a$ and $x_2$ are not neighbors and we have $ax_2 \in \arc (C)$.  Since $ax_2$ crosses $x_1 y_1$, we know that $a x_2 \not\in S'$.  The maximality of $S$ implies that there is an arc $cd \in S$ such that $cd$ and $a x_2$ cross.  Using the fact that $cd$ does not cross $ab$, we may, up to renaming, assume that there is 6-cycle $\langle a,c,x_2,d,b,y_2 \rangle$ (or possibly a 5-cycle $\langle a,c,x_2,d=b,y_2 \rangle$).

However, this implies that $cd$ crosses $x_2 y_2$ --- a contradiction since $cd \in S'$ and $x_2 y_2$ does not cross any arcs in $S'$.  We obtain the required contradiction and may conclude that $x_1 = x_2$, as required.

Similarly, one shows that $y_1 = y_2$ and hence $x_1 y_1 = x_2 y_2$.  We conclude that there is at most one arc (different from $ab$) that we can add to $S'$ and obtain a set of pairwise noncrossing arcs.  This proves the statement.
\end{proof}

\begin{definition}\label{definition:Flips}
Let $S$ be a maximal set of pairwise noncrossing arcs, and let $ab \in S$ be an arc.  If there exists a $T$ as in Proposition \ref{proposition:AtMostOneFlip}, then we say that $ab$ is \emph{exchangeable} in $S$; we write $\mu_{ab} S$ for $T$.  We say that $T$ is obtained \emph{by exchanging} $ab$ and the arc in $T \setminus S$ is called the \emph{flip} of $ab$.
\end{definition}

\begin{example}\label{example:Flipping}
The following table indicates the cases of Example \ref{example:NoncrossingSets} where all arcs are exchangeable.  Since Definition \ref{definition:Flips} only makes sense for maximal sets, we restrict to those.
\begin{center}
\begin{tabular}{ r | c  c  c  c  c  c }
{} & 4 & 5 & 6 & 9 & 10 & 11 \\
\hline
Connected & $\times$ & $\times$ & $\times$ & \checkmark & \checkmark & \checkmark \\
Maximal & \checkmark & \checkmark & \checkmark & \checkmark & \checkmark & \checkmark \\
Triangulation & $\times$ & \checkmark & \checkmark & $\times$ & \checkmark & \checkmark \\
Locally finite & $\times$ & \checkmark & $\times$ & $\times$ & $\times$ & \checkmark \\ 
All arcs exchangeable & \checkmark & \checkmark & \checkmark & $\times$ & \checkmark & \checkmark \\
\end{tabular}
\end{center}
\end{example}

\begin{proposition}\label{proposition:FlipsAndTriangles}
Let $S$ be a maximal set of pairwise noncrossing arcs.  An arc $ab$ is exchangeable in $S$ if and only if there is a 4-gon $axby$ in $S$.  In this case, the flip of $ab$ is $xy$.
\end{proposition}

\begin{proof}
If there is a 4-gon $axby$ in $S$, then it is readily verified that $\mu_{ab} S = (S \setminus \{ab\}) \cup \{xy\}$.  For the other direction, let $ab$ be an exchangeable arc in $S$ and let $xy$ be the flip.  Since $xy$ crosses $ab$, we may, up to possibly switching the role of $x$ and $y$, assume that $axby$ is a 4-cycle in {$C$}.  We wish to show that $axby$ is a 4-gon {in $S$}.

We will only show that $ax$ is either an edge or an arc in $S$; the other cases are similar. Seeking a contradiction, assume that $ax$ is neither an edge in $C$ nor an arc in $S$, thus $ax \in \arc(C) \setminus S$.  By maximality of $S$, the arc $ax$ crosses an arc $cd \in S$.  To fix notation, we will choose $c,d$ such that $R(a,c,d)$, so that there is a 4-cycle $\langle a,c,x,d \rangle$ (Lemma~\ref{lemma:Noncrossing}).  As $cd$ does not cross $xy$, we can use Lemma~\ref{lemma:Noncrossing} together with $x \in (c,d)$ to see that $y \in [c,d]$.  We may exclude $y = c$ (as $axby$ is a 4-gon), so there is either a 5-gon $\langle a,c,x,y,d \rangle$, or $y=d$ and hence a 4-gon $\langle a,c,x,y=d \rangle$.  Again using the fact that $axby$ is a 4-gon, we find either $\langle a,c,x,b,y,d \rangle$ or $\langle a,c,x,b,y=d \rangle$.  This shows that $ab$ crosses $cd$: a contradiction.  We have established that $ax$ is either an edge or an arc in $S$.
\end{proof}

\begin{remark}\label{remark:FlippingAlternateDefinition}
Based on Proposition \ref{proposition:FlipsAndTriangles}, one could remove the maximality condition in Definition \ref{definition:Flips} and define an exchangeable arc in a set $S$ of pairwise noncrossing arcs to be an arc $ab$ lying in a 4-gon $axby$ in $S$.  The flip would then be $xy$.  In this paper, we require maximality since it is closer to how mutation is defined in cluster categories.
\end{remark}

\begin{corollary}\label{corollary:FlipsAndTriangles}
Let $S$ be a maximal set of pairwise noncrossing arcs.  An arc $ab$ is exchangeable if and only if $ab$ lies in two triangles.
\end{corollary}

\begin{corollary}\label{corollary:Flipping}
Let $S$ be a maximal set of pairwise noncrossing arcs.  If $S$ is a triangulation, then every arc is exchangeable.
\end{corollary}

\begin{proposition}
Let $S$ be a maximal set of pairwise noncrossing arcs and let $ab \in S$ be an exchangeable arc.  If the flip of $ab$ with respect to $S$ is $xy$, then $xy$ is exchangeable in $\mu_{ab} S$ and the flip of $xy$ with respect to $\mu_{ab} S$ is $ab$.
\end{proposition}

\begin{proof}
This follows directly from Proposition \ref{proposition:FlipsAndTriangles}.
\end{proof}

As can be seen from Example \ref{example:Flipping}, the converse of Corollary \ref{corollary:Flipping} does not hold.  Indeed, while it follows from Proposition \ref{proposition:FlipsAndTriangles} that every arc lies in two triangles, it is not true that every edge lies in a triangle.  The next result indicates that the converse does hold, if one additionally requires the set $S$ to be connected.

\begin{theorem}\label{theorem:FlippingIsTriangulated}
Let $S$ be a set of pairwise noncrossing arcs.  If $S$ is connected, then the following are equivalent:
\begin{enumerate}
\item $S$ is a triangulation,
\item $S$ is maximal and every arc is exchangeable.
\end{enumerate}
\end{theorem}

\begin{proof}
Assume that $S$ is a connected triangulation.  We know from Theorem \ref{theorem:ConnectedTriangulationsAreMaximal} that $S$ is maximal, and from Proposition \ref{proposition:FlipsAndTriangles} that every arc of $S$ is exchangeable.

For the other direction, assume that $S$ is connected, maximal, and that every arc is exchangeable; we need to show that $S$ is a triangulation. It follows from Proposition \ref{proposition:FlipsAndTriangles} that every arc of $S$ lies in two triangles.  Let $xy$ be any edge, say $y = x+1$.

By Lemma \ref{lemma:Maximal}, we know that either $x$ and $x+1$ are both incident with $S$, or $S$ contains one of the arcs $\{x+1, x-1\}$ or $\{x, x+2 \}$.  In the latter case $\{x,x+1\}$ lies in a triangle, so we may assume the former, {i.e. both $x$ and $x+1$ are incident with $S$.}

Since $S$ is connected, we know that there is a path $x = x_0, x_1, \ldots, x_n = x+1$.  We will assume that the length $n$ of the path is minimal.  By Proposition~\ref{proposition:FlipsAndTriangles}, we know that the arc $\{x,x_1\}$ lies in two triangles and by Proposition \ref{proposition:AtMost2Triangles}, we know that there is a triangle $xcx_1$, thus $x_1 \in (c,x)$.  Note that $x+1 \in (x,c]$.

It follows from Proposition \ref{proposition:CuttingUpS} that the path $x_1, x_2, \ldots, x_n = x+1$ passes either $x$ or $c$.  Minimality of $n$ implies that the path passes $c$, and hence $c = x_k$ for some $1 < k \leq n$.  We now have a path $a = x_0, x_k, x_{k+1}, \ldots, x_n = x+1$.  The minimality of $n$ implies that $k=n$ and hence there is a triangle $\langle x, x+1, x_1 \rangle$.  This finishes the proof.
\end{proof}

\begin{proposition}
Let $S$ be a maximal set of noncrossing arcs.  Let $ab \in S$ be an exchangeable arc.
\begin{enumerate}
\item If $S$ is a triangulation, so is $\mu_{ab} S$.
\item If $S$ is connected, so is $\mu_{ab} S$.
\item If $cd \in S$ is exchangeable and $cd \not= ab$, then $cd \in \mu_{ab} S$ is exchangeable.
\end{enumerate}
\end{proposition}

\begin{proof} We write $T$ for $\mu_{ab} S$.  Let $xy$ be the flip of $ab$.  Note that there is a 4-gon $axby$ in $S$ and $T$.
\begin{enumerate}
\item Using Proposition \ref{proposition:CuttingUpS}(\ref{enumerate:IyamaYoshinoTriangulated}) four times, we find the following statement: $T$ is a triangulation for $C$ if and only if $T|_{[a,x]}$ is a triangulation for $[a,x]$, $T|_{[x,b]}$ is a triangulation for $[x,b]$, $T|_{[b,y]}$ is a triangulation for $[b,y]$, $T|_{[y,a]}$ is a triangulation for $[y,a]$, and $\{xy\}$ is a triangulation for $axby$.

Since $T|_{[a,x]} = S|_{[a,x]}$, $T|_{[x,b]} = S|_{[x,b]}$, $T|_{[b,y]} = S|_{[b,y]}$, and $T|_{[y,a]} = S|_{[y,a]}$, the first four properties follow from Proposition \ref{proposition:CuttingUpS}(\ref{enumerate:IyamaYoshinoTriangulated}).  The last property is trivial.

\item {Recall from Proposition \ref{proposition:FlipsAndTriangles} that there is a 4-gon $axby$ where $xy$ is the flip of $ab$.  Note that the statement is trivial if $C = \{a,b,x,y\}$, so we may assume that at least one of $ax,xb,by,ya \in \arc(C)$ (and hence also in $S$).  Without loss of generality, assume that $ax \in S$ and hence $x$ is incident with $S \setminus \{ab\}$.

First, assume that $S \setminus \{ab\}$ is connected.  To check that $\mu_{ab} S = (S \setminus \{ab\}) \cup \{xy\}$ is connected, we need only check that, for any $c \in C$ incident with $S \setminus \{ab\}$, there is a path from $c$ to $y$.  We obtain such a path from a path from $c$ to $x$ and the arc $xy$.

Next, assume that $S \setminus \{ab\}$ is not connected.  Let $c,d \in C$ incident with $S \setminus \{ab\}$ such that there is no path from $c$ to $d$ in $S \setminus \{ab\}$.  Let $c = x_0, x_1, \ldots, x_n = d$ be a path in $S$ (recall that $S$ is connected), and assume that the path has minimal length.  Since this path requires the arc $ab$, we know that $a$ and $b$ occur in adjacent positions in this path.

We want to replace $a,b$ with either $a,x,b$ or $a,x,y,b$.  This is possible if either $xb \in S$ or $by \in S$.  We can thus finish the proof by contradiction: assume that $xb$ and $by$ are edges, so that $[x,y] = \{b-1,b,b+1\}$.  As $axby$ is a 4-gon in $S$ (and hence also $S \setminus \{ab\}$), we may infer that $b$ is not incident with $S \setminus \{ab\}$.  Indeed, any $z \in C$ such that $bz \in S \setminus \{ab\}$ needs to satisfy $z \in (a,x)$ (but then $bz$ crosses $ax$), $z \in (y,a)$ (impossible if $ya$ is an edge, and $bz$ would cross the arc $ya$ otherwise), or $z \in [x,y]$ (which can be excluded as $[x,y] = \{b-1,b,b+1\}$).  Hence, $b$ can only occur in a minimal path from $c$ to $d$ if $b=c$ or $b=d$ (as the only successor and only predecessor of $b$ in this path is $a$).  As $b$ is then not incident with $S \setminus \{ab\}$, we may exclude $b=c$ and $b=d$, completing the proof.
}

\item Since the arc $cd \in S$ is exchangeable, it follows from Proposition \ref{proposition:FlipsAndTriangles} that there is a 4-gon $cudv$ in $S$.  This is a 4-gon in $T$, except if $ab$ is one of the edges of the 4-gon.

Assume first that $a = c$ and $b = u$.  Note that $ab$ lies in the triangles $aby$ and $abd = cud$.  Proposition \ref{proposition:AtMost2Triangles} implies that $y=d$.  There is thus a 4-gon $cxdv = axyv$ in $T$.  This implies that $cd$ is exhangeable in $T$.  For the case where $a = u$ and $b = c$, note that there are triangles $axb$ and $udc = adb$ so that Proposition \ref{proposition:AtMost2Triangles} implies that $x=d$.  We find the 4-gon $cyxv$ in $T$.  The other cases are similar. \qedhere
\end{enumerate}
\end{proof}

\begin{corollary}
Let $(C,R)$ be a cyclic order and $S \subseteq \arc(C)$ be a set of pairwise noncrossing arcs.  If $S$ is a connected triangulation, then for every $ab \in S$, $\mu_{ab} S$ exists and is a connected triangulation.
\end{corollary}
 
\begin{definition}
Let $S \subseteq \arc(C)$ be a maximal triangulation.  An arc $ab \in \arc(C)$ is \emph{obtainable} from $S$ if there is a sequence $S = S_0, S_1, \ldots, S_n$ of maximal triangulations such that $|S_i \setminus S_{i+1}| = 1$ (for all $i \in \{0,1, \ldots, n-1\}$) and $ab \in S_n$.
\end{definition}

\begin{proposition}
Let $S \subseteq \arc(C)$ be a maximal triangulation.  An arc $ab \in \arc(C)$ is obtainable if and only if $ab$ crosses only finitely many arcs in $S$.
\end{proposition}

\begin{proof}
Since $S_i$ and $S_{i+1}$ differ in at most one arc, we know that $S_0$ and $S_n$ differ in at most $n$ arcs.  If $ab$ crosses infinitely many arcs in $S$, then $ab$ will cross infinitely many arcs in $S_n$.  Hence, if $ab$ is obtainable from $S$ then $ab$ crosses only finitely many arcs in $S$.

For the other implication, assume that $ab$ crosses only finitely many arcs in $S$.  Let $n$ be the number of arcs of $S$ crossing $ab$ and let $S = S_0$.  Let $axy$ be the triangle from Proposition \ref{proposition:TowardClusterTilting}(\ref{enumerate:ClosestTriangle}).  Since $S_0$ is a maximal triangulation, it follows from Proposition \ref{corollary:Flipping} that the arc $xy \in S_0$ is exchangeable.  Let $S_1 = \mu_{xy} S_0$.  The arc $ab$ now crosses $n-1$ arcs in $S_1$.  Iterating this procedure gives a sequence $S = S_0, S_1, \ldots, S_n$ where $ab$ crosses no arcs in $S_n$.  Maximality of $S_n$ then implies that $ab \in S_n$, as required.  
\end{proof}

\subsection{Locally finite triangulations}

\begin{theorem}\label{theorem:LocallyFiniteTriangulation}
Let $S$ be a set of pairwise noncrossing arcs.  If $S$ is locally finite, then $S$ is a triangulation if and only if $S$ is maximal.
\end{theorem}

\begin{proof}
Assume that $S$ is maximal.  We will start by showing that every edge in $C$ lies in a triangle of $S$.  Let $\{a,a+1\}$ be any edge.  Since $S$ is locally finite, we may consider the $(n+1)$-cycle $\langle a, a_1, \ldots, a_n\rangle$ consisting of $a$ and all elements $a_i \in C$ such that $a a_i \in S$.

We will work in the interval $[a,a_1]$.  Note that $a+1 \in (a,a_1)$.  Since $S|_{[a,a_1]}$ is a maximal set of pairwise noncrossing arcs on $[a,a_1]$ (by Proposition \ref{proposition:CuttingUpS} (\ref{enumerate:IyamaYoshinoMaximal})) and since $a$ is not incident with any arcs of $S|_{[a,a_1]}$ (by construction), we may use Lemma \ref{lemma:Maximal} to see that there is an arc $\{a+1,a_1\} \in S|_{[a,a_1]} \subseteq S$.  This gives a triangle $\langle a, a+1, a_1 \rangle$, and hence the edge $\{a,a+1\}$ lies in a triangle in $S$.

Next, let $ab \in S$; we want to show that there are triangles $axb$ and $aby$ in $S$.  We know that $S|_{[a,b]}$ and $S|_{[b,a]}$ are maximal (by Proposition \ref{proposition:CuttingUpS} (\ref{enumerate:IyamaYoshinoMaximal})) and locally finite (easy) on $[a,b]$ and $[b,a]$, respectively.  Furthermore, $ab$ is an edge in both $[a,b]$ and in $[b,a]$.  We have previously established that $ab$ lies in one triangle in $S|_{[a,b]}$ and in one triangle in $S|_{[b,a]}$.  This shows that $ab$ lies in two triangles in $S$.

This establishes that $S$ is a triangulation.

For the other direction, assume that $S$ is a triangulation.  Let $ab \in \arc(C)$ be any arc which does not cross $S$.  We want to show that $ab \in S$.  Since $ab \in \arc(C)$, we know that $a$ and $b$ are not neighbors.

By Lemma \ref{lemma:Maximal}, we know that either $a$ is incident with $S$ or $\{a-1, a+1\} \in S$.  Since $\{a-1, a+1\}$ would cross $ab$, we may exclude this case and assume that $a$ is incident with $S$.  Consider the $(n+1)$-cycle $\langle a, a_1, \ldots, a_n\rangle$ consisting of $a$ and all elements $a_i \in C$ such that $a a_i \in S$.  Seeking a contradiction, assume that there is no arc $ab \in S$, thus that $b \not= a_i$ for any $i$.  We then extend this $(n+1)$-cycle by adding $b$: $\langle a, a_1, \ldots, a_k, b, a_{k+1}, \ldots, a_n \rangle$.  Since $a$ is incident with $S$, we know that $n \not= 0$.

If $k = 0$,  then we start with the $(n+1)$-cycle: $\langle a, b, a_{1}, \ldots, a_n \rangle$.  Since $S$ is {a triangulation}, the arc $aa_1$ is contained in a triangle $axa_1$.  However, by construction of the $(n+1)$-cycle $\langle a, a_1, \ldots, a_n\rangle$, we know that there is no arc $ax \in S$.  Hence, $ax$ is an edge in $C$ (thus $x = a+1$) and there is an arc $xa_1 \in S$.  This arc crosses the arc $ab$, which is the required contradiction.

If $k \not= 0$, then the $(n+1)$-cycle is the $(n+1)$-cycle: $\langle a, a_1, \ldots, a_k, b, a_{k+1}, \ldots, a_n \rangle$ and the contradiction is obtained  by considering a triangle $a a_k y$.

We conclude that $ab \in S$, and hence that $S$ is maximal.
\end{proof}

{Locally finite triangulations provide a combinatorial model for locally bounded cluster-tilting subcategories in \S\ref{section:CominatoricsOfTypeA}. We will need the following proposition to establish this connection.}

\begin{proposition}\label{proposition:LocallyFiniteCrossings}
Let $S \subseteq \arc C$ be a set of noncrossing arcs.
\begin{enumerate}
\item If every arc $ab \in \arc(C)$ crosses only finitely many arcs in $S$, then $C$ is locally finite,
\item if $S$ is maximal and connected, then $C$ being locally finite implies that every arc $ab \in \arc(C)$ crosses only finitely many arcs in $S$.
\end{enumerate}
\end{proposition}

\begin{proof}
For the first statement, assume that every arc $ab \in C$ crosses finitely many arcs in $S$.  Since for each $l \in \LL$ the arc $\{l-1,l+1\}$ crosses only finitely many arcs, we infer that $l$ can only be incident to finitely many arcs.  Hence, $S$ is locally finite.

The second statement follows directly from Proposition \ref{proposition:TowardClusterTilting}(\ref{enumerate:LocallyFiniteCrossings}).
\end{proof}

{A basic question is, for which locally discrete cyclically ordered sets $(C,R)$ does a locally finite connected triangulation even exist?  This problem is closely related to one of our main results (Theorem~\ref{theorem:IntroductionCountableReachable}, see also Theorems~\ref{theorem:LocallyFiniteClusterTilting} and~\ref{theorem:ExistenceOfCCMap} below).  The rest of the section is devoted to providing a complete answer to this question.}

\begin{construction}\label{Construction:ConstructionLocallyFinite}
{Suppose that $C$ is countable, and fix a bijection $\varphi: \bN \to C$}.  We will construct a set of pairwise noncrossing arcs $S \subseteq \arc(C)$ inductively.

We start with $S_0 = \emptyset$.  For each $n > 0$, we set
\[S_n = S_{n-1} \cup \left\{ \varphi(i)\varphi(n) \in \arc(C) \left| \begin{array}{ll}\mbox{$i < n$, and} \\ \mbox{$S_{n-1} \cup \{ \varphi(i) \varphi(n) \}$ is a set of pairwise noncrossing arcs} \end{array} \right. \right\}.\]

We then set $S = \cup_{i \in \bN} S_i$.
\end{construction}

\begin{remark}
Since each $S_n$ is a set of pairwise noncrossing arcs, so is $S$.
\end{remark}

\begin{lemma}\label{lemma:CreatingLocallyFinite}
Let $\langle a_0, a_1, \ldots,a_{n-1}, a_n \rangle$ be the $(n+1)$-cycle given by $\varphi(\{0,1, \ldots, n\})$ for $n \geq 3$ and where $a_n = \varphi(n)$.  We have $a_0 a_{n-1} \in S_{n-1}.$
\end{lemma}

\begin{proof}
Note that $a_0$ and $a_{n-1}$ are not neighbors in $C$ so that $a_1 a_{n-1} \in \arc(C)$.  Consider $j = \max\{ \varphi^{-1}(a_0), \varphi^{-1}(a_{n-1}) \}$.  The arcs in $S_{j-1}$ are only incident with elements from $\varphi(\{0,1, \ldots, j-1\}) \subset C$.  Since $(a_{n-1},a_0) \cap \varphi(\{0,1, \ldots, j-1\}) = \emptyset$, we know by Lemma \ref{lemma:Noncrossing} that the arc $a_0 a_{n-1}$ does not cross any arcs in $S_{j-1}$ and thus $a_0 a_{n-1} \in {S_{j}} \subseteq S_{n-1}$.
\end{proof}

\begin{proposition}\label{proposition:CountableLocallyFiniteTriangulation}
{Let $S \subseteq \arc(C)$ be an output from Construction~\ref{proposition:LocallyFiniteCrossings}. Then:}
\begin{enumerate}
\item $S$ is maximal,
\item $S$ is connected,
\item $S$ is locally finite,
\item $S$ is {a triangulation}.
\end{enumerate}
\end{proposition}

\begin{proof}
\begin{enumerate}
\item Let $\varphi(i) \varphi(j) \in \arc(C)$.  Assume without loss of generality that $i < j$.  If $\varphi(i) \varphi(j)$ does not cross any arcs in $S_{j-1} \subseteq S$, then $\varphi(i) \varphi(j) \in S_j \subseteq S$.  Thus, either $\varphi(i) \varphi(j)$ {crosses an arc in $S$} or $\varphi(i) \varphi(j) \in S$.  Hence, $S$ is maximal.

\item We will prove that $S_n$ is connected by induction on $n$.  From this it follows that {$S = \cup_n S_n$} is connected.  Since $S_0 = \emptyset$ and $S_1 = \{\varphi(0), \varphi(1)\}$, we know that the statement is true for $n \leq 1$.

For $n = 2$, we either have that $S_2 = S_1$ and hence {$S_2$} is connected, or all arcs in $S_2 \setminus S_1$ are incident with either $\varphi(0)$ or $\varphi(1)$.  We easily infer that $S_2$ is connected.

Let $n \geq 3$.  We consider the $(n+1)$-cycle given by $\varphi(\{0,1, \ldots, n\}):$
$$\langle a_0, a_1, \ldots,a_{n-1}, a_n \rangle.$$
Up to rotation, we may assume that $\varphi(n) = a_n$.  By Lemma \ref{lemma:CreatingLocallyFinite}, we know that $a_0 a_{n-1} \in S_{n-1}$.  Hence, any arc in $S_n \setminus S_{n-1}$ is incident with either $a_0$ or $a_{n-1}$.  Since $S_{n-1}$ is connected, so is $S_n$. 

\item We claim that, for each $n \in \bN$, there is an $N \in \bN$ such that all arcs in $S$ incident with $\varphi(n)$ lie in $S_N$.  Let {$N = \max\{\varphi^{-1}(\varphi(n)-1), \varphi^{-1}(\varphi(n)+1)\}$}. Thus, $\varphi(n) \pm 1 \in \varphi(\{0,1, \ldots, N\})$.

Seeking a contradiction, let $k > N$ (note that this implies that $k \geq 4$) be the smallest natural number such that $S_k \setminus S_{k-1}$ has an arc incident with $\varphi(n)$.  Consider the $(k+1)$-cycle given by $\varphi(\{0,1, \ldots, k\}):$
$$\langle a_0, a_1, \ldots,a_{k-1}, a_k \rangle.$$
Up to rotation, we may assume that $\varphi(k) = a_k$.  By Lemma \ref{lemma:CreatingLocallyFinite}, we know that $a_0 a_{k-1} \in S_{n-1}$, so that $S_k \setminus S_{k-1} \subseteq \{a_0 a_k, a_{k-1} a_k\}$.
{Since $\varphi(n)-1, \varphi(n), \varphi(n)+1 \in \{a_0, a_1, \ldots,a_{k-1}\}$, there exists $1 \le i \le k-2$ such that $a_{i-1} = \varphi(n)-1$, $a_i = \varphi(n)$ and $a_{i+1} = \varphi(n)+1$.  Hence, $S_k \setminus S_{k-1}$ has no arcs incident with $a_i = \varphi(n)$ and we have established our initial claim.}

Since $S_N$ is finite, this proves that $S$ is locally finite.

\item This follows from Theorem \ref{theorem:LocallyFiniteTriangulation}.
\end{enumerate}
\end{proof}

{
\begin{figure}
  \centering
  \begin{tikzpicture}
    \TikzLocFinTriangle{0}{0}
  \end{tikzpicture}

\caption{A locally finite connected triangulation for $(C,R) = (\LL, \leq)_{\rm cyc}$, where $\LL$ is the linearly ordered set $\{0,1,2\} \stackrel{\rightarrow}{\times} \bZ$.}
\label{figure:LocallyFiniteTriangle}
\end{figure}
}

Now we can characterize the situation when a locally finite connected triangulation exists (see Figure~\ref{figure:LocallyFiniteTriangle} for an example of such a triangulation).

\begin{theorem}\label{theorem:CountableLocallyFiniteTriangulation}
Let $(C,R)$ be a locally discrete cyclically ordered set.  There is a locally finite connected triangulation of $(C,R)$ if and only if $C$ is countable.
\end{theorem}

\begin{proof}
First, assume that there is a locally finite connected triangulation $S$ of $C,R$.  Lemma \ref{lemma:Maximal} shows that every element of $C$ is either incident with an arc of $S$, or connected by an edge of $C$ to such an element.  Consider a graph whose vertices are the elements of $C$ and whose edges are given by the arcs of $S$ and the edges of $C$.  Since this graph is locally finite and connected, we infer that $C$ is countable.

The other direction follows from Proposition \ref{proposition:CountableLocallyFiniteTriangulation}.
\end{proof}
\section{\texorpdfstring{Cluster combinatorics of type $A$}{Cluster combinatorics of type A}}\label{section:CominatoricsOfTypeA}

In this section, we will discuss a class of examples of triangle-free cluster categories---an infinite version of cluster categories of Dynkin type~$A$. The combinatorics of cluster-tilting subcategories are reminiscent of the finite case studied by Caldero et al.\ in~\cite{CalderoChapotonSchiffler06}, but as was discussed in the last section, one should be careful as infinite triangulations bring new and sometimes peculiar phenomena. Our class of examples includes in particular the $\infty$-gon and its triangulations (studied by Holm and J\o{}rgensen~\cite{HolmJorgensen12}) and infinite Dynkin quivers of type $\mathbb{A}^\infty_\infty$ and triangulations of the infinite strip (studied by Liu and Paquette~\cite{LiuPaquette15}).

\subsection{Representations of linearly ordered sets}

Let $\PP$ be a linearly ordered set, and let $\LL = \PP \stackrel{\rightarrow}{\times} \bZ$ be the lexicographical product. That is, $\LL$ is an unbounded locally discrete linearly ordered set (see \S\ref{subsection:Lin ord sets}). We interpret $\LL$ as a poset category in the usual way. As before, we write $\Rep \LL$ for the category $\Fun(\LL^\circ, \Mod k)$ of contravariant functors from $\LL$ to $\Mod k$, and $\rep \LL$ for the full subcategory of finitely presented objects.  We write $\repc \LL$ for the full subcategory of $\Rep \LL$ consisting of those objects which are finitely presented and finitely copresented in $\Rep \LL$.

\begin{proposition}
The category $\repc \LL$ is a hereditary abelian category with Serre duality.
\end{proposition}

\begin{proof}
See \cite{Ringel02} or \cite[Proposition 4.6]{vanRoosmalen06}.
\end{proof}

\begin{remark}
The category $\rep \LL$ has nonzero projective objects, but no nonzero injective objects; it does not satisfy Serre duality.  This is the reason we consider the category $\repc \LL$.
\end{remark}

We will recall some features of the category $\repc \LL$.

\begin{proposition}\label{proposition:ClassifyingRepcLL}
Every indecomposable object in $\repc \LL$ is of the form
$$X_{a,b} = \im (k\LL(-,b) \to k\LL(a,-)^*),$$
where $a \leq b$.  Furthermore, $\repc \LL$ has no nonzero projective or injective objects.
\end{proposition}

\begin{proof}
The first statement follows from \cite[Lemma 4.5]{vanRoosmalen06}.  Since each object $X_{a,b}$ admits a nonsplit epimorphism $X_{a-1,b} \to X_{a,b}$ and a nonsplit monomorphism $X_{a,b} \to X_{a,b+1}$, the second statement follows.
\end{proof}

\begin{remark}\label{remark:Resolutions}
The indecomposable object $X_{a,b}$ has a projective resolution
$$0 \to k\LL(-,a-1) \to k\LL(-,b) \to X_{a,b} \to 0$$
and an injective resolution
$$0 \to X_{a,b} \to k\LL(a,-)^* \to k\LL(b+1,-)^* \to 0$$
in $\Rep \LL$.
\end{remark}

\begin{proposition}\label{proposition:TauIsShifting}
The category $\repc \LL$ has Auslander-Reiten triangles.  The Auslander-Reiten translation is given by $\tau X_{a,b} \cong X_{a-1,b-1}$.
\end{proposition}

\begin{proof}
See \cite[Proposition 1]{Ringel02} or \cite[Proposition 4.6]{vanRoosmalen06}.
\end{proof}

We recall the following result from \cite[Final Remark]{Ringel02} (see also \cite{vanRoosmalen06}).

\begin{proposition} \label{proposition:DerivedEqPosets}
Let $\LL$ and $\LL'$ be linearly ordered sets.  The categories $\repc \LL$ and $\repc \LL'$ are derived equivalent if and only if $\LL$ and $\LL'$ are rotationally equivalent (see Definition \ref{definition:RotationallyEquivalent}).
\end{proposition}

\subsection{\texorpdfstring{The associated cluster category $\CC_\LL$}{The associated cluster category}}

In this subsection, we will give some properties of the cluster category $\CC_\LL$ associated to the category $\repc \LL$, thus $\CC_\LL \cong (\Db \repc \LL) / \bS[-2]$.

\begin{remark}
If $\LL$ is a bounded linearly ordered locally discrete set, then $\repc \LL \cong \rep \LL$.  For such a linearly ordered set $\LL$, we have $\LL \cong \bN \cdot (\PP \stackrel{\rightarrow}{\times} \bZ) \cdot (-\bN)$, and thus
\[\Db(\repc \LL) \simeq \Db(\mod kQ),\]
where $Q$ is a thread quiver with a single thread arrow $\xymatrix@1{x \ar@{..>}[r]^-{\PP'} & y}$.  Hence, $\CC_\LL$ has a directed cluster-tilting subcategory (this uses Proposition \ref{proposition:CTFromEnoughProj}).  Since any unbounded linearly ordered locally discrete set $\LL'$ is rotationally equivalent to a bounded one, it follows from Proposition \ref{proposition:DerivedEqPosets} that $\CC_\LL$ has a directed cluster-tilting subcategory as well (see also \cite[Example 8.4]{BergVanRoosmalen14} or \cite[Remark 7.8]{vanRoosmalen06}).
\end{remark}

The following proposition holds since $\repc \LL$ has no nonzero projective or injective objects (see Proposition~\ref{proposition:Cluster without proj-inj}).

\begin{proposition}\label{proposition:DescriptionCCLL}
We have
\begin{eqnarray*}
\Ob \CC_\LL &=& \repc \LL \\
\Hom_{\CC_\LL}(X,Y) &\cong& \Hom_{\repc \LL}(X,Y) \oplus \Ext_{\repc \LL}^1(X, \tau^{-} Y) \\
\Ext^1_{\CC_\LL}(X,Y) &\cong& \Ext_{\repc \LL}^1(X,Y) \oplus \Ext_{\repc \LL}^1(Y,X)^*.
\end{eqnarray*}
\end{proposition}

As an easy consequence, one obtains that Hom and Ext spaces between indecomposables are at most one-dimensional.

\begin{proposition}\label{proposition:DimensionsInCCLL}
For indecomposable $X,Y \in \ind \CC_\LL$, we have
\begin{enumerate}
\item $\Hom_{\CC_\LL}(X,X) \cong k$ and $\Ext^1_{\CC_\LL}(X,X) = 0$, and
\item $\dim \Hom_{\CC_\LL}(X,Y) \leq 1.$
\end{enumerate}
\end{proposition}

\begin{proof}
This follows from Proposition \ref{proposition:DescriptionCCLL} by a direct computation. Indeed, $\Hom_{\repc \LL}(X_{a,b},X_{c,d})$ is one-dimensional if $a \le c \le b \le d$ in $\LL$ and vanishes otherwise. Similarly, since $\Hom_{\repc \LL}(X,Y) \cong \Ext^1_{\repc \LL}(Y,\tau X)^*$ (see~\S\ref{subsection:Serre and CY}), the dimension of $\Ext^1_{\repc \LL}(X_{a,b},X_{c,d})$ is $1$ if $c+1 \le a \le d+1 \le b$ and vanishes otherwise.
\end{proof}

Similar to the case of $A_n$, the cluster category $\CC_\LL$ can be described using arcs in $(\LL, \leq)_{\rm cyc}$.  
For the next proposition, recall the rotation map $\rho\colon \arc(C) \to \arc(C)$ from Definition~\ref{definition:RotationMap}.

\begin{proposition}\label{proposition:ObjectsArcs}
Let $\LL$ be an unbounded linearly ordered locally discrete set.  The map
\begin{eqnarray*}
\Phi: \arc(\LL) &\to& \ind \CC_\LL \\
ab &\mapsto& \mbox{$X_{a+1,b-1}$ ($a \leq b$ in $\LL$)}
\end{eqnarray*}
is a bijection.  Furthermore,
\begin{enumerate}
\item\label{enumerate:CrossingArcsMeansExt} the arcs $ab$ and $cd$ intersect if and only if
$$\Ext^1_{\CC_\LL}(\Phi(ab), \Phi(cd)) \cong \Ext^1_{\CC_\LL}(\Phi(cd), \Phi(ab))^* \not= 0,$$
\item\label{enumerate:TauMeansRho} $\tau \Phi(ab) \cong \Phi(\rho(ab))$.
\end{enumerate}
\end{proposition}

\begin{proof}
The map $\Phi$ is the composition of the bijections

\begin{center}
\begin{tabular}{ccccc}
$\arc \LL$ & $\to$ & $\ind (\repc \LL)$ & $\to$ & $\ind \CC_\LL$ \\
$ab$ &$\mapsto$ & $X_{a+1,b-1}$ & $\mapsto$ & $X_{a+1,b-1}$
\end{tabular}
\end{center}

\noindent
of Propositions \ref{proposition:ClassifyingRepcLL} and \ref{proposition:DescriptionCCLL}. Furthermore, (\ref{enumerate:CrossingArcsMeansExt}) can be verified directly by the same computation as in the proof of Proposition~\ref{proposition:DimensionsInCCLL}, while (\ref{enumerate:TauMeansRho}) follows from Proposition \ref{proposition:TauIsShifting}.
\end{proof}

\begin{corollary}\label{corollary:MorphismsArcs}
Let $ab, ac \in \arc(\LL)$ be two different arcs.  The following are equivalent:
\begin{enumerate}
\item[(a)] $ab$ crosses $\rho(cd)$, and
\item[(b)] $\Hom_{\CC_\LL}(\Phi(cd), \Phi(ab)) \not= 0$.
\end{enumerate}
If $ab$ and $cd$ do not cross, then the following are equivalent:
\begin{enumerate}
\item\label{enumerate:AdjacentArcs} $ab$ and $cd$ are equal or adjacent (i.e., $\{a,b\} \cap \{c,d\} \not= \emptyset$), and
\item\label{enumerate:HomInADirection} $\Hom_{\CC_\LL}(\Phi(ab), \Phi(cd)) \not= 0$ or $\Hom_{\CC_\LL}(\Phi(cd), \Phi(ab)) \not= 0$.
\end{enumerate}
\end{corollary}

We will also consider a map $\Psi$ which maps a set $S \subseteq \arc(\LL)$ to the additive closure of $\Phi(S)$ in~$\CC_\LL$.

\begin{proposition}\label{proposition:PsiBijection}
The map $\Psi$ is a bijection between subsets of $\arc(\LL)$ and full replete (= closed under isomorphisms) Karoubi subcategories of $\CC_\LL$.
\end{proposition}

\begin{proof}
This is obvious.
\end{proof}

\subsection{\texorpdfstring{Maximal rigid and cluster-tilting subcategories of $\CC_\LL$}{Maximal rigid and cluster-tilting subcategories}}

In Proposition \ref{proposition:ObjectsArcs}, we have established a bijection between sets of arcs on the cyclically ordered set $(\LL, \leq)_{\rm cyc}$ and the indecomposable objects in the cluster category $\CC_\LL$.  Our goal is to describe those sets of arcs $S \subseteq \arc(\LL)$ for which $\Psi(S) \subseteq \CC_\LL$ is a cluster-tilting subcategory (see Corollary \ref{corollary:WhenClusterTilting} below).  We will do this by comparing the properties on $S$ discussed in Section \ref{section:Triangulations} (namely noncrossing, maximal, connected, triangulation, and whether $S$ allows flips) to properties of $\Psi(S)$.  This will be done in Proposition \ref{proposition:ClusterTiltingSomeProperties}.

We will use the following definition.

\begin{definition}\label{definition:AllowMutation}
Let $\CC$ be a Hom-finite Krull-Schmidt triangulated category, and let $\TT \subseteq \CC$ be a maximal rigid subcategory.  Let $a \in \ind \TT$.  We say that \emph{$\TT$ \emph{allows mutation} at $a \in \ind \TT$} if
\begin{enumerate}
\item\label{enumerate:AllowMutation1} there is a unique rigid indecomposable $a^* \in \ind \CC$, nonisomorphic to $a$, such that $\add ((\ind \TT \setminus \{a\}) \cup \{a^*\})$ is a maximal rigid subcategory of $\CC$, and
\item\label{enumerate:AllowMutation2} there are triangles
$$a \stackrel{f}{\rightarrow} B \stackrel{g}{\rightarrow} a^* \rightarrow a[1] \mbox{ and } a^* \stackrel{s}{\rightarrow} B' \stackrel{t}{\rightarrow} a \rightarrow a^*[1]$$
where $f$ and $s$ are minimal left $\add(\ind \TT \setminus \{a\})$-approximations, and $g$ and $t$ are minimal right $\add(\ind \TT \setminus \{a\})$-approximations.
\end{enumerate}
\end{definition}

\begin{proposition}\label{proposition:AllowMutation}
Let $\TT \subseteq \CC_\LL$ be a maximal rigid subcategory.  If $\TT$ satisfies (\ref{enumerate:AllowMutation1}) from Definition \ref{definition:AllowMutation}, then $\TT$ satisfies (\ref{enumerate:AllowMutation2}).
\end{proposition}

\begin{proof}
It follows easily from (\ref{enumerate:AllowMutation1}) that $a^* \not \in \TT$ and that $\Ext_{\CC_\LL}^1(a, a^*) \not= 0 \not= \Ext_{\CC_\LL}^1(a^*, a)$.  Let
$$a \stackrel{f}{\rightarrow} B \stackrel{g}{\rightarrow} a^* \to a[1]$$
be any nonsplit triangle.  For any $T \in \add(\ind \TT \setminus \{a\})$, we can use $\Ext^1(T,a) = 0 = \Ext^1(T,a^*)$ to find $\Ext^1(T,B) = 0$.  Applying the functor $\Hom(a[-1],-)$ to this triangle, we find the exact sequence:
$$\Hom(a,a[1]) \to \Hom(a,B[1]) \to \Hom(a,a^*[1]) \to \Hom(a,a[2]).$$
Using the fact that $a[2] \cong \bS a$ and $\dim \Hom(a,a[2]) \le 1$, we see that any nonzero map $w\colon a \to a[2]$ extends to an Auslander-Reiten triangle $a[1] \to E \to a \overset{w}\to a[2]$ and, hence, the composition $a \overset{w}\to a[2] \overset{f[2]}\to B[2]$ vanishes. In particular, $w$ factors through $a^*[1] \to a[2]$. This means that $\Hom(a,a^*[1]) \to \Hom(a,a[2])$ is surjective and, looking at the dimensions given in Proposition~\ref{proposition:DimensionsInCCLL}, we see that it is also bijective and, hence, $\Hom(a,B[1]) = 0$.

So far, we have shown that $\Ext^1(U,B) = 0$ for all $U \in \TT$.
Let $b$ be an indecomposable direct summand of $B$.  It follows from Proposition \ref{proposition:DimensionsInCCLL} that $\Ext^1(b,b) = 0$ and thus from the maximality of $\TT$ that $b \in \TT$.  We conclude that $B \in \TT$.

We can now use $\Ext^1(T, a) = 0$ to show that any morphism $T \to a^*$ factors through $g: B \to a^*$.  This shows that $g$ is a minimal right $\add(\ind \TT \setminus \{a\})$-approximation.  The other statements are proved similarly.

\end{proof}

Now we can relate the combinatorial properties of collections of arcs in $\LL$ to properties of subcategories of $\CC_\LL$. The definitions of locally finite collections and locally bounded categories are given in Definition~\ref{definition:MainTriangulationProperties} and~\S\ref{subsection:Reps of categories}, respectively.

\begin{proposition}\label{proposition:ClusterTiltingSomeProperties}
Let $\LL$ be an unbounded linearly ordered locally discrete set and let $S \subset \arc(\LL)$.  The set $S$ is a set of pairwise noncrossing arcs if and only if $\Phi(S)$ is rigid.  In this case, we also have:
\begin{enumerate}
\item $S$ is connected if and only if $\Psi(S)$ is indecomposable,
\item $S$ is maximal if and only if $\Psi(S)$ is a maximal rigid subcategory of $\CC_\LL$,
\item $S$ is locally finite if and only if $\Psi(S)$ is locally bounded, and
\item if $S$ is maximal, then every arc in $S$ is exchangeable if and only if $\Psi(S)$ allows all mutations.
\end{enumerate}
\end{proposition}

\begin{proof}
It follows from Proposition \ref{proposition:ObjectsArcs} that $S$ is a set of pairwise noncrossing arcs if and only if $\Psi(S)$ is rigid.  Assume now for the rest of the proof that $S$ is a set of pairwise noncrossing arcs.
\begin{enumerate}
\item This statement follows from Proposition \ref{proposition:DecompositionAdditiveCategories} and Corollary \ref{corollary:MorphismsArcs}.
\item This is clear.
\item Suppose that $S$ is not locally finite, that is, there is $a \in \mathcal{L}$ with infinitely many incident arcs in $S$. For any pair $ab$, $ac$ of such arcs, $\Phi(ab)$ and $\Phi(ac)$ admit a non-zero morphism in one direction by Corollary~\ref{corollary:MorphismsArcs}. Hence, $\Psi(S)$ is not locally bounded.

\begin{sloppypar} % Formatting aid (least evil, hopefully)
Conversely, if $\Psi(S)$ is not locally bounded, there is an arc $ab \in S$ such that $\Hom_{\CC_\LL}(\Phi(ab), \Phi(cd)) \not= 0$ or $\Hom_{\CC_\LL}(\Phi(cd), \Phi(ab)) \not= 0$, for infinitely many arcs $cd \in S$.  Hence, there are infinitely many arcs in $S$ adjacent to $ab$ by Corollary~\ref{corollary:MorphismsArcs}. In particular, one of $a$ and $b$ is incident with infinitely many arcs in $S$ and $S$ is not locally finite.
\end{sloppypar}

\item This is straightforward, based on Proposition \ref{proposition:AllowMutation}. \qedhere
\end{enumerate}
\end{proof}

Before stating the main result of the section, we observe the following.

\begin{lemma}\label{lemma:CompositionInTypeA}
Let $A,C \in \ind \CC_\LL$ such that $\Hom(A,C) \not= 0$.  For any $B \in \ind \CC_\LL$ with $\Hom(A,B) \not= 0 \not= \Hom(B,C)$, the composition map $\Hom(B,C) \otimes_k \Hom(A,B) \to \Hom(A,C)$ is an isomorphism.
\end{lemma}

\begin{proof}
This is easily confirmed to hold when $Q = A_n$.  We can use Theorem \ref{theorem:Covering} to reduce the general case to this case.
\end{proof}

Now we are ready to characterize cluster-tilting objects combinatorially, generalizing both~\cite[Theorem 4.4]{HolmJorgensen12} and~\cite[Theorem 6.9]{LiuPaquette15}.

\begin{theorem}\label{theorem:WhenClusterTilting}
Let $\LL$ be an unbounded linearly ordered locally discrete set.  A set of pairwise noncrossing arcs $S$ of arcs on $\LL$ is a connected triangulation if and only if $\Phi(S)$ is a cluster-tilting subcategory.
\end{theorem}

\begin{proof}
Assume that $\Phi(S)$ is a cluster-tilting subcategory of $\CC_\LL$.  It follows from Theorem \ref{theorem:IndecomposableClustertiltingSubcategories} that $\Phi(S)$ is indecomposable, and hence, by Proposition \ref{proposition:ClusterTiltingSomeProperties}, we know that $S$ is connected.  Furthermore, since $\Phi(S)$ is maximal rigid, Proposition \ref{proposition:ClusterTiltingSomeProperties} shows that $S$ is maximal.  Finally, since $\CC_\LL$ has a cluster structure (Theorem \ref{theorem:ClusterStructure}), we know that $\Phi(S)$ allows for mutation at every indecomposable object (in the sense of Definition \ref{definition:AllowMutation}) and hence every arc of $S$ is exchangeable, again by Proposition \ref{proposition:ClusterTiltingSomeProperties}. Theorem~\ref{theorem:FlippingIsTriangulated} implies that $S$ is a triangulation.

For the other direction, assume that $S$ is a connected triangulation.  It follows from Theorem \ref{theorem:ConnectedTriangulationsAreMaximal} that $S$ is maximal.  Using Proposition \ref{proposition:ClusterTiltingSomeProperties}, we infer that $\Psi(S)$ is a (connected) maximal rigid subcategory.  In view of Remark~\ref{remark:DefClusterTilting}(\ref{enumerate:MaxRigidVsClusterTilting}), we need only show that $\Psi(S)$ is functorially finite in $\CC_\LL$.  We will show that $\Psi(S)$ is covariantly finite in $\CC_\LL$, proving that $\Psi(S)$ is contravariantly finite is dual.

Let $ab \in \arc(\LL)$ be any arc (so that $\Phi(ab)$ is any indecomposable object in $\CC_\LL$).  We wish to find a left $\Psi(S)$-approximation $\Phi(ab) \to T$. If $ab \in S$, then we can choose $T = \Phi(ab)$.

Thus, assume that $ab \not\in S$ and let $\{x_i y_i\}_i \subseteq S$ be the finite set of Corollary~\ref{Corollary:TowardsFunctoriallyFiniteness} applied to $\rho(ab) \in \arc(\LL)$.  If we write $T_i = \Phi(x_i y_i)$, we specifically get that $x_i y_i$ crosses $\rho(ab)$ and that $\Hom(\Phi(ab),T_i) \ne 0$ by Corollary~\ref{corollary:MorphismsArcs}.
We claim that the map $\varphi: \Phi(ab) \to T \otimes_{\End T} \Hom(\Phi(ab),T)$, where $T = \oplus_i T_i$, is a left $\Psi(S)$-approximation.

Let $X \in \CC_\LL$ be any indecomposable object (say $X = \Phi(xy)$) and $f\colon \Phi(ab) \to X$.  We need to show that the map $f$ factors through $\varphi$.  By Lemma~\ref{lemma:CompositionInTypeA}, it suffices to show that $\Hom(T,X) \not= 0$ if $f \not= 0$.  However, if $f \not= 0$, we know by Corollary \ref{corollary:MorphismsArcs} that $xy$ and $\rho(ab)$ cross. This shows that $xy$ crosses the set $\{\rho(x_i x_i)\}_i$, which implies that $\Hom(T,X) \not= 0$, as required.  We conclude that $\Psi(S)$ is covariantly finite in $\CC_\LL$.
\end{proof}

\begin{corollary}\label{corollary:WhenClusterTilting}
Let $\LL$ be an unbounded linearly ordered locally discrete set, and let $S \subseteq \arc \LL$ be a set of pairwise noncrossing arcs on $\LL$.  The following are equivalent:
\begin{enumerate}
\item\label{enumerate:WhenClusterTilting1} $S$ is a connected triangulation,
\item\label{enumerate:WhenClusterTilting2} $S$ is connected, maximal, and every arc can be flipped,
\item\label{enumerate:WhenClusterTilting3} $\Phi(S)$ is a cluster-tilting subcategory,
\item\label{enumerate:WhenClusterTilting4} $\Phi(S)$ is indecomposable, maximal {rigid}, and every indecomposable object can be mutated.
\end{enumerate}
\end{corollary}

\begin{proof}
Theorem \ref{theorem:WhenClusterTilting} states that (\ref{enumerate:WhenClusterTilting1}) and (\ref{enumerate:WhenClusterTilting3}) are equivalent, and it follows from Theorem \ref{theorem:FlippingIsTriangulated} that (\ref{enumerate:WhenClusterTilting1}) and (\ref{enumerate:WhenClusterTilting2}) are equivalent.  Finally, Proposition \ref{proposition:ClusterTiltingSomeProperties} yields that (\ref{enumerate:WhenClusterTilting2}) and (\ref{enumerate:WhenClusterTilting4}) are equivalent.
\end{proof}

\begin{corollary}\label{corollary:LocallyBoundedClusterTilting}
Let $\LL$ be an unbounded linearly ordered locally discrete set.  There is a locally bounded cluster-tilting subcategory for $\CC_\LL$ if and only if $\LL$ is countable.
\end{corollary}

\begin{proof}
It follows from Theorem \ref{theorem:WhenClusterTilting} and Proposition \ref{proposition:ClusterTiltingSomeProperties} that $\CC_\LL$ has a locally bounded cluster-tilting subcategory if and only if $\LL$ has a locally finite triangulation.  The latter is satified if and only if $\LL$ is countable (see Theorem \ref{theorem:CountableLocallyFiniteTriangulation}).
\end{proof}

\section{Comparing algebraic cluster categories with directed cluster-tilting subcategories}\label{section:Comparing}

We will now consider algebraic cluster categories.  Let $\CC$ and $\CC'$ be algebraic Krull-Schmidt 2-Calabi-Yau categories with directed cluster-tilting subcategories $\TT$ and $\TT'$, respectively.  Assume that there is an equivalence $\Phi: \TT \to \TT'$.  If $\ind \TT$ is finite, then it follows from \cite{KellerReiten08} that $\CC \cong \CC'$ as triangulated categories.  Although we cannot prove that $\CC$ and $\CC'$ are equivalent in general, we will use ideas from \S\ref{section:Reduction} to compare $\CC$ and $\CC'$.

\begin{construction}\label{construction:F}
We define a map $F: \Ob \CC \to \Ob \CC'$ as follows: for $X \in \CC$, we choose a minimal right $\TT$-approximation triangle $T_1 \stackrel{f}{\rightarrow} T_0 \to X \to T_1[1]$ and let 
$$F(X) = \cone(\Phi(f): \Phi(T_1) \to \Phi(T_0)).$$
\end{construction}

\begin{remark}
The construction of $F$ above requires, for each object $X \in \CC$, a choice of a minimal right $\TT$-approximation and a choice for the cone of $\Phi(f)$.  However, the isomorphism class of $F(X)$ is independent of these choices, as the following lemma indicates.
\end{remark}

\begin{lemma}\label{lemma:Nonfunctor1}
For all $X_1, X_2 \in \CC$, we have:
\begin{enumerate}
\item $T^{(X_1)}_1 \to T^{(X_1)}_0 \to X_1 \to T^{(X_1)}_1[1]$ is a minimal right $\TT$-approximation triangle of $X_1$ if and only if $\Phi(T^{(X_1)}_1) \to \Phi(T^{(X_1)}_0) \to F(X_1) \to \Phi(T^{(X_1)}_1)[1]$ is a minimal right $\TT'$-approximation triangle of $F(X_1)$,
\item $F(X_1) \cong F(X_2)$ if and only if $X_1 \cong X_2$, and
\item $F(X_1 \oplus X_2) \cong F(X_1) \oplus F(X_2)$.
\end{enumerate}
\end{lemma}

\begin{proof}
\begin{enumerate}
\item We will prove one direction, namely that if $T^{(X_1)}_1 \to T^{(X_1)}_0 \to X_1 \to T^{(X_1)}_1[1]$ is a minimal right $\TT$-approximation triangle of $X_1$, then $\Phi(T^{(X_1)}_1) \to \Phi(T^{(X_1)}_0) \to F(X_1) \to \Phi(T^{(X_1)}_1)[1]$ is a minimal right $\TT'$-approximation triangle of $F(X_1)$.  

Let $f \in  \Hom_{\CC'}(T',F(X_1))$ be any morphism where $T' \in \TT'$.  Since $\TT'$ is a cluster-tilting subcategory of $\CC'$, we know that $\Ext^1(T', \Phi(T^{(X_1)}_1)) = 0$ and hence $f$ factors through the morphism $\Phi(T^{(X_1)}_0) \to F(X_1)$.  This shows that $\Phi(T^{(X_1)}_0) \to F(X_1)$ is a right $\TT'$-approximation.

The right $\TT'$-approximation is minimal if and only if the morphism $\Phi(T^{(X_1)}_1) \to \Phi(T^{(X_1)}_0)$ is radical.  This property is being kept under equivalence, and hence $\Phi(T^{(X_1)}_1) \to \Phi(T^{(X_1)}_0)$ is radical since $T^{(X_1)}_1 \to T^{(X_1)}_0$ is radical.

\item We fix an isomorphism $X_1 \to X_2$.  Using the properties of minimal right $\TT$-approximations, we get a commutative diagram
$$\xymatrix{
T^{(X_1)}_1 \ar[r] \ar@{-->}[d] & T^{(X_1)}_0 \ar[r] \ar@{-->}[d] & X_1 \ar[r] \ar[d] & T^{(X_1)}_1[1] \ar@{-->}[d] \\
T^{(X_2)}_1 \ar[r] & T^{(X_2)}_0 \ar[r] & X_2 \ar[r] & T^{(X_2)}_1[1]
}$$
where the rows are minimal right $\TT$-approximation triangles and the vertical arrows are isomorphisms.  Applying the equivalence $\Phi: \TT \to \TT'$ on the left square gives a commutative square in $\TT'$:
$$\xymatrix{
\Phi(T^{(X_1)}_1) \ar[r] \ar[d] & \Phi(T^{(X_1)}_0) \ar[d] \\
\Phi(T^{(X_2)}_1) \ar[r] & \Phi(T^{(X_2)}_0)
}$$
where the vertical maps are isomorphisms.  It is now clear that the cones of the horizontal maps, $F(X_1)$ and $F(X_2)$, are isomorphic.

The other implication follows from the first statement.

\item Consider the following minimal right $\TT$-approximation triangles:
$$\xymatrix@R=5pt{
T^{(X_1)}_1 \ar[r] & T^{(X_1)}_0 \ar[r]& X_1 \ar[r] & T^{(X_1)}_1[1],\\
T^{(X_2)}_1 \ar[r] & T^{(X_2)}_0 \ar[r] & X_2 \ar[r] & T^{(X_2)}_1[1].
}$$
We find the following right minimal $\TT'$-approximations:
{
$$\xymatrix@R=5pt@C-5pt{
\Phi(T^{(X_1)}_1) \oplus \Phi(T^{(X_2)}_1) \ar[r] & \Phi(T^{(X_1)}_0) \oplus \Phi(T^{(X_2)}_0) \ar[r]& F(X_1) \oplus F(X_2) \ar[r] & [\Phi(T^{(X_1)}_1) \oplus \Phi(T^{(X_2)}_1)][1],\\
\Phi(T^{(X_1)}_1 \oplus T^{(X_2)}_1) \ar[r] & \Phi(T^{(X_1)}_0 \oplus T^{(X_2)}_0) \ar[r]& F(X_1 \oplus X_2) \ar[r] & [\Phi(T^{(X_1)}_1 \oplus T^{(X_2)}_1)][1].
}$$
}
This shows that $F(X_1 \oplus X_2) \cong F(X_1) \oplus F(X_2)$, as required.
\end{enumerate}
\end{proof}

We will now compare the endomorphism rings of $X$ and $F(X)$.  In particular, for each $X \in \CC$, we want to construct an algebra-isomorphism
$$F_X: \End (X) \to \End (FX).$$
For this, we consider the diagram in Figure \ref{figure:Nonfunctor}, which we will now explain.

\begin{figure}
$$\xymatrix{
& \ZZ^{\perp_1} \ar[r] \ar[d] & \CC \\
\Db \mod \Lambda \ar[d]_{- \otimes_{\Lambda} \Lambda'} \ar[r]^-{\pi} & \CC_{[\ZZ]} \ar@{-->}[d]^G \\
\Db \mod \Lambda'\ar[r]^-{\pi'} & \CC'_{[\ZZ']} \\
& (\ZZ')^{\perp_1} \ar[r] \ar[u] & \CC'
}$$
	\caption{Comparing algebraic cluster categories $\CC$ and $\CC'$ using Iyama-Yoshino reductions.}
	\label{figure:Nonfunctor}
\end{figure}

We start with an element $X \in \CC$, and we choose a minimal $T \in \TT$ such that {$X \in \DD_T$ and $F(X) \in \DD_{\Phi(T)}$} (see Notation~\ref{notation:DT}, the existence of such a minimal $T$ follows from Lemma~\ref{lemma:AboutDT}).  It follows from Lemma \ref{lemma:Covering} that $X$ lives in $\CC_{[\ZZ]}$ and $X'$ lives in $\CC'_{[\ZZ']}$ for well-chosen rigid subcategories $\ZZ \subseteq \CC$ and $\ZZ' \subseteq \CC'$.  Moreover, $T \in \CC_{[\ZZ]}$ and $\Phi(T) \in \CC'_{[\ZZ']}$ are cluster-tilting objects; we write $\Lambda = \End(T)$ and $\Lambda' = \End(\Phi(T))$. Clearly $\Lambda \cong \Lambda'$ as algebras.  It follows from Corollary \ref{corollary:Trianglefree} that $\TT$ is semi-hereditary, so that Proposition~\ref{proposition:ShIsLocal} implies that $\Lambda$ is hereditary. 

\begin{sloppypar}
It follows from \cite{KellerReiten08} that $\CC_{[\ZZ]} \cong (\Db \mod \Lambda)/{\bS \circ [-2]}$; in particular there is a functor $\pi\colon \Db \mod \Lambda \to \CC_{[\ZZ]}$ such that the restriction $\pi|_{\add(\Lambda[0])}: \add \Lambda[0] \to \add(T)$ is an equivalence.  Similarly, we have a functor $\pi': \Db \mod \Lambda' \to \CC'_{[\ZZ']}$ such that the restriction $\pi'|_{\add(\Lambda'[0])}: \add \Lambda'[0] \to \add(T')$ is an equivalence.  The 2-universal property of orbit categories yields a (unique up to unique natural isomorphism) equivalence $G: \CC_{[\ZZ]} \stackrel{\sim}{\rightarrow} \CC'_{[\ZZ']}$ such that the square
$$\xymatrix{
\Db \mod \Lambda \ar[d]_{- \otimes_{\Lambda} \Lambda'} \ar[r]^-{\pi} & \CC_{[\ZZ]} \ar@{-->}[d]^G \\
\Db \mod \Lambda'\ar[r]^-{\pi'} & \CC'_{[\ZZ']}
}$$
essentially commutes. In fact, by \cite[Theorem 4]{Keller05} we can choose $G$ to be a triangle functor.
\end{sloppypar}

It follows from Lemma \ref{lemma:Covering} that $\add T \subseteq \ZZ^{\perp_1}$ and that the composition $\add T \to \ZZ^{\perp_1} \to \CC_{[\ZZ]}$ is fully faithful.  Similarly, we can see $\add T'$ as a full subcategory of $\CC'_{[\ZZ']}$.

The above commutative square shows that $G|_{\add T} \cong \Phi|_{\add T}$.

Let $T_1 \to T_0 \to X \to T_1[1]$ be a minimal right $\TT$-approximation triangle of an object $X \in \CC$.  Following \S\ref{Subsection:IyamaYoshinoReduction}, we know that there is a triangle $T_1 \to T_0 \to X \to T_1 \langle 1 \rangle$ in $\CC_{[\ZZ]}$.  Likewise (see Lemma \ref{lemma:Nonfunctor1}) we know that there is a triangle $\Phi(T_1) \to \Phi(T_0) \to F(X) \to \Phi(T_1) \langle 1 \rangle$ in $\CC'_{[\ZZ']}$.

Using the fact that $G$ is a triangle functor and that $G|_{\add T} \cong \Phi|_{\add T}$, we find that $F(X) \cong G(X)$.  {Summarized}, we have
$$\begin{array}{rclr}
\End_\CC(X) & \cong & \End_{\CC_{[\ZZ]}}(X) & \mbox{since $X$ lives in $\CC_{[\ZZ]}$ by Lemma \ref{lemma:Covering}} \\
&\cong& \End_{\CC'_{[\ZZ']}}(G(X)) & \mbox{since $G$ is an equivalence} \\
&\cong& \End_{\CC'_{[\ZZ']}}(F(X)) & \mbox{since $F(X) \cong G(X)$} \\
&\cong& \End_{\CC'}(F(X)) & \mbox{since $F(X)$ lives in $\CC'_{[\ZZ']}$ by Lemma \ref{lemma:Covering}.} \\
\end{array}$$

We are now ready to show the following proposition.

\begin{proposition}\label{proposition:EndomorphismsUnderF} Let $X,Y \in \CC$.
\begin{enumerate}
\item\label{enumerate:EndomorphismsUnderF} There is an algebra-isomorphism $F_X: \End (X) \to \End(F(X))$,
\item\label{enumerate:ExtensionsUnderF} $\Ext^1_\CC(X,Y) \cong \Ext^1(F(X), F(Y))$.
\end{enumerate}
\end{proposition}

\begin{proof}
\begin{enumerate}
\item This has been established above.
\item Let $T \in \CC$ such that {$X,Y \in \DD_T \subseteq \CC$ and $F(X),F(Y) \in \DD_{\Phi(T)} \subseteq \CC'$}.  In this case, there are rigid and functorially finite subcategories $\ZZ \subseteq \CC$ and $\ZZ' \subseteq \CC'$ such that {$\DD_T \subseteq \ZZ^{\perp_1}$ and $\DD_{\Phi(T)} \subseteq (\ZZ')^{\perp_1}$} (see Lemma~\ref{lemma:AboutDT}).

It then follows from \cite[Lemma 4.8]{IyamaYoshino08} that $\Ext^1_\CC(X,Y) \cong \Ext^1_{\CC_{[\ZZ]}}(X,Y)$ and $\Ext^1_{\CC'}(X,Y) \cong \Ext^1_{\CC'_{[\ZZ']}}(F(X),F(Y))$.  Using the knowledge that $G(X) \cong F(X)$ and $G(Y) \cong F(Y)$ in $\CC'_{[\ZZ]}$ and that $G: \CC_{[\ZZ]} \to \CC'_{[\ZZ']}$ is a triangle equivalence, we see that $\Ext^1(X,Y) \cong \Ext^1(F(X), F(Y))$.
\end{enumerate}
\end{proof}

\begin{corollary}\label{corollary:IndecomposablesUnderF}
$F: \Ob \CC \to \Ob \CC'$ maps indecomposable objects to indecomposable objects.
\end{corollary}

\begin{corollary}\label{corollary:EndomorphismsUnderF}
For each $X \in \CC$, there is an equivalence $F_X\colon \add_\CC X \to \add_{\CC'} F(X)$, whose action on objects is given by $F$.
\end{corollary}

\begin{proof}
This follows directly from Proposition \ref{proposition:EndomorphismsUnderF}(\ref{enumerate:EndomorphismsUnderF}).
\end{proof}

\begin{theorem}\label{theorem:EquivalentClusterStructures}
Let $\CC$ and $\CC'$ be Krull-Schmidt algebraic 2-Calabi-Yau categories with cluster-tilting subcategories, $\TT$ and $\TT'$, respectively.  Assume that $\TT$ and $\TT'$ are directed and equivalent.

There is a map $F: \Ob \CC \to \Ob \CC'$ satisfying the following properties:
\begin{enumerate}
\item\label{enumerate:FBijectionObjects} $F$ induces a bijection on isomorphism classes of objects,
\item\label{enumerate:FBijectionCT} $F$ induces a bijection on cluster-tilting subcategories,
\item\label{enumerate:FMutation} $F$ respects mutations.
\end{enumerate}
\end{theorem}

\begin{proof}
Let $\Phi: \TT \to \TT'$ be an equivalence, and let $\Phi': \TT' \to \TT$ be a quasi-inverse.  We let $F: \Ob \CC \to \Ob \CC'$ and $F': \Ob \CC' \to \Ob \CC$ be the maps from Construction \ref{construction:F}, based on $\Phi$ and $\Phi'$, respectively.

\begin{enumerate}
\item This follows from $F' \circ F (X) \cong X$ and $F \circ F' (X') \cong X'$, for all $X \in \CC$ and $X \in \CC'$, as is readily verified.
\item Let $\UU \subseteq \CC$ be any cluster-tilting subcategory. We will first show that $F(\UU) = F(\UU)^{\perp_1}$ (so that also $F(\UU) = {}^{\perp_1} F(\UU)$ since $\CC'$ is 2-Calabi-Yau). Clearly $F(\UU) \subseteq F(\UU)^{\perp_1}$ by Proposition~\ref{proposition:EndomorphismsUnderF}(\ref{enumerate:ExtensionsUnderF}). If on the other hand $Y \in F(\UU)^{\perp_1}$, we choose $X \in \CC$ such that $F(X) \cong Y$. Then $X \in \UU^{\perp_1}$, again by Proposition~\ref{proposition:EndomorphismsUnderF}(\ref{enumerate:ExtensionsUnderF}), so that $X \in \UU$. It follows that $Y \in F(\UU)$.

We will now prove that $F(\UU)$ is functorially finite in $\CC'$.  Let $X' \in \CC'$ and let $f\colon U \to F'(X')$ in $\CC$ be a right $\UU$-approximation map for $F'(X')$.  We claim that the canonical map $\epsilon\colon \Hom(F(U), X') \otimes F(U) \to X'$ is a right $F(\UU)$-approximation of $X'$.  This would then establish that $F(\UU)$ is {contravariantly} finite in $\CC'$.

Since every map $F(U) \to X'$ factors through $\epsilon\colon \Hom(F(U), X') \otimes F(U) \to X'$, we only need show that for every morphism $g'\colon U' \to X'$ (with $U' \in F(U)$), there is a map $f'\colon F(U) \to X'$ such that $g'$ factors through $f'$, that is to say, for every $g'\colon U' \to X'$, there are morphisms $h'\colon U' \to F(U)$ and $f'\colon F(U) \to X'$ such that the diagram
\[
\xymatrix{&U' \ar[d]^{g'} \ar[ld]_{h'} \\ F(U) \ar[r]_{f'} & X'}
\]
commutes.  It follows from the equivalence $F'_{Z'}\colon \add_{\CC'} Z' \to \add_\CC F'(Z')$ from Corollary \ref{corollary:EndomorphismsUnderF} that it suffices to show that for every $g\colon F'(U') \to F'(X')$, there is a morphism $f_1\colon U \to F'(X')$ such that $g$ factors through $f_1$. Here, we may choose $f_1$ to be the right $\UU$-approximation $f\colon U \to F'(X')$.  The equivalence $F'_{Z'}\colon \add_{\CC'} Z' \to \add_\CC F'(Z')$ now establishes the existence of the morphism $f'\colon F(U) \to X'$ in $\CC'$, as required.

That $F(U)$ is {covariantly} finite in $\CC'$ is shown similarly.

\item Let $\UU, \VV \subseteq \CC$ be cluster-tilting subcategories of $\CC$, such that $\VV$ is the mutation of $\UU$ at an indecomposable object $X \in \UU$.  We have already verified in (\ref{enumerate:FBijectionCT}) that $F(\UU)$ and $F(\VV)$ are cluster-tilting subcategories of $\CC'$.

It follows from (\ref{enumerate:FBijectionObjects}) that $F(\UU \cap \VV)$ is an almost complete cluster-tilting subcategory (see~\S\ref{subsection:pairs}), and Theorem~\ref{theorem:IyamaYoshinoMutation} yields that $F(\VV)$ is the mutation of $F(\UU)$.
\qedhere
\end{enumerate}
\end{proof}

\begin{remark}
Let $\LL$ be an unbounded linearly ordered locally discrete set, and consider the associated cyclically ordered set $(C_\LL, R) = (\LL, \leq)_{\rm cyc}$.  We will consider the associated categories $\CC_{\phi}(C_\LL)$ from \cite[\S 2]{IgusaTodorov15} (here, {$\phi\colon C \to C, c \mapsto c+1$}) and $\CC_\LL$ from \S\ref{section:CominatoricsOfTypeA}.  Note that in both categories, the objects correspond to arcs in $(C,R)$ (this is \cite[Lemma 2.1.6]{IgusaTodorov15} for $\CC_{\phi}(C_\LL)$ and Proposition \ref{proposition:ObjectsArcs} for $\CC_\LL$).

The categories $\CC_{\phi}(C_\LL)$ and $\CC_\LL$ are both algebraic and have equivalent (directed) cluster-tilting subcategories (given by a triangulation as in Example \ref{example:LightConeArcs}).  Theorem \ref{theorem:EquivalentClusterStructures} thus allows for a comparison between the cluster-tilting subcategories.

In general, the clusters considered in \cite[Definition 2.4.6]{IgusaTodorov15} are not indecomposable (as additive categories), and hence do not correspond to cluster-tilting subcategories.  Indeed, consider, for example, the case where $\LL = \bZ \cdot \bZ$, and consider the triangulation of Example \ref{example:NoncrossingSets}(\ref{enumerate:Example6}).  This is a cluster in the sense of \cite[Definition 2.4.6]{IgusaTodorov15} but is not connected in the sense of Definition \ref{definition:MainTriangulationProperties} and hence does not correspond to a cluster-tilting subcategory (see Theorem \ref{theorem:WhenClusterTilting}).  Therefore, it is cluster in the sense of Theorem \ref{theorem:ClusterStructure}.
\end{remark}

\section{About the existence of cluster maps}\label{section:ClusterMaps}

Let $\CC$ be an algebraic Krull-Schmidt 2-Calabi-Yau triangulated category with a directed cluster-tilting subcategory $\TT$. Assume furthermore that $\ind \TT$ is countable.  In this section, we prove that there exists a cluster map $\ind \CC \to \bQ(x_i)_{i \in \bN}$.  The motivation here is to allow for a better understanding of cluster algebras of infinite rank studied in~\cite{GrabowskiGratz14,Gratz15}.

The existence of the cluster map will follow from \cite{JorgensenPalu13} after we show that $\CC$ has a locally bounded cluster-tilting subcategory $\UU$. Note that this other cluster-tilting subcategory $\UU$ often \emph{cannot} be chosen to be directed {(or semi-hereditary, see Corollary~\ref{corollary:Trianglefree})}. We will exhibit a particular example in Remark~\ref{remark:LocFinAndDirectedMayBeIncompatible}.

We start by recollecting some properties of thread quivers from \cite{BergVanRoosmalen14}.

\subsection{About thread quivers}

Apart from the information recalled in \S\ref{subsection:ThreadQuivers}, we will need the following concepts (which were used in the proof of Theorem~\ref{theorem:ClassificationIntroduction} in \cite{BergVanRoosmalen14}). Let $\TT$ be a semi-hereditary dualizing $k$-variety.  For $X,Y \in \ind \TT$, we write $[X,Y]$ for the full replete (= closed under isomorphisms) additive subcategory of $\CC$ such that for all $Z \in \ind \CC$, we have $Z \in \ind [X,Y] \Leftrightarrow \Hom_\TT(X,Z) \not= 0$ and $\Hom_\TT(Z,Y) \not= 0$.

\begin{definition}
An indecomposable object $X \in \ind \TT$ is called a \emph{thread object} if there is a left almost split map $X \to M$ in $\TT$ and a right almost split map $N \to X$ in $\TT$ where $M$ and $N$ are indecomposable.  In this case, we write $X^+$ for $M$ and $X^-$ for $N$.

For $X,Y \in \ind \TT$, the subcategory $[X,Y]$ is called a \emph{thread} if every indecomposable object in $[X,Y]$ is a thread object in $\TT$.  A thread is called \emph{maximal} if it is not a proper subset of another thread.  It is called \emph{infinite} if it contains infinitely many nonisomorphic indecomposables.

An indecomposable object in $\TT$ which is not a thread object is called a \emph{nonthread object}.
\end{definition}

We recall some relevant properties.

\begin{proposition}\label{proposition:MaximalThreads}
Let $\TT$ be a semi-hereditary dualizing $k$-variety and let $[X,Y] \subseteq \TT$ be a thread.
\begin{enumerate}
\item The thread $[X,Y]$ is contained in a maximal thread $[X',Y'] \subseteq \TT$.
\item There is a bounded linearly ordered locally discrete set $\LL$ such that $[X,Y] {\simeq} k\LL$.
\end{enumerate}
\end{proposition}

\begin{proof}
The first statement is \cite[Corollary 6.6]{BergVanRoosmalen14}, the second statement \cite[Corollary 6.3]{BergVanRoosmalen14}.
\end{proof}

\begin{proposition}\label{proposition:ThreadQuivers}
Let $\TT$ be a semi-hereditary dualizing $k$-variety and let $\{[X_i,Y_i]\}_i$ be the set of all maximal infinite threads in $\TT$.  We write $\ZZ = \bigoplus_i [X_i, Y_i]$ and $\QQ = \add(\ind \TT \setminus \ind \ZZ)$.
\begin{enumerate}
\item The map $\ZZ \to \TT$ has a left and a right adjoint,
\item the map $\QQ \to \TT$ is fully faithful and has a left and a right adjoint,
\item a nonzero morphism from $Z_i \in [X_i, Y_i]$ to $Z_j \in [X_j, Y_j]$ factors through $\QQ \subseteq \TT$ if and only if $i \not= j$, and
\item there is a strongly locally finite quiver $Q$ such that $\QQ {\simeq} kQ$.
\end{enumerate}
\end{proposition}

\begin{proof}
The statements follow from \cite[Proposition 5.2]{BergVanRoosmalen14}, \cite[Corollary 6.8]{BergVanRoosmalen14}, \cite[Corollary 6.4]{BergVanRoosmalen14} and \cite[Propositions 7.1 and 7.2]{BergVanRoosmalen14}, respectively.
\end{proof}

\begin{proposition}\label{proposition:BeyondTheEdge}
Let $[X_i, Y_i] \subseteq \TT$ be a maximal thread and $Z \in \ind [X_i, Y_i]$.  For the left and right adjoints $l,r\colon \TT \to \QQ$ to the embedding $\QQ \to \TT$, we have $l(Z) \cong Y_i^+$ and $r(Z) \cong X_i^-$.
\end{proposition}

\begin{proof}
This is \cite[Corollary 6.4]{BergVanRoosmalen14}, together with its dual.  Here, we use the assumption that $[X_i, Y_i]$ is maximal to conclude that $X_i^-, Y_i^+ \in \QQ$. 
\end{proof}

\begin{lemma}\label{lemma:InThreadFactoring}
Let $\ZZ_i = [X_i,Y_i]$ be a thread in $\TT$, and let $Q \in \QQ$.  Let $X,Y \in \ind \ZZ_i$ and let $f\colon X \to Y$ be a morphism.  If $f$ is nonzero, then every morphism $X \to Q$ factors through $f\colon X \to Y$.
\end{lemma}

\begin{proof}
If $X \cong Y$, then $f\colon X \to Y$ is an isomorphism and the statement follows easily.  Assume now that $X \not\cong Y$, so that the morphism $f\colon X \to Y$ factors through the right almost split map $Y^- \to Y$.  We consider now the thread $[X,Y^-]$; since $\Hom(X,Y^-) \not= 0$, we know that $[X,Y^-]$ is nonempty, and since $X,Y^- \in [X_i, Y_i]$, we know that $[X,Y^-] \subseteq [X_i, Y_i]$ is a thread. It follows from \cite[Corollary 6.4]{BergVanRoosmalen14} that $X \to Q$ factors through a morphism $X \to (Y^-)^+ \cong Y$.  Since $\dim \Hom(X,Y) = 1$ (see, for example, \cite[Proposition 6.2]{BergVanRoosmalen14}), we may assume that $X \to Q$ factors through $f\colon X \to Y$.
\end{proof}

\subsection{A locally bounded cluster-tilting subcategory}  In this subsection, we want to construct a locally bounded cluster-tilting subcategory of $\CC$.  We construct this subcategory in three steps.  First, we consider the set $\{[X_i,Y_i]\}_{i}$ of all maximal infinite threads in $\TT$.  Let $\ZZ = \bigoplus_i [X_i, Y_i]$ and $\QQ = \add(\ind \TT \setminus \ind \ZZ)$.  Since the embedding $\QQ \to \TT$ has a left and a right adjoint (see Proposition \ref{proposition:ThreadQuivers}), $\QQ$ is a functorially finite subcategory of $\CC$ and we know that  $\CC_{[\QQ]}$ is an algebraic cluster category with a cluster-tilting subcategory equivalent to $\ZZ$ (see~\S\ref{Subsection:IyamaYoshinoReduction}).

Our second step is based on Theorem \ref{theorem:EquivalentClusterStructures}.  Let $\LL_i$ be a bounded linearly ordered locally discrete set such that $[X_i,Y_i] \simeq k\LL_i$.  We know from Proposition \ref{proposition:ThreadQuivers} that the cluster-tilting subcategory of $\CC_{[\QQ]}$ is equivalent to $\ZZ = \bigoplus_i [X_i, Y_i]$.  Thus:
$$\xymatrix{
\CC & \ar[l] \QQ^{\perp_1} \ar[d]_{\pi} & \ar[l] \TT \ar[d]^{\pi} \\
& \CC_{[\QQ]} & \ar[l] \ZZ \ar[r]^-{\sim} & \bigoplus_i [X_i, Y_i] 
}$$
where $\pi: \QQ^{\perp_1} \to \CC_{[\QQ]} = \QQ^{\perp_1} / [\QQ]$ is the quotient functor.  Theorem \ref{theorem:EquivalentClusterStructures} yields that $\CC_{[\QQ]}$ has a locally bounded cluster-tilting subcategory if and only if $\bigoplus_i \CC_{\LL_i}$ has a locally bounded cluster-tilting subcategory.  Since $\ind \TT$ is countable, we know that $\ind \ZZ$ is countable, and hence each $\LL_i$ is countable.  It follows from Corollary \ref{corollary:LocallyBoundedClusterTilting} that $\CC_{\LL_i}$ has a locally bounded cluster-tilting subcategory $\WW_i$.  Hence, $\CC_{[\QQ]}$ has a locally bounded cluster-tilting subcategory $\WW \simeq \bigoplus_i \WW_i$.

In the third step, we lift the locally bounded cluster-tilting subcategory $\WW$ from $\CC_{[\QQ]}$ to a cluster-tilting subcategory $\pi^{-1}(\WW)$ of $\CC$ (see \cite[Theorem 4.9]{IyamaYoshino08}):
$$\xymatrix{
\CC & \ar[l] \QQ^{\perp_1} \ar[d]_{\pi} & \ar[l] {\pi^{-1}(\WW)} \ar[d]^{\pi} \\
& \CC_{[\QQ]} & \ar[l] \WW 
}$$
We will show in Theorem \ref{theorem:LocallyFiniteClusterTilting} below that $\pi^{-1}(\WW)$ is locally bounded.

To study the category $\pi^{-1}(\WW)$, we will consider the following parts: the full subcategory $\QQ$ of $\pi^{-1}(\WW)$, and for every maximal thread $[X_i,Y_i]$ in $\TT$, the subcategory $\VV_i = \add(\ind \pi^{-1}(\WW_i) \setminus \ind \QQ)$ of $\pi^{-1}(\WW)$.

\begin{remark} On objects, there is a disjoint union $\ind \pi^{-1}(\WW) = \ind \QQ \cup \bigcup_i \ind \VV_i.$
\end{remark}

\begin{proposition}\label{proposition:TypesOfApproximationTriangles}
Let $\ZZ_i = [X_i,Y_i] \subseteq \TT$ be a maximal infinite thread.  For each indecomposable $V \in \ind \VV_i \subseteq \CC$, there are three possibilities of minimal right $\TT$-approximation triangles:
\begin{enumerate}
\item\label{enumerate:TriangleForm1} $X \to Y \to V \to X[1]$ where $X,Y \in \ind \ZZ_i$,
\item\label{enumerate:TriangleForm2} $X \to {Y_i^+} \to V \to X[1]$ where $X \in \ind \ZZ_i$, or
\item\label{enumerate:TriangleForm3} $0 \to X \to V \to 0$ where $X \in \ind \ZZ_i$.
\end{enumerate}
Moreover, for each thread, there are only finitely many isomorphism classes of indecomposables in $\VV_i$ which have $\TT$-approximation triangles of the last two types.
\end{proposition}

\begin{proof}
We start with a minimal $\ZZ_i$-approximation triangle of $\pi(V) \in \CC_{[\QQ]}$: $X \to Y \to \pi(V) \to X\langle 1 \rangle$.  By the construction of $\CC_{[\QQ]}$, we know that the first three terms are the image of a triangle in $\CC$
$$X \oplus Q_1 \to Y \oplus Q_2 \to V \oplus Q_3 \to (X \oplus Q_1)[1]$$
(with $X,Y,V \in \QQ^{\perp_1}$ and $Q_1,Q_2,Q_3 \in \QQ$) under the functor $\pi: \QQ^{\perp_1} \to \QQ^{\perp_1} / [\QQ]$.

We observe that $Q_3, X \oplus Q_1 \in \TT$, so that $\Ext^1(Q_3, X \oplus Q_1) = 0$. Since by construction $V \in \QQ^{\perp_1}$, we also have $\Ext^1(V,Q_1) = 0$ and the rightmost map of the above triangle is a direct sum of $X\to V[1]$ and the zero map $Q_3 \to Q_1[1]$. In particular, the triangle splits into a sum of the triangles
$$ Q_1 \to Q_1 \oplus Q_3 \to Q_3 \overset{0}\to Q_1[1] \quad\mbox{and}\quad X \to Y \oplus Q \to V \to X[1] $$

Hence, we can assume that the $\ZZ_i$-approximation triangle of $\pi(V)$ is the image under $\pi$ of a triangle $X \to Y \oplus Q \to V \to X[1]$ in $\CC$ such that $X,Y,V \in \QQ^{\perp_1} \subseteq \CC$ have no indecomposable summands in $\QQ$ and $Q \in \QQ$.

Assume first that both $X$ and $Y$ are nonzero. Then $\Hom_\CC (X,Y) \ne 0$ because otherwise $\Hom_{\CC_{[\QQ]}} (X,Y) = 0$ and consequently $\pi(V) \cong X \langle 1 \rangle \oplus Y$ in $\CC_{[\QQ]}$.  Let $Y'$ be an indecomposable direct summand of $Y$ such that $\Hom(X, Y') \not= 0$ and that $Y'$ is minimal among the indecomposable direct summands of $Y$ with this property (minimal with respect to the ordering on $\ZZ_i \cong k\LL_i$). Similarly, let $X'$ be an indecomposable summand of $X$ such that $\Hom(X', Y') \not= 0$ and $X'$ is maximal with this property (again maximal with respect to the ordering on $\ZZ_i \cong k\LL_i$).

We can now apply Lemma \ref{lemma:InThreadFactoring} to see that the nonzero morphism $X' \to Y \oplus Q$ factors through $X' \to Y'$. Similarly, we deduce that the non-zero morphism $X \to Y'$ factors through $X' \to Y'$. Now we have the following commutative diagram
$$\xymatrix{
X' \ar[r] \ar[d] & Y' \ar[r] \ar[d]& V' \ar[r] \ar@{-->}[d] & X'[1] \ar[d] \\
X \ar[r] \ar[d] & Y \ar[d]\oplus Q \ar[r] & V \ar[r] \ar@{-->}[d] & X[1] \ar[d] \\
X' \ar[r] & Y' \ar[r] & V' \ar[r] & X'[1]
}$$
where the rows are triangles and the vertical compositions are isomorphisms.  This shows that $V'$ is a direct summand of $V$, from which we infer that either $V' = 0$ or $V \cong V'$.  We may exclude the former case (since this implies that $X' \cong Y'$, contradicting the minimality of the approximation triangle), so we may assume the latter case.  Here, we find the minimal right $\TT$-approximation $X' \to Y' \to V \to X'[1]$, as in~{(\ref{enumerate:TriangleForm1})}.

The next case we consider is where $Y=0$.  Here, it follows from \cite[Lemma 4.3(2)]{IyamaYoshino08} that the map $X \to Q$ is a left $\QQ$-approximation so that Proposition \ref{proposition:BeyondTheEdge} implies that the $\TT$-approximation triangle of $V$ is as in~{(\ref{enumerate:TriangleForm2})}. The case where $X = 0$ trivially leads to a $\TT$-approximation triangle as in~(\ref{enumerate:TriangleForm3}).

Now, we need only show that only finitely many $\TT$-approximation triangles have the form (\ref{enumerate:TriangleForm2}) or (\ref{enumerate:TriangleForm3}).  This means that in $\CC_{[\QQ]}$, only finitely many isomorphism classes of indecomposable objects in {$\WW_i$} lie in $\ZZ_i$ or in $\ZZ_i\langle 1 \rangle$. This follows since {$\WW_i$} is locally bounded and $\ZZ_i$ is an infinite thread.
\end{proof}

\begin{corollary}\label{corollary:WhenDoesQuiverObjectSeeATread1}
Let $Q \in \QQ \subseteq \TT$ and $V \in \VV_i$.  If $V$ has a $\TT$-approximation triangle as in Proposition~\ref{proposition:TypesOfApproximationTriangles}(\ref{enumerate:TriangleForm1}), then $\Hom_\CC(Q,V) = 0$.
\end{corollary}

\begin{proof}
Let $X \to Y \to V \to X[1]$ be the minimal right $\TT$-approximation triangle of $V$ with $X,Y \in \ind \ZZ_i$. We apply the functor $\Hom(Q,-)$ to obtain
$$\Hom(Q,X) \to \Hom(Q,Y) \to \Hom(Q,V) \to \Hom(Q,X[1]).$$
Since $\TT$ is rigid, we know that $\Hom(Q,X[1]) = 0$.  It follows from Proposition \ref{proposition:BeyondTheEdge} that $\Hom_\CC(Q,X) \cong \Hom(Q, X_i^-) \cong \Hom(Q,Y)$.  Since $\TT$ is semi-hereditary, we know that $\Hom(Q,X) \to \Hom(Q,Y)$ is a monomorphism, and thus a bijection.  We conclude that $\Hom(Q,V) = 0$. 
\end{proof}

\begin{corollary}\label{corollary:WhenDoesQuiverObjectSeeATread2}
Let $Q \in \QQ \subseteq \TT$. Then $\Hom_\CC(Q,\VV_i) \ne 0$ only if $\Hom_\CC(Q,X_i^- \oplus Y_i^+) \ne 0$.
\end{corollary}

\begin{proof}
If $V \in \VV_i$ is such that $\Hom(Q,V) \ne 0$, then the minimal right $\TT$-approximation triangle of $V$ must be as in Proposition~\ref{proposition:TypesOfApproximationTriangles}(\ref{enumerate:TriangleForm2}) or (\ref{enumerate:TriangleForm3}).  In the case of Proposition~\ref{proposition:TypesOfApproximationTriangles}(\ref{enumerate:TriangleForm2}), any map $f\colon Q \to V$ factors through $Y_i^+$ since $Q,X \in \TT$ and $\Ext^1(Q,X) = 0$. Hence $\Hom(Q,Y_i^+) \ne 0$. In the other case, any map $f\colon Q \to V \cong X$ factors through $X_i^-$ thanks to Proposition~\ref{proposition:BeyondTheEdge}.  Thus, $\Hom(Q,X_i^-) \ne 0$ in that case.
\end{proof}

We will also use the dual of Proposition \ref{proposition:TypesOfApproximationTriangles} and its corollaries.

\begin{proposition}\label{proposition:TypesOfApproximationTrianglesDual}
Let $\ZZ_i = [X_i,Y_i] \subseteq \TT$ be a maximal infinite thread.  For each indecomposable $Z \in \ZZ_i \subseteq \CC$, there are three possibilities of right $\TT$-approximation triangles:
\begin{enumerate}
\item\label{enumerate:TriangleForm1Dual} $V \to X \to Y \to V[1]$ where $X,Y \in \ind \ZZ_i$,
\item\label{enumerate:TriangleForm2Dual} $V \to {X_i^-} \to Y \to V[1]$ where $Y \in \ind \ZZ_i$, or
\item\label{enumerate:TriangleForm3Dual} $V \to X \to 0 \to V[1]$ where $X \in \ind \ZZ_i$.
\end{enumerate}
Moreover, for each thread, there are only finitely many isomorphism classes of indecomposables in $\VV_i$ which have $\TT$-approximation triangles of the last two types.
\end{proposition}

\begin{corollary}\label{corollary:WhenDoesQuiverObjectSeeATreadDual}
Let $Q \in \QQ \subseteq \TT$ be a nonthread object. Then the following hold:

\begin{enumerate}
\item If $V \in \VV_i$ has a $\TT$-approximation as in Proposition~\ref{proposition:TypesOfApproximationTrianglesDual}(\ref{enumerate:TriangleForm1Dual}), then $\Hom_\CC(V,Q) = 0$.
\item $\Hom_\CC(\VV_i,Q) \ne 0$ only if $\Hom_\CC(X_i^- \oplus Y_i^+,Q) \ne 0$.
\end{enumerate}
\end{corollary}

Now we can now prove the following theorem.

\begin{theorem}\label{theorem:LocallyFiniteClusterTilting}
Let $\CC$ be a Krull-Schmidt 2-Calabi-Yau triangulated category with a directed cluster-tilting subcategory $\TT$.  If $\ind \TT$ is countable, then $\CC$ has a locally bounded cluster-tilting subcategory. 
\end{theorem}

\begin{proof}
We will show that the cluster-tilting subcategory $\UU = \pi^{-1}(\WW)$ of $\CC$ which we constructed above is locally bounded.

We start by considering an object $Q \in \ind \QQ \subseteq \TT$.  Since $\QQ$ is locally bounded, there are only finitely many objects $Q' \in \QQ$  such that $\Hom(Q,Q') \not= 0$ or $\Hom(Q',Q) \not= 0$. It follows from Corollary~\ref{corollary:WhenDoesQuiverObjectSeeATread2} that $\Hom(Q,\VV_i) \not= 0$ only for finitely many maximal infinite threads $[X_i,Y_i]$. Even if $\Hom(Q,\VV_i) \not= 0$, there are by Proposition~\ref{proposition:TypesOfApproximationTriangles} and Corollary~\ref{corollary:WhenDoesQuiverObjectSeeATread1} only finitely many $V \in \VV_i$ such that $\Hom(Q,V) \not= 0$. Dually, $\Hom(V,Q) \not= 0$ only for finitely many $V \in \bigcup_i \ind(\VV_i)$ by Proposition~\ref{proposition:TypesOfApproximationTrianglesDual} and Corollary~\ref{corollary:WhenDoesQuiverObjectSeeATreadDual}.  Thus, there are only finitely many $U \in \ind(\UU)$ such that $\Hom(Q,U) \not= 0$ or $\Hom(U,Q) \not= 0$.

Next, we turn our attention to an object $V \in \ind \VV_i$.  Since $\QQ$ is locally bounded, there are only finitely many $Q \in \ind \QQ$ such that $\Hom(V,Q) \not= 0$ or $\Hom(Q,V) \not= 0$ by Corollaries~\ref{corollary:WhenDoesQuiverObjectSeeATread2} and~\ref{corollary:WhenDoesQuiverObjectSeeATreadDual}.

Finally, consider $V' \in \bigcup_j \ind(\VV_j)$.  By construction, we know that there are only finitely many such $V'$ for which $\Hom_{\CC_{[\QQ]}}(V,V') \not= 0$. Although it may happen that $\Hom_{\CC_{[\QQ]}}(V,V') = 0$ and $\Hom_\CC(V,V') \not= 0$, we have $\Hom_\CC(V,V') = [\QQ](V,V') \not= 0$ in such a case.  However, there are only finitely many $Q \in \ind(\QQ)$ such that $\Hom_\CC(V,Q) \ne 0$ and finitely many $V' \in \bigcup_j \ind(\VV_j)$ such that $\Hom_\CC(Q,V') \not= 0$, so there can be only finitely many $V'$ such that $[\QQ](V,V') \not= 0$.  In summary, there can be only finitely many $U \in \ind(\UU)$ such that $\Hom_\CC(V,U) \not= 0$. The proof that $\Hom_\CC(U,V) \not= 0$ only for finitely many $U \in \ind(\UU)$ is similar.
\end{proof}

\begin{remark}\label{remark:CountableNecessary}
The condition that $\ind \TT$ is countable in Theorem \ref{theorem:LocallyFiniteClusterTilting} is necessary if $\CC$ is a block.  Indeed, if $\ind \TT$ is uncountable, then $\ind \VV$ is uncountable for any cluster-tilting subcategory $\VV$ of $\CC$ (see \cite[Theorem 2.4]{DehyKeller08}), and if $\CC$ is a block, then $\VV$ is indecomposable (see Theorem \ref{theorem:IndecomposableClustertiltingSubcategories}).  Such a subcategory $\VV$ cannot be locally bounded.
\end{remark}

\subsection{A cluster map for cluster categories with locally bounded cluster-tilting subcategories}

Our main application of Theorem \ref{theorem:LocallyFiniteClusterTilting} is the existence of a cluster map.  We recall the definition; here, let $\VV$ be a cluster-tilting subcategory of $\CC$, $\Gamma_\VV$ be the connected component of the exchange graph of $\CC$ containing $\VV$, and $\EE_\VV$ be the additive category $\add \{\UU \subset \CC \mid \UU \in \Gamma_\VV \}$, so that $\EE$ is the additive subcategory generated by all objects which are reachable from $\VV$.

\begin{definition}
A map
$$\varphi_\VV: \Ob \EE_\VV \to \bQ(x_v)_{v \in \ind \VV}$$
is said to be a cluster map, if the following conditions are satisfied:
\begin{enumerate}
\item $\varphi_\VV$ is constant on isomorphism classes, i.e. if $E \cong E'$ then $\varphi_\VV(E) = \varphi_\VV(E')$,
\item $\varphi_\VV(E \oplus E') = \varphi_\VV(E) \cdot \varphi_\VV(E')$, for all $E, E' \in \EE$,
\item if $\dim \Ext^1(E,E') = 1$, then $\varphi_\VV(E) \cdot \varphi_\VV(E') = \varphi_\VV(M) + \varphi_\VV(M')$ where $M$ and $M'$ are given by the nonsplit triangles:
\[\mbox{$E \to M \to E' \to E[1]$ and $E' \to M' \to E \to E'[1]$,}\]
\item there is a cluster-tilting subcategory $\VV'$, reachable from $\VV$, for which $\varphi_\VV(\ind \VV')$ is a transcendence basis of $\bQ(x_v)_{v \in \ind \VV}$ over $\bQ$.
\end{enumerate}
\end{definition}

\begin{remark}\label{remark:ClusterAlgebraFromClusterMap}
For each rigid indecomposable $X \in \CC$, the image $\rho_\VV(X) \in \bQ(x_v)_{v \in \ind \VV}$ is a cluster variable.  The subalgebra $\bF$ of $\bQ(x_v)_{v \in \ind \VV}$ generated by these cluster variables is the cluster algebra; the clusters correspond to the cluster-tilting subcategories in $\Gamma_\VV$ (see \cite[Proposition IV.1.2]{BuanIyamaReitenScott09}).
\end{remark}

\begin{theorem}\label{theorem:ExistenceOfCCMap}
Let $\CC$ be a Krull-Schmidt 2-Calabi-Yau triangulated category with an indecomposable directed cluster-tilting subcategory $\TT$ {(in particular, $\CC$ is a block)}. If $\ind \TT$ is countable, then there exists a cluster-tilting subcategory $\VV$ of $\CC$ such that:
\begin{enumerate}
\item $\ind \VV$ is countable,
\item $\EE_\VV$ contains all rigid objects, and
\item there exists a cluster map $\rho_\VV: \Ob \EE_\VV \to \bQ(x_v)_{v \in \ind \VV}$.
\end{enumerate}
\end{theorem}

\begin{proof}
From Theorem \ref{theorem:LocallyFiniteClusterTilting} we know that $\CC$ has a locally bounded cluster-tilting category, say $\VV$.  It follows from \cite[Theorem 2.4]{DehyKeller08} that $\ind \VV$ is countable.  It further follows from Corollary \ref{corollary:Bongartz} that $\EE_\VV$ contains all rigid objects.  Finally, it follows from \cite{JorgensenPalu13} (see \cite[Remark 5.4]{JorgensenPalu13}, based on \cite{CalderoChapoton06, Palu08}) that a cluster map $\rho_\VV: \Ob \EE_\VV \to \bQ(x_v)_{v \in \ind \VV}$ exists.
\end{proof}

\begin{remark}
An explicit formula for the map $\rho_\VV: \Ob \EE_\VV \to \bQ(x_v)_{v \in \ind \VV}$ from Theorem \ref{theorem:ExistenceOfCCMap} is given in \cite[\S 1.8]{JorgensenPalu13} (it is a version of the Caldero-Chapoton fomula from~\cite[\S3]{CalderoChapoton06}).  In fact --as also noted in~\cite{JorgensenPalu13}-- this formula is defined on all of $\Ob \CC$ instead of merely on $\Ob \EE_\VV$.
\end{remark}

\begin{remark} \label{remark:LocFinAndDirectedMayBeIncompatible}
It may happen that the cluster-tilting category $\VV$ cannot be chosen to be directed. The smallest example of this phenomenon occurs for $\CC = \CC_\LL$, where $\LL = \{0,1,2\} \stackrel{\rightarrow}{\times} \bZ = \bZ \cdot \bZ \cdot \bZ$.

A connected triangulation which corresponds (via Theorem~\ref{theorem:WhenClusterTilting}) to a directed cluster-tilting subcategory is given in Example~\ref{example:LightConeArcs}, while a locally finite connected triangulation which defines a locally bounded cluster-tilting subcategory is depicted in Figure~\ref{figure:LocallyFiniteTriangle}. One can convince oneself that there is no connected triangulation of $(\LL, \leq)_{\rm cyc}$ which would both be locally finite and correspond to a directed cluster-tilting subcategory of $\CC$.
\end{remark}

\begin{remark}\label{remark:WhyCountable}
If $\TT$ is indecomposable and $\ind \TT$ is not countable, then no cluster-tilting subcategory of $\CC$ is locally bounded (see Remark \ref{remark:CountableNecessary}), and it follows from Corollary \ref{corollary:Bongartz} that not every rigid object is reachable.  Hence, the condition that $\ind \TT$ is countable is needed for the statement that $\EE_\VV$ contains all rigid objects.
\end{remark}

\begin{remark}\label{remark:DifferentClusterStructures}
In the setting of Theorem \ref{theorem:ExistenceOfCCMap}, the exchange graph of $\CC$ contains a connected component $\Gamma$ such that for any rigid object $X \in \CC$, there is a cluster-tilting subcategory $\UU$ of $\CC$ such that $X \in \UU$ and $\UU \in \Gamma$.  Indeed, the connected component $\Gamma$ is the one that contains the cluster-tilting subcategory $\VV$ from Theorem \ref{theorem:ExistenceOfCCMap}.

Under the cluster map $\rho_\VV: \Ob \EE_\VV \to \bQ(x_v)_{v \in \ind \VV}$, the cluster-tilting subcategories in $\Gamma$ correspond to the seeds of the cluster algebra $\bF$ (see Remark~\ref{remark:ClusterAlgebraFromClusterMap}).  Let $C_\VV$ be the set of clusters of this cluster algebra.

Let $\UU$ be a locally bounded cluster-tilting subcategory of $\CC$ and let $\Gamma_\UU$ be the connected component of the exchange graph containing $\UU$.  Following \cite[Proposition IV.1.2(a)]{BuanIyamaReitenScott09}, the cluster map $\rho_\VV: \Ob \EE_\VV \to \bQ(x_v)_{v \in \ind \VV}$ gives $\bF$ (see Remark \ref{remark:ClusterAlgebraFromClusterMap}) the structure of a cluster algebra whose seeds are given by the cluster-tilting subcategories in $\Gamma_\UU$ instead of the cluster-tilting subcategories in $\Gamma_\VV$ (we require that $\UU$ is locally bounded so that $\rho(\UU)$ is a transcendence basis of $\bQ(x_v)_{v \in \ind \VV}$ over $\bQ$);  let $C_\UU$ be the set of clusters on $\bF$ rendered in this way.

We say that a finite set $S$ of cluster variables in $\bF$ is \emph{compatible} with $C_\UU$ or $C_\VV$ if there is a cluster in $C_\UU$ or $C_\VV$, respectively, which contains $S$.  For any finite set of rigid indecomposables $\{X_i \in \CC\}_i$, the cluster variables $\{\rho_\VV(X_i)\}_i$ lie in a cluster (of both $C_\VV$ and $C_\UU$, meaning that the set $\{\rho_\VV(X_i)\}_i$ of cluster variables is compatible with both $C_\VV$ and $C_\UU$) if and only if $\oplus_i X_i$ is rigid in $\CC$ (this follows from Corollary \ref{corollary:Bongartz}).

If $\UU$ is not reachable from $\VV$ (for example, if $|\ind \VV|$ is countably infinite and $\UU = \VV[1]$) then $C_\VV \not= C_\UU$.  This shows that $\bF$ is not unistructural in the sense of \cite{AssemShramchenkoSchiffler14}, meaning that the clusters in $\bF$ are not uniquely determined by cluster variables.  In particular, \cite[Conjecture 1.2]{AssemShramchenkoSchiffler14} fails when it comes to the (infinite rank) cluster algebra $\bF$.  Moreover, the clusters are not determined by knowing all (finite) compatible sets.
\end{remark}
\appendix

\section{A Bongartz-type completion for functorially finite rigid subcategories}\label{section:Bongartz}

In this section, we will work with a Krull-Schmidt 2-Calabi-Yau triangulated category $\CC$ over a field $k$ (not necessarily algebraically closed).  The main result is to show the following theorem.

\begin{theorem}\label{theorem:BongartzComplement}
Let $\CC$ be a Krull-Schmidt 2-Calabi-Yau triangulated category, and let $\WW$ be a functorially finite rigid subcategory of $\CC$.  If $\CC$ has a cluster-tilting subcategory, then $\CC$ has a cluster-tilting subcategory $\RR$ which contains $\WW$.
\end{theorem}

We will construct this cluster-tilting subcategory {$\RR = \add(\WW \cup \FF)$} in Construction \ref{construction:Bongartz}.  The above theorem will be obtained as a trivial corollary of Proposition \ref{proposition:Bongartz} below.

\begin{construction}\label{construction:Bongartz}
Let $\CC$ be a 2-Calabi-Yau category with a cluster-tilting subcategory $\TT \subseteq \CC$.  Let $\WW \subseteq \CC$ be a functorially finite rigid subcategory.  For every $T \in \TT$, we consider the triangle
\[\xymatrix@1{
F_T \ar[r] & W \ar[r]^{f} & T[1] \ar[r] &F_T[1]
}\]
based on the right $\WW$-approximation $f: W \to T[1]$.  Let $\FF$ be the additive subcategory $\add \{F_T\}_{T \in \TT}$ and let $\RR = \add(\FF \cup \WW)$.
\end{construction}

\begin{remark}
A similar result was presented in \cite[Definition-Proposition 4.21]{Jasso13} for when $\CC$ satisfies the additional condition of admitting a cluster-tilting object instead of a (more general) cluster-tilting subcategory.  The construction and part of the proof are adaptations of \cite{Jasso13}; the main difference is that we need to show that {$\RR$} is functorially finite in $\CC$.
\end{remark}

\begin{proposition}\label{proposition:Bongartz}
\begin{enumerate}
\item $\Ext^1(\WW, \FF) = \Ext^1(\FF, \WW) = \Ext^1(\FF, \FF) = 0$,
\item {$\RR = \add( \FF \cup \WW)$ is maximal rigid in $\CC$},
\item {$\RR = \add( \FF \cup \WW)$ is a cluster-tilting subcategory in $\CC$.}
\end{enumerate}
\end{proposition}

\begin{remark}
The cluster-tilting subcategory $\RR$ is called the \emph{Bongartz completion} of $\WW$ by $\TT$ in \cite{Jasso13} (see \cite[2.1 Lemma]{Bongartz81}).
\end{remark}

\begin{proof}[{Proof of Proposition~\ref{proposition:Bongartz}}]
\begin{enumerate}
\item
By applying the functor $\Hom(W'[-1], -)$ to the above triangle (for any $W' \in \WW$), and using the fact that $f: W \to T$ is a right $\WW$-approximation, we find that $\Hom(W'[-1], F_T) = 0$ and thus also that $\Hom(F_T[-1],W') = 0$.  This shows that $\FF \subseteq \WW^{\perp_1}$.

Next, consider $T,T' \in \TT$, and let $g: F_{T'}[-1] \to F_T$ be any morphism.  Consider the following diagram
\[\xymatrix@C+10pt{
T'[-1] \ar[r] & F_{T'}[-1] \ar[r] \ar[d]^{g} \ar@{-->}[dl]_{h}& W'[-1] \ar[r]^{-f'[-1]}& T' \\
T \ar[r] & F_{T} \ar[r] & W \ar[r]_{f}& T[1]
}\]
where the rows are triangles, and the morphisms $f, f'$ are right $\WW$-approximations.  Since $\Hom(F_{T'}[-1], W) = 0$, we know that $g: F_{T'}[-1] \to F_T$ factors through a morphism $h: F_{T'}[-1] \to T$.  However, the composition $T'[-1] \to F_{T'[-1]} \to T$ is zero as well (since $\TT$ is rigid).  The morphism $h$ thus factors through a morphism $j: W'[-1] \to T$:
 \[\xymatrix@C+10pt{
W'[-2] \ar[r]^{f'[-2]}& T'[-1] \ar[r] & F_{T'}[-1] \ar[r] \ar[dl]_{h}& W'[-1] \ar@{-->}[dll]^{j} \\
W[-1] \ar[r]_{-f[-1]}& T \ar[r] & F_{T} \ar[r] & W
}\]
However, since $\Hom(W'[-1], F_T) = 0$ (as established before), we may conclude that $g=0$.  We have shown that $\FF$ is rigid.

\item We will prove that {$\RR = \add(\FF \cup \WW)$} is a maximal rigid subcategory of $\CC$ by contradiction.  Thus, assume the existence of an object $X \in \CC \setminus \RR$ such that {$\add(\RR \cup \{X\})$} is rigid.

Consider a right $\TT$-approximation triangle of $X$: $T_0 \stackrel{t}{\rightarrow} T_1 \to X \to T_0[1]$ and the triangle $F_{T_0} \to W \stackrel{f}{\rightarrow} T_0[1] \to F_{T_0}[1]$ built on the right $\WW$-approximation morphism $f: W \to T_1[1]$.

Since $\Ext^1(W',X) = 0$ for all $W' \in \WW$, any morphism $g: W' \to T_1[1]$ factors through $t[1]: T_0[1] \to T_1[1]$ and hence through $f: W \to T_1[1].$  This shows that $t[1] \circ f: W \to T_1[1]$ is a right $\WW$-approximation (albeit not necessarily minimal), and we find that $\cone(t[1] \circ f) \in \RR[1]$.  Using \cite[Proposition 1.1.11]{BeilinsonBernsteinDeligne82}, we have a commutative diagram
\[
\xymatrix{
& & X \ar@{=}[r] \ar[d] & X \ar[d]^{h} \\
F_{T_0} \ar[r] \ar[d] & W \ar[r]^{f} \ar@{=}[d] & T_0[1] \ar[r] \ar[d]^{t[1]} & F_{T_0}[1] \ar[d] \\
R \ar[r] & W \ar[r] & T_1[1] \ar[r] \ar[d] & R[1] \ar[d] \\
& & X[1] \ar@{=}[r] & X[1]\\
}
\]
where the middle two rows and the last two columns are triangles, and $R \in \RR$.  Since $\Ext^1(X, \FF) = 0$, we know that $h=0$ and hence $X$ is a direct summand of $R \in \RR$.  This shows that $\RR$ is maximal.

\item To prove that {$\add(\FF \cup \WW)$} is a cluster-tilting subcategory of $\CC$, it suffices to show that {$\add(\FF \cup \WW)$} is functorially finite (see \cite[Theorem 2.6]{ZhouZhu11}) in $\CC$, or thus, by \cite[Theorem 4.9(1)]{IyamaYoshino08}, that $\FF$ becomes functorially finite in $\CC_{[\WW]}$.

Let $X \in \WW^{\perp_1}$ and let $g:T_X \to X$ be a right $\TT$-approximation of $X$.  There is a commutative diagram
\[ \xymatrix@C+15pt{W_X[-1] \ar[r]^{-f_X[-1]} & T_X \ar[d]_{g} \ar[r]^{a_X} & F_{T_X} \ar@{-->}[ld]^{g'} \ar[r] & W_X \\ & X} \]
where the top row is a triangle (thus $F_{T_X} \in \FF$).  The existence of $g'$ follows from $\Hom(W_X[-1], X) = 0$.

We claim that the map $g': F_{T_X} \to X$ becomes the right $\FF$-approximation of $X$ after application of the functor $\pi: \WW^{\perp_1} \to \CC_{[\WW]}$.  For this, we consider an object $F \in \FF$ and a map $h: F \to X$ and show that there exists a map $q: F \to F_{T_X}$ such that $h - g' \circ q$ factors through $\WW$.

Without loss of generality, we may assume that $F \cong F_T$ for some $T \in \TT$, so there is a triangle
\[ \xymatrix{W[-1] \ar[r] & T \ar[r]^{a} & F \ar[r] & W} \]
with $T \in \TT$ and $W \in \WW$.  Since $g: T_X \to X$ is a right $\TT$-approximation, the map $h \circ a: T \to X$ factors as $g \circ p$ for some $p: T \to T_X$.  Using the fact that $f_X: W_X \to T_X[1]$ is a right $\WW$-approximation, we get a commutative diagram
\[ \xymatrix@C+15pt{W[-1] \ar[r] \ar@{-->}[d] & T \ar[d]_{p} \ar[r]^{a} & F \ar@{-->}[d]^{q} \ar[r] & W \ar@{-->}[d] \\ 
W_X[-1] \ar[r]_{-f_X[-1]} & T_X \ar[d]_{g} \ar[r]^{a_X} & F_{T_X} \ar@{-->}[ld]^{g'} \ar[r] & W_X \\ & X
} \]
We now have
\[ h \circ a = g \circ p = g' \circ a_X \circ p = g' \circ q \circ a,\]
so that $(h - g' \circ q) \circ a = 0$.  This shows that $h - g' \circ q$ factors through $\WW$ as required.   
\end{enumerate}
\end{proof}

\begin{corollary}\label{corollary:BongartzObjects}
Let $\CC$ be a Krull-Schmidt 2-Calabi-Yau triangulated category with a cluster-tilting object.  Any rigid object in $\CC$ is a direct summand of a cluster-tilting object.
\end{corollary}

\begin{proof}
Let $W$ be the rigid object in the statement of the category.  It follows from Theorem \ref{theorem:BongartzComplement} that there is a cluster-tilting subcategory $\RR$ such that $W \in \RR$ (or $\add(W) \subseteq \RR$).  It then follows from \cite{DehyKeller08} (see also \cite[Corollary 4.5]{AdachiIyamaReiten14} or \cite[Corollary 3.7]{ZhouZhu11}) that there is an object $R$ such that $\RR = \add(R)$, meaning that $R$ is a cluster-tilting object of $\CC$.  Since $W \in \RR$, we know that $W$ is a direct summand of $R$.
\end{proof}
\section{Indecomposable cluster-tilting subcategories}\label{section:IndecomposableCTS}

Let $(\AA_i)_{i \in I}$ be a family of small additive categories indexed by a set $I$.  The coproduct $\oplus_i \AA_i$ (in the 2-category of small additive categories) is a small additive category whose objects are families $(A_i)_{i \in I}$ with $A_i \in \AA_i$ (all except finitely many of these $A_i$'s are required to be zero), and where a morphism $(A_i)_i \to (B_i)_i$ is given by a family of morphisms $(f_i)_i$ with $f_i \in \Hom_{\AA_i}(A_i,B_i)$.  The composition is pointwise.  In $\oplus_i \AA_i$, we have $(A_i)_i \oplus (B_i)_i \cong (A_i \oplus B_i)_i$.

Let $\BB$ be any small additive category, and let $F_i: \AA_i \to \BB$ be a family of functors.  We consider the functor
\begin{eqnarray*}
\bigoplus_i \AA_i &\to& \BB \\
(A_i)_i &\mapsto& \oplus_i F_i(A_i).
\end{eqnarray*}
This functor is fully faithful if and only if the functors $F_i: \AA_i \to \BB$ are fully faithful and $\Hom_\BB(F_i(A_i), F_j(A_j)) = 0$ for all $i \not= j$. We have the following lemma (of which we omit the proof).

\begin{lemma}\label{lemma:WhenCoproductAdditive}
Let $\AA$ be an additive category and let $\AA_i \subseteq \AA$ be a family of full additive subcategories.  We have $\AA \simeq \oplus_i \AA_i$ if and only if $\Hom_\AA(\AA_i, \AA_j) = 0$ for $i \not= j$ and every object in $\AA$ is of the form $\oplus_i A_i$ (where only finitely many $A_i$'s are nonzero).
\end{lemma}

An additive category $\AA$ is called \emph{indecomposable} if it is nonzero and if it is not the coproduct of two nonzero additive categories.  When $\AA \cong \AA_1 \oplus \AA_2$, where $\AA_1$ is indecomposable, we will call $\AA_1$ a \emph{connected component} of $\AA$.

Let $\AA$ be an additive Krull-Schmidt category.  An \emph{unoriented path of length $n$} in $\AA$ is a sequence $X_0, X_1, \ldots, X_n$ of indecomposable objects such that for all $i=0, 1, \ldots, n-1$ we have that either $\Hom(X_i,X_{i+1}) \not= 0$ or $\Hom(X_{i+1},X_i) \not= 0$.  We will say that $\AA$ is \emph{path-connected} if there is an unoriented path between every two indecomposable objects of $\AA$.

\begin{proposition}\label{proposition:DecompositionAdditiveCategories}
An additive Krull-Schmidt category $\AA$ is path-connected if and only if it is indecomposable.  Furthermore, $\AA \cong \oplus_i \AA_i$ where each $\AA_i$ is indecomposable.
\end{proposition}

\begin{proof}
It is clear that $\AA$ is indecomposable if it is path-connected.  For the converse, let $A \in \ind \AA$ and let $\BB$ be the full subcategory of $\AA$ consisting of all objects $B \in \ind \AA$ such that there is an unoriented path between $A$ and $B$, thus $\add(\BB)$ is the path-connected component of $\AA$ which contains $A$.  Let $\CC = (\ind \AA) \setminus \BB$.  It follows from Lemma \ref{lemma:WhenCoproductAdditive} that $\AA = \add(\BB) \oplus \add(\CC)$ and that $\BB$ is indecomposable.

Therefore, every object $A \in \AA$ lies in a path-connected (and hence indecomposable) component of $\AA$.  This shows that $\AA \cong \oplus_i \AA_i$ where each $\AA_i$ is indecomposable.
\end{proof}

Although triangulated categories are additive categories, this notion of indecomposability is too strong for our purposes.  If $\CC$ and $\DD$ are triangulated categories, then the coproduct $\CC \oplus \DD$ defined above can be given the structure of a triangulated category in a natural way.  We will call a triangulated category a \emph{block} if and only if it is nonzero and not the coproduct of two nonzero triangulated categories, i.e.\ if it is indecomposable in the 2-category of (essentially small) triangulated categories.

\begin{remark}
An indecomposable triangulated category is a block, but the converse does not hold.  For example, $\Db(\mod k)$ is a block, but not indecomposable (as additive category).
\end{remark}

\begin{remark}
The notion of a block was introduced in \cite{Ringel05}.
\end{remark}

Let $\CC$ be a triangulated Krull-Schmidt category.  An \emph{unoriented suspended path} between indecomposable objects $A, B \in \CC$ is a sequence $A = X_0, X_1, \ldots, X_n = B$ of indecomposable objects such that for all $i=0, 1, \ldots, n-1$ we have that either $\Hom(X_i,X_{i+1}) \not= 0$, $\Hom(X_{i+1},X_i) \not= 0$, or $X_{i+1} \cong X_i[n]$ for some $n \in \bZ$.  We will say that $\CC$ is \emph{path-connected} if there is an unoriented suspended path between every two indecomposable objects of $\CC$.

\begin{lemma}\label{lemma:WhenCoproductTriangulated}
Let $\CC$ be a triangulated category and let $\CC_i \subseteq \CC$ be a family of full triangulated subcategories.  We have $\CC \simeq \oplus_i \CC_i$ if and only if $\Hom_\CC(\CC_i, \CC_j) = 0$ for $i \not= j$ and every object in $\CC$ is of the form $\oplus_i \CC_i$ (where only finitely many $\CC_i$'s are nonzero).
\end{lemma}

\begin{proposition}\label{proposition:DecompositionTriangulatedCategories}
A triangulated Krull-Schmidt category $\CC$ is path-connected if and only if it is a block.  Furthermore, $\CC \cong \oplus_i \CC_i$ where each $\CC_i$ is a block.
\end{proposition}

\begin{proof}
See \cite[Lemma 4]{Ringel05}.
\end{proof}

Our main result of this section is Theorem \ref{theorem:IndecomposableClustertiltingSubcategories}.  We start with the following proposition.

\begin{proposition}\label{proposition:IndecomposableApproximation}
Let $\CC$ be a Krull-Schmidt 2-Calabi-Yau triangulated category, and let $\TT \subseteq \CC$ be a cluster-tilting subcategory.  Let $C \in \CC$ be indecomposable.  For any minimal $\TT$-approximation triangle $T_1 \to T_0 \to C \to T_1[1]$ of $C$, the object $T_1 \oplus T_0$ lies in a single connected component of $\TT$.
\end{proposition}

\begin{proof}
Following Proposition \ref{proposition:DecompositionAdditiveCategories}, we write $\TT = \oplus_i \TT_{(i)}$ where each $\TT_{(i)}$ is indecomposable.  This induces decompositions $T_1 \cong \oplus_i T_{(i),1}$ and $T_0 \cong \oplus_i T_{(i),0}$, and we consider the compositions
$$f_i: T_{(i),1} \to T_1 \to T_0 \to T_{(i),0}$$
so that $f = \oplus_i f_i$.  Let $C_i = \cone(f_i:  T_{(i),1} \to T_{(i),0})$.  We have a commutative diagram
\[\xymatrix{
\oplus_i T_{(i),1} \ar[r] \ar[d]_{\sim} &  \oplus_i T_{(i),0} \ar[r] \ar[d]_{\sim} & \oplus_i C_i \ar[r] \ar@{-->}[d]_{\sim} & \oplus_i T_{(i),1}[1] \ar[d]_{\sim} \\
T_1 \ar[r] & T_0 \ar[r] & C \ar[r] & T_1[1]}\]
where the rows are triangles.  Since $C$ is indecomposable, all but one of the $C_i$'s are zero, and thus all but one of the $f_i$'s are isomorphisms.  Since we started with a minimal $\TT$-approximation triangle, we infer that all but one of the $f_i$'s are zero.  This implies that $T_1$ and $T_0$ lie in the same connected component of $\TT$.
\end{proof}

\begin{lemma}\label{lemma:PerpendicularNotEmpty}
Let $\CC$ be a Krull-Schmidt 2-Calabi-Yau triangulated category, and let $\TT \subseteq \CC$ be a cluster-tilting subcategory.  If $\TT = \UU \oplus \VV$, then $\UU \subseteq \VV^\perp$.
\end{lemma}

\begin{proof}
We need to show that $\Hom(\VV, \UU[n]) = 0$, for all $n \in \bZ$.  We will start by proving this for $n \geq 0$ by induction.  We have $\Hom(\VV, \UU) = 0$ (since $\TT \cong \UU \oplus \VV$) and $\Hom(\VV, \UU[1]) = 0$ (since $\TT$ is rigid).

Let $n \geq 2$ and let $U \in \UU$ be indecomposable.  We want to show that $\Hom(\VV, U[n]) = 0$.  Consider the minimal $\TT$-approximation triangle of $U[n-1]$:
\[\xymatrix@1{T_1 \ar[r] & T_0 \ar[r] & U[n-1] \ar[r] & T_1[1].}\]
 By the induction hypothesis, we know that $\Hom(\VV, U[n-1]) = 0$ and thus by Proposition \ref{proposition:IndecomposableApproximation} that $T_0,T_1 \in \UU$.

If $U[n-1] \in \UU$, then $\Hom(\VV,U[n]) = 0$ and we are done.  If $U[n-1] \not\in \UU$, then $T_1 \not= 0$.  Let $T_1'$ be an indecomposable direct summand of $T_1$, and consider the Auslander-Reiten triangle $T_1'[1] \to M \to T_1' \stackrel{f}{\rightarrow} T_1'[2]$ (here, we used $\bS T_1' \cong T_1'[2]$).  Since the composition $T_1' \stackrel{f}{\rightarrow} T_1'[2] \to T_1[2] \to T_0[2]$ is zero, we infer that there is a nonzero morphism $T_1' \to U[n]$, and hence $\Hom(\UU, U[n]) \not= 0$.

Let $S_1 \to S_0 \to U[n] \to S_1[1]$ be a minimal $\TT$-approximation triangle for $U[n]$.  Since $\Hom(\UU, U[n]) \not= 0$, we see that $\Hom(\UU, S_0) \not= 0$ and thus by Proposition \ref{proposition:IndecomposableApproximation} that $S_0,S_1 \in \UU$.  We now find easily that $\Hom(\VV, U[n]) = 0$.  Induction shows that $\Hom(\VV, \UU[n]) = 0$ for all $n \geq 0$.

The argument for $\Hom(\VV, \UU[n]) = 0$ for $n\le 0$ is similar.
\end{proof}

\begin{theorem}\label{theorem:IndecomposableClustertiltingSubcategories}
Let $\CC$ be a Hom-finite Krull-Schmidt 2-Calabi-Yau triangulated category, and let $\TT \subseteq \CC$ be a cluster-tilting subcategory.  The triangulated category $\CC$ is a block if and only if the additive category $\TT$ is indecomposable.
\end{theorem}

\begin{proof}
If $\TT$ is indecomposable, then one can combine Propositions \ref{proposition:DecompositionAdditiveCategories} and \ref{proposition:DecompositionTriangulatedCategories} to see that $\CC$ is a block.

Thus, assume that $\TT$ is decomposable; we write $\TT = \UU \oplus \VV$.  Since $\bS \cong [2]$, we have $\VV^\perp = (\VV[2])^\perp = {}^\perp\VV$ as subcategories of $\CC$.  Similarly, we have $(\VV^\perp)^\perp = {}^{\perp}(\VV^\perp)$.

It follows from Lemma \ref{lemma:WhenCoproductTriangulated} that $\CC = \VV^\perp \oplus {}^\perp(\VV^\perp)$.  We have $\VV \subseteq {}^{\perp}(\VV^\perp)$ (obvious) and $\UU \subseteq \VV^\perp$ (by Lemma \ref{lemma:PerpendicularNotEmpty}), meaning that neither $\VV^\perp$ nor ${}^\perp(\VV^\perp)$ are the zero category.  This shows that $\CC$ is not a block.
\end{proof}

\def\cprime{$'$}
\providecommand{\bysame}{\leavevmode\hbox to3em{\hrulefill}\thinspace}
\providecommand{\MR}{\relax\ifhmode\unskip\space\fi MR }
\providecommand{\MRhref}[2]{%
  \href{http://www.ams.org/mathscinet-getitem?mr=#1}{#2}
}
\providecommand{\href}[2]{#2}

\end{document}